\documentclass[12pt]{amsart}
\usepackage{SSdefn}
\usepackage{eucal}

\DeclareMathOperator{\maxdeg}{maxdeg}

\DeclareMathOperator{\Ext}{Ext}
\DeclareMathOperator{\Cl}{Cl}
\DeclareMathOperator{\Isom}{Isom}

\newcommand{\nfp}{\mathrm{nfp}}

\newcommand{\FI}{\mathbf{FI}}

\newcommand{\univ}{\mathrm{univ}}

\newcommand{\stacks}[1]{\cite[\href{http://stacks.math.columbia.edu/tag/#1}{Tag~#1}]{stacks}}

\DeclareMathOperator{\uExt}{\ul{Ext}}
\newcommand{\uK}{\ul{\rK}}
\newcommand{\uMod}{\ul{\Mod}}
\DeclareMathOperator{\ini}{in}
\DeclareMathOperator{\sd}{sd}
\newcommand{\fpres}{\mathrm{fp}}

\newcommand{\lpp}{(\!(}
\newcommand{\rpp}{)\!)}

\newcommand{\coloneq}{\mathrel{\mathop:}\mkern-1.2mu=}

\title{The module theory of divided power algebras}
\date{\today}

\author{Rohit Nagpal}
\address{Department of Mathematics, University of Chicago, Chicago, IL}
\email{\href{mailto:nagpal@math.uchicago.edu}{nagpal@math.uchicago.edu}}
\urladdr{\url{http://math.uchicago.edu/~nagpal/}}

\author{Andrew Snowden}
\address{Department of Mathematics, University of Michigan, Ann Arbor, MI}
\email{\href{mailto:asnowden@umich.edu}{asnowden@umich.edu}}
\urladdr{\url{http://www-personal.umich.edu/~asnowden/}}

\thanks{AS was supported by NSF grants DMS-1303082 and DMS-1453893.}


\begin{document}

\begin{abstract}
We study modules for the divided power algebra $\bD$ in a single variable over a commutative noetherian ring $\bk$. Our first result states that $\bD$ is a coherent ring. In fact, we show that there is a theory of Gr\"obner bases for finitely generated ideals, and so computations with finitely presented $\bD$-modules are in principle algorithmic. We go on to determine much about the structure of finitely presented $\bD$-modules, such as: existence of certain nice resolutions, computation of the Grothendieck group, results about injective dimension, and how they interact with torsion modules. Our results apply not just to the classical divided power algebra, but to its $q$-variant as well, and even to a much broader class of algebras we introduce called ``generalized divided power algebras.'' On the other hand, we show that the divided power algebra in two variables over $\bZ_p$ is not coherent.
\end{abstract}

\maketitle
\tableofcontents

\section{Introduction}

Let $\bD$ be the divided power algebra in a variable $x$ over the commutative ring $\bk$. Recall that $\bD$ is free as a $\bk$-module with basis $x^{[0]}, x^{[1]}, \ldots$, and multiplication is defined by
\begin{displaymath}
x^{[n]} x^{[m]} = \binom{n+m}{n} x^{[n+m]}.
\end{displaymath}
If $\bQ \subset \bk$ then $\bD$ is isomorphic to the polynomial ring $\bk[x]$ via $x^{[n]} \mapsto \frac{x^n}{n!}$. However, in general $\bD$ is quite different from the polynomial ring: for example, if $\bk=\bF_p$ then $\bD$ is not even noetherian.

This paper is an investigation of the theory of $\bD$-modules. Even though $\bD$ is typically non-noetherian, we show that finitely presented $\bD$-modules are well-behaved, and prove a variety of results concerning them. In the remainder of the introduction, we summarize our results and explain our motivation for studying $\bD$-modules.

\subsection{Summary of results}

Recall that a ring $R$ is {\bf coherent} if every finitely generated ideal is finitely presented. Equivalently, $R$ is coherent if the kernel of any map of finitely presented modules is again finitely presented; this ensures that the category of finitely presented modules is abelian. Our first result about $\bD$ is:

\begin{theorem} \label{mainthm:coh}
If $\bk$ is noetherian then $\bD$ is coherent.
\end{theorem}

In fact, our results are more precise, in two ways. First, we give a complete characterization of the rings $\bk$ for which $\bD$ is graded-coherent (meaning any finitely generated homogeneous ideal is finitely presented): namely, $\bD$ is graded-coherent if and only if $\bk$ is coherent and for any finitely generated ideal $\fa$ of $\bk$, the torsion submodule of $\bk/\fa$ has finite exponent. And second, when $\bk$ is noetherian, we actually prove that $\bD$ is {\bf Gr\"obner-coherent}. This is a notion we introduce in \cite[\S 4]{grobcoh}, which basically means that there is a good theory of Gr\"obner bases for finitely generated ideals in $\bD$. In particular, this means that calculations with finitely presented $\bD$-modules are algorithmic, at least in principle.

We next investigate the structure of finitely presented $\bD$-modules, working under the assumption that $\bk$ is noetherian. We first construct certain nice resolutions. In general, finitely presented $\bD$-modules do not have finite projective dimension. There are two obsturctions. First, if $M$ is a $\bk$-module with infinite projective dimension then $M \otimes_{\bk} \bD$ has infinite projective dimension as a $\bD$-module. And second, if $\bk/p\bk$ is non-zero and $q>1$ is a power of $p$ then
%\rohit{replaces $=$ by $\coloneq$ here to take care of the referee's comment.}
\begin{displaymath}
(\bD/p \bD)^{(q)} \coloneq \bigoplus_{q \mid n} (\bk/p\bk) x^{[n]}
\end{displaymath}
is naturally a $\bD$-module, and does not have finite projective dimension. We therefore introduce a class of modules that includes the above two counterexamples: we say that a $\bD$-module $N$
%\rohit{Earlier we wrote $M$ which seems confusing as we use $M$ later in the sentence for something else.}
is {\bf special} if it has a finite length filtration such that the graded pieces have the form $(M \otimes_{\bk} \bD)^{(q)}$, where $M$ is a finitely generated $\bk$-module, $q$ is a power of a prime $p$, and $pM=0$ if $q \ne 1$. We then prove:

\begin{theorem} \label{intro:sr}
Every finitely presented $\bD$-module $M$ admits a finite resolution
\begin{displaymath}
0 \to P_r \to \cdots \to P_0 \to M \to 0
\end{displaymath}
by special modules. In fact, one can take $P_0, \ldots, P_{r-1}$ to be free and $P_r$ to be special, and one can bound $r$ in terms of the Krull dimension of $\bk$ (if it is finite).
\end{theorem}

As consequence of this theorem, we obtain a useful spanning set for the Grothendieck group $\rK(\bD)$ of finitely presented $\bD$-modules. With more work, we prove:

\begin{theorem} \label{intro:intK}
There is a canonical short exact sequence
\begin{displaymath}
0 \to \rK(\bk) \to \rK(\bD) \to \bigoplus_p \bQ_p/\bZ_p \otimes \rK(\bk/p \bk) \to 0,
\end{displaymath}
where the direct sum is taken over all prime numbers. This sequence splits if and only if for every prime $p$ the natural map $\rK(\bk/p \bk) \to \rK(\bk)$ is zero.
\end{theorem}

We next investigate the injective dimension of $\bD$-modules. We begin by completely characterizing the injective dimension of $\bD$ over itself when $\bk$ is a field. We next show that if $E_0$ is a bounded complex of $\bk$-modules with finitely generated homology and injective amplitude $[a,b]$ then $E=\bD \otimes_{\bk} E_0$ has injective amplitude $[a,b+2]$, and show that $b+2$ can be lowered to $b+1$ or even $b$ in certain cases. Using this, we show that if $\bk$ is noetherian
%\rohit{Added  ``if $\bk$ is noetherian" to take care of referee's comment.}
and  $\omega_{\bk}$ is a dualizing complex for $\bk$ then $\omega_{\bD}=\omega_{\bk} \otimes_{\bk} \bD$ is one for $\bD$. In particular, if $\bk$ admits a dualizing complex then the bounded derived category of finitely presented $\bD$-modules is self-dual.

Next we study the relationship between torsion modules and finitely presented modules. Here, an element $m$ of a $\bD$-module $M$ is said to be {\bf torsion} if $x^{[n]} m=0$ for all $n \gg 0$, and $M$ is said to be {\bf torsion} if all of its elements are. If $\bk$ contains a field of characteristic~0 then $\bD=\bk[x]$, and finitely presented $\bD$-modules can have torsion, e.g., $\bk[x]/(x)$. However, this is essentially the only example:

\begin{theorem}
The following conditions are equivalent:
\begin{enumerate}
\item Every finitely presented $\bD$-module is torsion-free.
\item Every maximal ideal of $\bk$ has positive residue characteristic.
\end{enumerate}
\end{theorem}

When $\bk$ is $p$-adically complete, we can say much more, at least for graded modules:

\begin{theorem} \label{intro:torext}
Suppose $\bk$ is $p$-adically complete. Let $M$ and $T$ be graded $\bD$-modules, with $M$ finitely presented and $T$ torsion. Then $\uExt^i_{\bD}(T, M)=0$ for all $i \ge 0$.
\end{theorem}

Here $\uExt$ denotes a graded version of $\Ext$, see \S \ref{ss:notation}. The above theorem is false without the completeness hypothesis, see Remark~\ref{rmk:ext-torsion}. It is also false in the non-graded case, as Example~ \ref{ex:ext} (with $\bk=\bF_p$) shows.

Finally, motivated by certain applications (see \S \ref{ss:motivation}), we study ``nearly finitely presented'' (nfp) $\bD$-modules. A graded $\bD$-module $M$ is nfp if there exists a finitely presented $\bD$-module $N$ (called a {\bf weak fp-envelope} of $M$) such that $\tau_{\ge n}(M) \cong \tau_{\ge n}(N)$ as $\bD$-modules, where $\tau_{\ge n}$ denotes truncation to degrees $\ge n$ (see \S \ref{s:nfp}). When $\bk$ is complete, nfp modules are very well behaved:

\begin{theorem}
Suppose $\bk$ is $p$-adically complete. Then the kernel, cokernel, and image of a map of nfp modules is again nfp, and so the category of nfp modules is abelian. Furthermore, the weak fp-envelope $N$ of an nfp module $M$ is unique up to canonical isomorphism, and there exists a canonical map $M \to N$ that is universal among maps from $M$ to finitely presented modules.
\end{theorem}

As with the previous theorem, this one is false without the completeness assumption, see Proposition~\ref{prop:counter-example}. When $\bk$ is not complete, we can at least prove uniqueness of the weak fp-envelope in some cases:

\begin{theorem}
Suppose that every maximal ideal of $\bk$ has residue characteristic $p$, for some fixed prime number $p$. Then the weak fp-envelope of an nfp module is unique up to isomorphism.
\end{theorem}

The isomorphism in the above theorem is not canonical, and the weak fp-envelope of $M$ is not functorial in $M$. The above theorem can fail without the hypothesis on $\bk$: if $\bk$ is a number ring with non-trivial class group then weak nfp-envelopes are necessarily non-unique, see Proposition~\ref{prop:class-group}.

\subsection{$q$-divided power algebras and beyond}

Let $q \in \bk$. One then defines
\begin{displaymath}
[n]_q = 1+q+\cdots+q^{n-1} = \frac{q^n-1}{q-1}
\end{displaymath}
and
\begin{displaymath}
[n]_q! = [n]_q [n-1]_q \cdots [1]_q, \qquad {n \brack m}_q = \frac{[n]_q!}{[n-m]_q! [m]_q!}.
\end{displaymath}
The quantity ${n \brack m}_q$ is a polynomial in $q$, and called the {\bf $q$-binomial coefficient} (or {\bf Gaussian binomial coefficient}). One can modify the definition of the divided power algebra
%\rohit{Fixed the mistype}
by replacing the usual binomial coefficients with their $q$- counterparts. The result is called the {\bf $q$-divided power algebra}.
%\rohit{Fixed the mistype}.
All of the results in this paper apply to, and are proved for, $q$-divided power algebras.\footnote{Except Theorem~\ref{intro:intK}, for which we only have partial results.}

In fact, we work even more generally. Let $\pi_2, \pi_3, \ldots$ be a sequence of elements in $\bk$ that is ``admissible'' in the sense that if $\pi_n$ and $\pi_m$ belong to a common maximal ideal then $n \mid m$ or $m \mid n$. We define
\begin{displaymath}
a(n) = a(n; \pi_{\bullet}) = \prod_{d \mid n, d \ne 1} \pi_d
\end{displaymath}
and\footnote{We give a better definition in the body of the paper that does not require division.}
\begin{displaymath}
C(n,m; \pi_{\bullet}) = \frac{a(n) a(n-1) \cdots a(n-m+1)}{a(m) a(m-1) \cdots a(1)}.
\end{displaymath}
The above coefficients are called {\bf $\pi_{\bullet}$-binomial coefficients}. One can modify the definition of the divided power algebra by replacing the usual binomial coefficients with their $\pi_{\bullet}$ counterparts. The resulting algebra is denoted $\bD(\pi_{\bullet}; \bk)$ and called a {\bf generalized divided power algebra} (GDPA). See Proposition~\ref{prop:chargdpa} for an intrinsic characterization of these algebras and Proposition~\ref{prop:gcd-morphic} for an enlightening description of admissible sequences when $\bk$ is a PID.

As one would expect, the main examples of interest fit into this general setup:
\begin{enumerate}
\item If $\pi_n=1$ for all $n$ then $\bD(\pi_{\bullet}; \bk) \cong \bk[x]$.
\item Define $\pi_n$ to be $p$ if $n$ is a power of a prime $p$, and~1 if $n$ is not a prime power. Then $\bD(\pi_{\bullet}; \bk)$ is the usual divided power algebra over $\bk$. (In the context of GDPA's, we refer to this as the {\bf classical divided power algebra}.)
\item Suppose $\pi_n=\Phi_n(q)$ where $\Phi_n$ is the $n$th cyclotomic polynomial. Then $\bD(\pi_{\bullet}; \bk)$ is the $q$-divided power algebra.
\end{enumerate}
In fact, by taking $q=0$ or $q=1$, example (c) reverts to (a) or (b). Here is a slightly different example:
\begin{enumerate}
\setcounter{enumi}{3}
\item Fix $m \ge 0$ and take $\bD$ to be the $\bZ_p$-submodule of $\bQ_p[x]$ spanned by $\frac{x^n}{\lfloor n/p^m \rfloor !}$. One can show that $\bD$ is a GDPA; in fact, it is isomorphic to $\bD(\pi_{\bullet}; \bZ_p)$ where $\pi_n$ is $p$ if $n$ is a power of $p$ that exceeds $p^m$, and~1 otherwise. (The ideal $\bD_+$ carries a partial divided power structure, in the sense of Berthelot. We thank the referee for this comment.) 
\end{enumerate}
All of the results of this paper are proved for arbitrary GDPA's.\footnote{Same caveat regarding Theorem~\ref{intro:intK}.} Working in this generality actually makes the structure of many arguments more clear and is no more difficult than treating the $q$-divided power algebra.

\subsection{Motivation} \label{ss:motivation}

A famous theorem of Nakaoka \cite{nakaoka}
%\rohit{Reference added}
asserts that the cohomology of symmetric groups stabilizes: for any coefficient ring $\bk$, the restriction map
\begin{displaymath}
\rH^t(S_n, \bk) \to \rH^t(S_{n-1}, \bk)
\end{displaymath}
is an isomorphism for $n>2t$. Motivated by applications to the cohomology of configuration spaces, the first author \cite{thesis} generalized Nakaoka's theorem to allow for non-trivial coefficients, as follows. An {\bf FI-module} over $\bk$ is a sequence $(M_n)_{n \ge 0}$, where $M_n$ is a $\bk[S_n]$-module, equipped with certain transition maps $M_n \to M_{n+1}$. The main theorem of \cite{thesis} states that if $\bk$ is a field of characteristic $p$ and $M$ is a finitely generated $\FI$-module over $\bk$ then, for fixed $t$, the dimension of $\rH^t(S_n, M_n)$ is periodic in $n$ with period a power of $p$.

In the forthcoming paper \cite{periodicity}, we generalize this periodicity theorem and greatly simplify its proof. For an $\FI$-module $M$, define
\begin{displaymath}
\Gamma^t(M) = \bigoplus_{n \ge 0} \rH^t(S_n, M_n).
\end{displaymath}
Then $\Gamma^t(M)$ naturally has the structure of a graded $\bD$-module, where $\bD$ is the univariate classical divided power algebra over $\bk$. The main theorem of \cite{periodicity} states that when $\bk$ is noetherian and $M$ is finitely generated, the $\bD$-module $\Gamma^t(M)$ is nearly finitely presented. When $\bk$ is a field, this immediately implies the periodicity result of \cite{thesis}.

\subsection{Multivariate divided power algebras}

It would be natural to attempt to generalize the results of this paper to divided power algebras in several variables. Unfortunately, the basic coherence property fails in multiple variables (see \S \ref{ss:two-variable}), so it is not clear to us what one could hope to say. (We note, however, that coherence still holds when $\bk$ is a field. Indeed, if $\bk$ has characteristic~0 then the $r$-variable divided power algebra is isomorphic to the $r$-variable polynomial ring. If $\bk$ has characteristic~$p$ then the $r$-variable divided power algebra is isomorphic to
\begin{displaymath}
\bk[x_{i,j}]_{i \in \bN, 1 \le j \le r}/(x_{i,j}^p),
\end{displaymath}
is thus coherent; in fact, it is isomorphic, as an ungraded $\bk$-algebra, to the single variable divided power algebra.) 

\subsection{Notation and terminology} \label{ss:notation}

By a graded ring or module, we will always mean one graded by the integers. Our rings are almost always concentrated in non-negative degrees. Suppose $V$ is a graded module. For integers $q$ and $k$, we let $V^{(q;k)}$ be the direct sum of the graded pieces $V_n$ with $n \equiv k \pmod{q}$, and we omit the $k$ when $k=0$. If $R$ is a graded ring then $R^{(q)}$ is again a graded ring, and each $R^{(q;k)}$ is an $R^{(q)}$-module. By the {\bf regrade} of $V^{(q)}$ we mean the module that is $V_{qn}$ in degree $n$. We write $V_{<n}$ for the sum of the $V_k$ with $0 \le k < n$.

Let $A$ be a graded ring. We write $\uMod_A$ for the category of graded $A$-modules and $\Mod_A$ for the category of non-graded modules. Let $M$ and $N$ be graded $A$-modules. We write $\uHom_A(M,N)_n$ for the set of all modules maps $M[n] \to N$ where the grading is respected, and we put
\begin{displaymath}
\uHom_A(M,N) = \bigoplus \uHom_A(M,N)_n.
\end{displaymath}
This is a graded $A$-module. We write $\Hom_A(M,N)$ for the set of all module maps $M \to N$, ignoring the grading. There is a natural map $\uHom_A(M,N) \to \Hom_A(M,N)$ which is an isomorphism if $M$ is finitely generated. We write $\uExt$ for the derived functor of $\uHom$.

\subsection{Outline}

In \S \ref{s:adm}, we introduce and study $\pi$-sequences, and in \S \ref{s:gdpa} we use them to define GDPA's. In \S \ref{s:coh}, we prove coherence results for GDPA's, including a much more general version of Theorem~\ref{mainthm:coh}. In \S\S \ref{s:sr}, \ref{s:groth}, \ref{s:inj}, \ref{s:tor}, \ref{s:nfp}, we study finitely presented and nearly finitely presented $\bD$-modules. Finally, \S \ref{s:prob} lists some open problems.

\subsection*{Acknowledgments}

The second author thanks Bhargav Bhatt for helpful conversations.

\section{$\pi$-sequences} \label{s:adm}

\subsection{Divisible sequences} \label{subsec:divisible-sequence}

A {\bf divisible sequence} is a sequence of positive integers $b_0=1, b_1, \ldots$, either infinite or finite in length, such that $b_i$ properly divides $b_{i+1}$, whenever $b_{i+1}$ is defined. Every integer $n \ge 0$ can be written uniquely in the form
\begin{displaymath}
n = \sum_{i \ge 0} d_i b_i
\end{displaymath}
where $0 \le d_i < b_{i+1}/b_i$ (we assume that $b_i= \infty$ if $i$ is the least integer such that $b_i$ is not defined). We call this the {\bf base $b_{\bullet}$ representation} of $n$. Carries in base $b_{\bullet}$ addition will play a prominent role in this paper. For an integer $k \ge 1$, put
\begin{displaymath}
\epsilon_k(n,m) = \lfloor \tfrac{n+m}{k} \rfloor - \lfloor \tfrac{n}{k} \rfloor - \lfloor \tfrac{m}{k} \rfloor.
\end{displaymath}
One easily sees that $\epsilon_k(n,m)$ is either~0 or~1: in fact, $\epsilon_k(n,m)$ can be interpreted as the carry produced in the one's place when adding $n$ and $m$ in base $k$. Thus $\epsilon_{b_i}(n,m)$ is the carry in the $i$th place when adding $n$ and $m$ in base $b_{\bullet}$.

\subsection{$\pi$-sequences}

A {\bf $\pi$-sequence} in a ring $\bk$ is simply a sequence $\pi_{\bullet}=\{\pi_n\}_{n \ge 2}$ of elements of $\bk$. By convention, we always put $\pi_1=0$. Let $\Pi$ be the ring $\bZ[\pi_n^{\univ}]_{n \ge 2}$, where the elements $\pi_n^{\univ}$ are indeterminates. Then the $\pi$-sequence $\pi^{\univ}_{\bullet}$ in $\Pi$ is universal: if $\pi_{\bullet}$ is any $\pi$-sequence in a ring $\bk$ then there is a unique ring homomorphism $\varphi \colon \Pi \to \bk$ such that $\pi_{\bullet}=\varphi(\pi^{\univ})$.

A {\bf $\pi$-invariant} is a rule $c$ assigning to every $\pi$-sequence $\pi_{\bullet}$ in any ring $\bk$ an element $c(\pi_{\bullet})$ of $\bk$ such that if $\varphi \colon \bk \to \bk'$ is a ring homomorphism then $c(\varphi(\pi_{\bullet}))=\varphi(c(\pi_{\bullet}))$. We have $c(\pi^{\univ}_{\bullet})=F_c(\pi_{\bullet}^{\univ})$ for some polynomial $F_c \in \Pi$, and thus $c(\pi_{\bullet})=F_c(\pi_{\bullet})$ for every $\pi$-sequence $\pi_{\bullet}$ in any ring $\bk$. Conversely, if $F \in \Pi$ then $c(\pi_{\bullet})=F(\pi_{\bullet})$ defines a $\pi$-invariant.

Suppose $c$ and $d$ are $\pi$-invariants such that $F_d$ divides $F_c$ in $\Pi$. We then define $\{c/d\}$ to be the $\pi$-invariant associated to $F_c/F_d$. By abuse of notation, we typically write $\{c(\pi_{\bullet})/d(\pi_{\bullet})\}$ for the value of $\{c/d\}$ on $\pi_{\bullet}$. If $d(\pi_{\bullet}) \in \bk$ is a non-zero divisor then $\{c(\pi_{\bullet})/d(\pi_{\bullet})\}$ is equal to $c(\pi_{\bullet})/d(\pi_{\bullet})$, where the division is performed in $\bk$. However, if $d(\pi_{\bullet})$ is a zero-divisor then one cannot determine $\{c(\pi_{\bullet})/d(\pi_{\bullet})\}$ from $c(\pi_{\bullet})$ and $d(\pi_{\bullet})$.

\subsection{Generalized binomial coefficients}

For $n \ge 1$ define $\pi$-invariants $a(n)$ and $A(n)$ by
\begin{displaymath}
a(n) = a(n; \pi_{\bullet}) = \prod_{d \mid n, d \ne 1} \pi_d
\end{displaymath}
and
\begin{displaymath}
A(n)=A(n;\pi_{\bullet}) = a(n) a(n-1) \cdots a(2) a(1).
\end{displaymath}
We remark that M\"obius inversion gives the identity
\begin{displaymath}
\pi_n = \left\{ \prod_{d \mid n} a(n/d)^{\mu(d)} \right\},
\end{displaymath}
and so $\pi_n$ can be recovered from the $a(n;\pi_{\bullet})$'s provided they are not zero-divisors. For $0 \le m \le n$, we define a $\pi$-invariant $C(n,m)$ by
\begin{displaymath}
C(n,m)=C(n,m; \pi_{\bullet}) = \prod_{k \ge 2} \pi_k^{\epsilon_k(n-m,m)}.
\end{displaymath}
We think of $a(n)$ as the $\pi_{\bullet}$-analog of $n$, $A(n)$ as the analog of $n!$, and $C(n,m)$ as the analog of the binomial coefficient $\binom{n}{m}$. This last point is justified by the following observation:

\begin{proposition}
We have
\begin{displaymath}
C(n,m) = \left\{ \frac{A(n)}{A(n-m) A(m)} \right\}
\end{displaymath}
\end{proposition}

\begin{proof}
The power of $\pi_k$ in $a(n)$ is 1 if $n$ is a multiple of $k$, and 0 otherwise. Thus the power of $\pi_k$ in $A(n)$ is the number of multiples of $k$ between 1 and $n$, namely $\lfloor \tfrac{n}{k} \rfloor$. Therefore, the power of $\pi_k$ in the above fraction is $\epsilon_k(n-m, m)$, and this proves the identity.
\end{proof}

The following property of $C$ will be important later:

\begin{proposition} \label{prop:binom}
For $\ell \le m \le n$ we have
\begin{displaymath}
C(n,m) C(m,\ell) = C(n-\ell,m-\ell) C(n,\ell).
\end{displaymath}
\end{proposition}

\begin{proof}
Since $\epsilon_k$ is a coboundary, it satisfies the cocycle identity
\begin{displaymath}
\epsilon_k(n,m)+\epsilon_k(n+m,\ell) = \epsilon_k(n,m+\ell)+\epsilon_k(m,\ell).
\end{displaymath}
The proposition follows directly from this.
\end{proof}

\subsection{The $h$-transform}

Let $h \ge 1$ be an integer. We define the {\bf $h$-transform} of the $\pi$-sequence $\pi_{\bullet}$ to be the $\pi$-sequence $\pi^{[h]}_{\bullet}$ given by
\begin{displaymath}
\pi^{[h]}_n=\prod_{d \mid h, (h/d, n)=1} \pi_{dn}
\end{displaymath}
We write $a^{[h]}(n)$ in places of $a(n; \pi^{[h]}_{\bullet})$, and similarly for $A^{[h]}$, $C^{[h]}$, etc. The motivation for this construction is not immediately clear, but Example~\ref{ex:q-dgpa} should provide some intuition.

\begin{proposition}
We have $a^{[h]}(n)=\{ a(hn)/a(h) \}$.
\end{proposition}

\begin{proof}
It suffices to prove $a^{[h]}(n)=a(hn)/a(h)$ for the universal $\pi$-sequence. We have
\begin{displaymath}
a(n; \pi^{[h]}_{\bullet}) = \prod_{d \mid n, d \ne 1} \pi_d^{[h]} = \prod_{\substack{d \mid n, d' \mid h \\ d \ne 1, (h/d', d)=1}} \pi_{dd'}.\end{displaymath}
Multiplying both sides by $a(h)$, we obtain
\begin{displaymath}
a(h) a(n; \pi^{[h]}_{\bullet}) = \prod_{\substack{d \mid n, d' \mid h \\ dd' \ne 1, (h/d', d)=1}} \pi_{dd'} = \prod_{d \mid nh, d \ne 1} \pi_d = a(hn),
\end{displaymath}
where in the second step we use the following lemma. This proves the result.
\end{proof}

\begin{lemma}
Let $n$ and $h$ be positive integers. Then every divisor of $nh$ can be written uniquely as $dd'$, where $d$ divides $n$, $d'$ divides $h$, and $(h/d', d)=1$.
\end{lemma}

\begin{proof}
Let $m$ be a divisor of $nh$. Then $d'=(m,h)$ is a divisor of $h$ and $d=m/(m,h)$ is a divisor of $n$ such that $dd'=m$ and $(h/d',d)=(h/(m,h), m/(m,h))=1$. To show uniqueness, suppose that $m=dd'$ with $(h/d',d)=1$. Then $d=m/d'$, so $(h/d', m/d')=1$, which shows that $d'=(h,m)$, and so $d=m/(h,m)$.
\end{proof}

\begin{proposition} \label{dbltrans}
We have $(\pi^{[h]}_{\bullet})_{\phantom \bullet}^{[h']}=\pi^{[hh']}_{\bullet}$.
\end{proposition}

\begin{proof}
Again, it suffices to work in the universal case. We have
\begin{displaymath}
a(n; (\pi^{[h]}_{\bullet})_{\phantom \bullet}^{[h']})
= \frac{a(h'n; \pi^{[h]}_{\bullet})}{a(h'; \pi^{[h]}_{\bullet})}
= \frac{a(hh'n; \pi_{\bullet})/a(h; \pi_{\bullet})}{a(hh'; \pi_{\bullet})/a(h; \pi_{\bullet})}
=\frac{a(hh'n; \pi_{\bullet})}{a(hh'; \pi_{\bullet})}.
\end{displaymath}
Thus the two sequences $(\pi^{[h]}_{\bullet})_{\phantom \bullet}^{[h']}$ and $\pi^{[hh']}_{\bullet}$ define the same $a$'s, and are therefore equal by M\"obius inversion (which is allowed here since in the universal case the $\pi_n$'s are not zero-divisors).
\end{proof}

\subsection{Admissibility}

We say that the $\pi$-sequence $\pi_{\bullet}$ is {\bf admissible} if the following condition holds: if $(\pi_n, \pi_m)$ is not the unit ideal of $\bk$ then either $n \mid m$ or $m \mid n$. Equivalently, $\pi_{\bullet}$ is admissible if whenever $\pi_n$ and $\pi_m$ belong to some maximal ideal $\fm$ either $n \mid m$ or $m \mid n$. Suppose $\pi_{\bullet}$ is admissible, and let $\fa$ be a proper ideal of $\bk$. Define a sequence $b_{\fa,\bullet}$ inductively as follows. First, $b_{\fa,0}=1$. Having defined $b_{\fa,i}$, let $b_{\fa,i+1}$ be the smallest integer $n$ greater than $b_{\fa,i}$ such that $\pi_n \in \fa$; if no such integer exists, $b_{\fa,i+1}$ is undefined and the sequence has finite length. The sequence $b_{\fa,\bullet}$ is  divisible (as defined in \S \ref{subsec:divisible-sequence}).

\begin{proposition} \label{seq:adimg}
Let $\bk \to \bk'$ be a ring homomorphism, let $\pi_{\bullet}$ be an admissible $\pi$-sequence in $\bk$, and let $\pi'_{\bullet}$ be its image in $\bk'$. Then $\pi'_{\bullet}$ is also admissible.
\end{proposition}

\begin{proof}
If $(\pi'_n, \pi'_m)$ is not the unit ideal of $\bk'$ then certainly $(\pi_n, \pi_m)$ is not the unit ideal of $\bk$, and so $n \mid m$ or $m \mid n$.
\end{proof}

\begin{proposition}
If $\pi_{\bullet}$ is admissible so is $\pi^{[h]}_{\bullet}$, for any $h \ge 1$.
\end{proposition}

\begin{proof}
Let $\fm$ be a maximal ideal of $\bk$, and suppose $\pi^{[h]}_n$ and $\pi^{[h]}_m$ belong to $\fm$. Then $\pi_{dn} \in \fm$ for some $d \mid h$ with $(h/d, n)=1$, and similarly $\pi_{d'm} \in \fm$ for some $d' \mid h$ with $(h/d', m)=1$. Since $\pi_{\bullet}$ is admissible, either $dn \mid d'm$ or $d'm \mid dn$; without loss of generality, assume the former. We claim $n \mid m$, which will complete the proof. It suffices to show $v_p(n) \le v_p(m)$ for all primes $p$, where $v_p$ is the $p$-adic valution. Thus let $p$ be given. If $v_p(n)=0$ there is nothing to prove, so suppose $v_p(n)>0$. The divisibility $dn \mid d'm$ translates to the inequality
\begin{displaymath}
v_p(d)+v_p(n) \le v_p(d')+v_p(m).
\end{displaymath}
If $v_p(h)=0$ then $v_p(d')=0$, and so $v_p(n) \le v_p(m)$. Now suppose $v_p(h)>0$. Then the condition $(h/d, n)=1$ implies $v_p(d)=v_p(h)$. Thus
\begin{displaymath}
v_p(h)+v_p(n) \le v_p(d')+v_p(m) \le v_p(h)+v_p(m),
\end{displaymath}
and so $v_p(n) \le v_p(m)$. This completes the proof.
\end{proof}

\begin{proposition} \label{gdpa:elt}
Suppose $\pi_{\bullet}$ is admissible and $\pi_h$ belongs to the Jacobson radical of $\bk$. Then:
\begin{enumerate}
\item If $\pi_n$ is a non-unit of $\bk$ then either $h$ divides $n$ or $n$ divides $h$.
\item For $n>1$, we have $\pi^{[h]}_n=\pi_{hn}$ up to units.
\end{enumerate}
\end{proposition}

\begin{proof}
(a) Since $\pi_n$ is a non-unit, it belongs to some maximal ideal $\fm$. Since $\pi_h$ belongs to the Jacobson radical, it also belongs to $\fm$. Thus $n \mid h$ or $h \mid n$ by admissibility.

(b) We have
\begin{displaymath}
\pi^{[h]}_n = \prod_{d \mid h, (h/d,n)=1} \pi_{dn}.
\end{displaymath}
This has $\pi_{hn}$ as a factor, and we claim that all other factors are units. Suppose $\pi_{dn}$ is a non-unit appearing in the product. Then, by (a), $dn \mid h$ or $h \mid dn$. The condition $(h/d,n)=1$ is equivalent to $(h,nd)=d$. Thus if $dn \mid h$ then $d=nd$ and $n=1$, which is not the case, while if $h \mid dn$ then $d=h$, as claimed.
\end{proof}

\subsection{$\pi_{\bullet}$-torsion}

Let $M$ be a $\bk$-module, and let $m \in M$. We say that $m$ is {\bf $\pi_{\bullet}$-torsion} if it is annihilated by $a(n)$ for some $n \ge 1$. Equivalently, $m$ is $\pi_{\bullet}$-torsion if it is annihilated by a product of the form $\prod_{n \in S} \pi_n$, where $S$ is a finite subset of $\bZ_{>1}$. We write $T(M)=T(M; \pi_{\bullet})$ for the set of $\pi_{\bullet}$-torsion elements. It is a submodule of $M$. We say that $T(M)$ is {\bf bounded} if it is annihilated by $a(n)$ for some $n \ge 1$.

\begin{proposition}
If $M$ is noetherian then $T(M)$ is bounded.
\end{proposition}

\begin{proof}
Let $M_n \subset M$ be the submodule killed by $a(n!)$. Then $T(M)$ is the ascending union of the $M_n$, and so, by noetherianity, $M=M_n$ for some $n$.
\end{proof}

\subsection{Admissible sequences in a PID}

Suppose now that $\bk$ is a PID (or just a B\'ezout domain). A sequence $\{a(n)\}_{n \ge 1}$ in $\bk$ is {\bf GCD-morphic} if it satisfies the identity
\begin{displaymath}
a(\gcd(n,m))=\gcd(a(n),a(m))
\end{displaymath}
for all $n,m \ge 1$. The following proposition seems to be well-known (e.g., see \cite{nowicki} or \cite{gcd-morphic}), but we include a proof to be self-contained.

\begin{proposition} \label{prop:gcd-morphic}
Giving a never-zero admissible $\pi$-sequence $\pi_{\bullet}$ is the same as giving a never-zero GCD-morphic sequence $\{a(n)\}_{n \ge 1}$; precisely, $\pi_{\bullet} \mapsto \{a(n; \pi_{\bullet})\}_{n \ge 1}$ provides a bijection between these two classes of sequences.
\end{proposition}

\begin{lemma}
Let $\{a(n)\}_{n \ge 1}$ be a never-zero GCD-morphic sequence. Let $r \ge 1$ be minimal such that $a(r)$ is not a unit. Define
\begin{displaymath}
b(n) = \begin{cases}
a(n)/a(r) & \text{if $r \mid n$} \\
a(n) & \text{if $r \nmid n$} \end{cases}
\end{displaymath}
Then $\{b(n)\}_{n \ge 1}$ is GCD-morphic.
\end{lemma}

\begin{proof}
We must show $\gcd(b(n),b(m))=b(\gcd(n,m))$. This is immediate if $r$ divides both of $n$ and $m$, and also if $r$ divides neither $n$ nor $m$. Thus assume $r \mid n$ but $r \nmid m$. Then $\gcd(a(r),a(m))=a(\gcd(r,m))=1$ since $\gcd(r,m)<r$ and $r$ is minimal such that $a(r)$ is not a unit. Thus $a(m)$ is coprime to $a(r)$, and so
\begin{displaymath}
\gcd(b(n), b(m))=\gcd(a(n)/a(r), a(m))=\gcd(a(n),a(m))=a(\gcd(n,m))=b(\gcd(n,m)).
\end{displaymath}
This completes the proof.
\end{proof}

\begin{lemma}
Let $\{a(n)\}_{n \ge 1}$ be a never-zero GCD-morphic sequence. Then $a(n)=a(n;\pi_{\bullet})$ for a unique never-zero $\pi$-sequence $\pi_{\bullet}$.
\end{lemma}

\begin{proof}
By M\"obius inversion, $\pi_n=\prod_{d \mid n} a(n/d)^{\mu(d)}$ is the unique sequence $\pi_{\bullet}$, if it exists: we must show that this product belongs to $\bk$. Let $\pi_{\bullet}$ be a maximal never-zero sequence such that $\{a(n)/a(n;\pi_{\bullet})\}_{n \ge 1}$ is a GCD-morphic sequence in $\bk$. Here by ``maximal'' we mean that if $\pi'_{\bullet}$ is another other sequence with this property and $\pi_n \mid \pi'_n$ for all $n \ge 2$ then $\pi_n=\pi'_n$ up to units. It is easy to see that a maximal sequence exists: start by maximizing $\pi_2$, then go to $\pi_3$, and so on. Put $a'(n)=a(n)/a(n;\pi_{\bullet})$. We claim that $a'(n)$ is a unit for all $n$. Indeed, suppose not and let $r$ be minimal such that $a'(r)$ is not a unit. Define $\pi'_{\bullet}$ by $\pi'_n=\pi_n$ for $n \ne r$ and $\pi'_r=a(r)$. Then, by the previous lemma, $\{a(n)/a(n;\pi'_{\bullet})\}_{n \ge 1}$ is GCD-morphic. This contradicts the maximality of $\pi_{\bullet}$, and thus proves the claim. We thus see that $a(n)=a(n;\pi_{\bullet})$ up to units for all $n$, and so $\prod_{d \mid n} a(n/d)^{\mu(d)}$ is equal to $\pi_n$ up to units, and thus belongs to $\bk$.
\end{proof}

\begin{lemma}
Let $\pi_{\bullet}$ be a never-zero $\pi$-sequence in $\bk$ and put $a(n)=a(n;\pi_{\bullet})$. Then $\pi_{\bullet}$ is admissible if and only if $\{a(n)\}_{n \ge 1}$ is GCD-morphic.
\end{lemma}

\begin{proof}
Let $\ell=\gcd(n,m)$. Working from the definition of $a(n;\pi_{\bullet})$, we find
\begin{displaymath}
\gcd(a(n),a(m))=a(\ell) \gcd \left( \prod_{d \mid n, d \nmid \ell} \pi_d, \prod_{e \mid m, e \nmid \ell} \pi_e \right).
\end{displaymath}
Suppose $\pi_{\bullet}$ is admissible. If $d$ and $e$ are indices in the products, then $d \nmid e$ and $e \nmid d$, and so $(\pi_d, \pi_e)=1$ by admissibility. It follows that the gcd on the right is~1, and so $\{a(n)\}$ is GCD-morphic. Conversely, suppose $\{a(n)\}$ is GCD-morphic. Then the gcd on the right side above is~1, and so $(\pi_d, \pi_e)=1$ for all indices $d$ and $e$ in products. Suppose now $d$ and $e$ are given such that $d \nmid e$ and $e \nmid d$. Taking $n=d$ and $m=e$, we find $(\pi_e, \pi_d)=1$, and so $\pi_{\bullet}$ is admissible.
\end{proof}

\subsection{Examples}

We now give some examples to illustrate the definitions in this section.

\begin{example}
\label{ex:classical-dgpa}
Take $\bk=\bZ$ and define $\pi_k=p$ if $k$ is a power of a prime $p$ and $\pi_k=1$ otherwise. Then $a(n)=n$ and $A(n)=n!$ and $C(n,m)=\binom{n}{m}$, the usual binomial coefficient. We have $a^{[h]}(n)=a(hn)/a(h)=n$. Thus $A^{[h]}=A$ and $C^{[h]}=C$ as well. It is clear that the sequence $\pi_{\bullet}$ is admissible.
\end{example}

\begin{example}
\label{ex:fibonomial-dgpa}
Take $\bk=\bZ$ and define $a(n) = F_n$ where $F_n$ is the $n$th Fibonacci number. It is well-known that this sequence is GCD-morphic, and so, by Proposition~\ref{prop:gcd-morphic}, there exists a unique admissible sequence $\pi_{\bullet}$ such that $a(n)=a(n;\pi_{\bullet})$. More information on the sequence $\pi_{\bullet}$ can be found at \cite[\href{https://oeis.org/A061446}{A061446}]{sloane}. The coefficients $C(n,m)$ are the so-called ``Fibonomial coefficients.''
\end{example}

\begin{example} \label{ex:q-dgpa}
Take $\bk=\bZ[q]$ and define $\pi_k=\Phi_k(q)$ to be the $k$th cyclotomic polynomial. This is admissible by the following lemma. We have
\begin{displaymath}
a(n) = [n]_q = \frac{q^n-1}{q-1}
\end{displaymath}
and $A(n)=[n]_q!$ and $C(n,m)={n \brack m}_q$, the $q$-binomial coefficient. We have $a^{[h]}(n)=[hn]_q/[h]_q=[n]_{q^h}$ and $A^{[h]}(n)=[n]_{q^h}!$. Thus $C^{[h]}(n,m)={n \brack m}_{q^h}$ is the $q^h$-binomial coefficient.
\end{example}

\begin{lemma}
The $\pi$-sequence in $\bZ[q]$ given by $\pi_n=\Phi_n(q)$ is admissible.
\end{lemma}

\begin{proof}
Let $\fm$ be a maximal ideal of $\bZ[q]$ containing $\Phi_n$ and $\Phi_m$. Let $\kappa=\bZ[q]/\fm$ be the quotient field, which is necessarily finite, say of characteristic $p$, and let $\zeta$ be the image of $q$ in $\kappa$. Write $n=p^s n_0$ and $m=p^r m_0$ with $n_0$ and $m_0$ prime to $p$. Then $\Phi_n(q)=\Phi_{n_0}(q)^{p^s-p^{s-1}}$ modulo $p$, and similarly for $\Phi_m$. We thus see that $\Phi_{n_0}(\zeta)=\Phi_{m_0}(\zeta)=0$, and so it follows that $\zeta$ is both a primitive $n_0$ and $m_0$ root of unity. We conclude that $n_0=m_0$, and so $n \mid m$ or $m \mid n$ according to whether $s \le r$ or $r \le s$.
\end{proof}

\section{Generalized divided power algebras} \label{s:gdpa}

\subsection{The algebra associated to divisible sequence} \label{ss:Salg}

Given a divisible sequence $b_{\bullet}$ and a ring $\bk$, we define a ring $\bS=\bS(\bk, b_{\bullet})$ as follows. First suppose that $b_{\bullet}$ has infinite length. Then $\bS=\bk[y_0, y_1, \ldots]/(y_i^{b_{i+1}/b_i})$. Now suppose $b=(b_0, \ldots, b_r)$. Then $\bS=\bk[y_0, \ldots, y_r]/(y_i^{b_{i+1}/b_i})$, where the relations are imposed for $0 \le i < r$. In particular, $y_r$ is not nilpotent. We give $y_i$ degree $b_i$ which makes $\bS$ a graded $\bk$-algebra.

Let $n \ge 0$ be an integer and let $n=\sum n_i b_i$ be its base $b_{\bullet}$ expansion. We define $x^{[n]} \in \bS$ to be the element $\prod y_i^{n_i}$. It is clearly non-zero of degree~$n$, and the only element of degree~$n$ up to scalar multiples. Let $m=\sum m_i b_i$ be a second integer. Then $x^{[n]} x^{[m]}$ is non-zero if and only if $n_i+m_i<b_{i+1}/b_i$ for all $i$, that is, there is no base carry when computing $n+m$ in base $b_{\bullet}$. Assuming there is no carry, $n+m=\sum (n_i+m_i) b_i$ is the base $b_{\bullet}$ expansion of $n+m$, and so $x^{[n]} x^{[m]} = x^{[n+m]}$.

The following proposition summarizes the above discussion:

\begin{proposition}
\label{prop:GDPA-modulo-maximal-ideal}
As a $\bk$-module, we have $\bS = \bigoplus_{n \ge 0} \bk x^{[n]}$. In this basis, multiplication is given by
\begin{displaymath}
x^{[n]}x^{[m]} = \begin{cases}
x^{[n+m]} & \text{if there is no base $b_{\bullet}$ carry in $n+m$} \\
0 \qquad &\text{if there is a base $b_{\bullet}$ carry in $n+m$}
\end{cases}
\end{displaymath}
\end{proposition}

\subsection{Generalized divided power algebras}

Let $\bk$ be a ring and let $\pi_{\bullet}$ be a $\pi$-sequence in $\bk$. We define a graded $\bk$-algebra $\bD=\bD(\bk, \pi_{\bullet})$ as follows. As a graded $\bk$-module, $\bD$ is free with basis $x^{[i]}$ for $i \in \bN$, where $x^{[i]}$ has degree $i$. Multiplication is defined by $x^{[n-m]} x^{[m]} = C(n,m) x^{[n]}$. Proposition~\ref{prop:binom} ensures that multiplication is associative, while commutativity follows from the obvious relation $C(n,m)=C(n,n-m)$. The element $x^{[0]}$ is the unit, and so we write $x^{[0]}=1$. As usual, $\bD^{[h]}$ denotes $\bD(\bk, \pi^{[h]}_{\bullet})$.

\begin{definition}
A {\bf generalized divided power algebra} (GDPA) over $\bk$ is a graded $\bk$-algebra isomorphic to $\bD(\bk, \pi_{\bullet})$ for some admissible sequence $\pi_{\bullet}$.
\end{definition}

\begin{example} \label{ex:gdpa}
Some examples of GDPA's:
\begin{enumerate}
\item If $\pi_n=1$ for each $n>1$ then $\bD = \bk[x]$.
\item If $\pi_{\bullet}$ is as in Example~\ref{ex:classical-dgpa} then $\bD$ is the classical divided power algebra.
\item If $\pi_{\bullet}$ is as in Example~\ref{ex:q-dgpa} then $\bD$ is the $q$-divided power algebra. \qedhere
\end{enumerate}
\end{example}

\begin{proposition} \label{gdpa:bc}
Let $\bk'$ be a $\bk$-algebra, and let $\pi'_k$ be the image of $\pi_k$ in $\bk'$. Then $\bD' = \bD \otimes_{\bk} \bk'$ is isomorphic to $\bD(\bk', \pi'_{\bullet})$. In particular, the base change of a GDPA is still a GDPA.
\end{proposition}

\begin{proof}
It is clear that $\bD'$ is isomorphic to $\bD(\bk', \pi'_{\bullet})$. If $\pi_{\bullet}$ is admissible then so is $\pi'_{\bullet}$, by Proposition~\ref{seq:adimg}, and so $\bD'$ is a GDPA.
\end{proof}

We say that two sequences $\pi_{\bullet}$ and $\pi'_{\bullet}$ in $\bk$ are {\bf associate} if $\pi_n$ is associate to $\pi'_n$ for all $n$ (that is, $\pi_n$ is a unit times $\pi'_n$). In this case, one is admissible if and only if the other is.

\begin{proposition} \label{gdpa:assoc}
Suppose $\pi_{\bullet}$ is an arbitrary $\pi$-sequence and $\pi'_{\bullet}$ is an admissible $\pi$-sequence. Then $\bD=\bD(\bk, \pi_{\bullet})$ is isomorphic to $\bD'=\bD(\bk, \pi'_{\bullet})$ as graded $\bk$-algebras if and only if $\pi_{\bullet}$ and $\pi'_{\bullet}$ are associate.
\end{proposition}

\begin{proof}
First suppose that $\pi_{\bullet}$ and $\pi'_{\bullet}$ are associate, and write $\pi_n=\alpha_n \pi'_n$ with $\alpha_n$ a unit. For $n \ge 0$, define
\begin{displaymath}
\beta_n=\prod_{k \ge 2} \alpha_k^{\lfloor n/k \rfloor},
\end{displaymath}
which is also a unit. Then $C'(n,m)=\beta_{n+m} \beta_n^{-1} \beta_m^{-1} C(n,m)$, where $C$ is computed with $\pi_{\bullet}$ and $C'$ with $\pi'_{\bullet}$. Thus the map $\bD \to \bD'$ taking $x^{[n]}$ to $\beta_n y^{[n]}$ is an isomorphism of graded $\bk$-algebras.

Now suppose that $\bD$ and $\bD'$ are isomorphic. It suffices to check that $\pi_{\bullet}$ and $\pi'_{\bullet}$ are associate locally, so we assume that $\bk$ is local. Let $b_{\bullet}$ be the divisible sequence associated to $\pi'_{\bullet}$. We note that the isomorphism $\bD \cong \bD'$ implies that $C(n,m)$ and $C'(n,m)$ are associate for all $n$ and $m$. Assume that $\pi_k$ and $\pi'_k$ are associate for $k<n$, and let us prove that $\pi_n$ and $\pi'_n$ are associate. We consider two cases:
\begin{itemize}
\item {\it Case~1: $n \ne b_i$ for any $i$.} Write $n=a+b$ with $a,b>1$ such that there is no base $b_{\bullet}$ carry. Then $C'(n,a)$ is a unit, and so $C(n,a)$ is a unit. Since $\pi_n$ divides $C(n,a)$, it follows that $\pi_n$ is a unit, and thus associate to $\pi'_n$.
\item {\it Case~2: $n=b_i$.} In the addition $b_i=(b_i-b_{i-1})+b_{i-1}$ there is only a carry into $b_i$'s place. We thus see that
\begin{displaymath}
\left\{ \frac{C'(b_i,b_{i-1})}{\pi'_{b_i}} \right\}
\end{displaymath}
is a unit, and so, by the inductive hypothesis, so is
\begin{displaymath}
\left\{ \frac{C(b_i,b_{i-1})}{\pi_{b_i}} \right\}.
\end{displaymath}
Thus $\pi'_n$ is associate to $C'(b_i,b_{i-1})$, which is associate to $C(b_i,b_{i-1})$, which is associate to $\pi_n$. \qedhere
\end{itemize}
\end{proof}

\begin{corollary} \label{cor:adm}
If $\bD(\bk, \pi_{\bullet})$ is a GDPA then $\pi_{\bullet}$ is admissible.
%\rohit{Mistype fixed.}
\end{corollary}

\begin{corollary} \label{gdpa:poly}
Suppose that $\pi_n$ is a unit of $\bk$ for all $n \ge 2$. Then the natural map $\bk[x] \to \bD$ sending $x$ to $x^{[1]}$ is an isomorphism of graded $\bk$-algebras.
\end{corollary}

\begin{proof}
This follows from proposition above and Example~\ref{ex:gdpa}(a).
\end{proof}

\begin{proposition} \label{prop:gdpa-field}
The GDPA's over a field $\bk$ are exactly the algebras $\bS(\bk, b_{\bullet})$ from \S \ref{ss:Salg}.
\end{proposition}

\begin{proof}
Let $\pi_{\bullet}$ be an admissible $\pi$-sequence. By Proposition~\ref{gdpa:assoc}, we may as well assume $\pi_n \in \{0,1\}$ for all $n$. By admissibility, there is a divisible sequence $b_{\bullet}$ such that $\pi_n=0$ if and only if $n=b_i$ for some $i$. It now follows directly from the definition that $C(n,m)$ is~0 if there is a base $b_{\bullet}$ carry in computing $(n-m)+m$, and~1 otherwise. Thus $\bD$ is isomorphic to $\bS(b_{\bullet}, \bk)$ by Proposition~\ref{prop:GDPA-modulo-maximal-ideal}.

Now suppose $b_{\bullet}$ is a given divisible sequence. Define $\pi_n=0$ if $n=b_i$ for some $i$ and $\pi_n=1$ otherwise. Then $\bS(b_{\bullet}, \bk)$ is isomorphic to $\bD(\pi_{\bullet}, \bk)$, and so the former is a GDPA.
\end{proof}

\begin{remark}
In fact, over any ring $\bk$ the $\bS(\bk, b_{\bullet})$ are exactly the GDPA's $\bD(\pi_{\bullet}, \bk)$ for which $\pi_n$ is either~0 or a unit for all $n$.
\end{remark}

\subsection{Further properties of GDPA's}

For this section we fix a GDPA $\bD=\bD(\bk, \pi_{\bullet})$ and establish some basic results, mostly concerning what happens when $\pi_h=0$ for some $h$. We denote the $\bk$-subalgebra of $\bD$ generated by $x^{[n]}$ with $h \mid n$ by $\bD^{(h)}$.

\begin{proposition} \label{prop:regrade}
Suppose $\pi_h$ belongs to the Jacobson radical of $\bk$. Then:
\begin{enumerate}
\item For $0 \le k < h$, the module $\bD^{(h;k)}$ is free over $\bD^{(h)}$ and generated by $x^{[k]}$.
\item The natural map $\bD_{<h} \otimes_{\bk} \bD^{(h)} \to \bD$ is an isomorphism of graded $\bD^{(h)}$-modules.
\item The regrade of $\bD^{(h)}$ is isomorphic to $\bD^{[h]}$.
\end{enumerate}
\end{proposition}

\begin{proof}
(a) Suppose $n$ is a multiple of $h$, $0 \le k < h$, and $m$ either divides $h$ or is a multiple of $h$. Then one finds $\epsilon_m(n,k)=0$. Combined with Proposition~\ref{gdpa:elt}(a), this shows that $C(n+k,k)$ is a unit. We thus see that $\bD^{(h;k)}$ is free of rank one over $\bD^{(h)}$, generated by $x^{[k]}$.

(b) Follows from (a).

(c) Let $\pi'_n=\pi_{nh}$, let $C'(n,m)=C(n,m;\pi'_{\bullet})$, and let $S$ be the set of positive integers that do not divide $h$ or are not divisible by $h$. Define
\begin{displaymath}
u_n=\prod_{k \in S} \pi_k^{\lfloor hn/k \rfloor}.
\end{displaymath}
This is a unit by part~(a). We have
\begin{displaymath}
C(h(n+m),hm)=\left( \prod_{k \in S} \pi_k^{\epsilon_k(hn, hm)} \right) \left( \prod_{k \mid h} \pi_k^{\epsilon_k(hn,hm)} \right) \left( \prod_{h \mid k, k>h} \pi_k^{\epsilon_k(hn, hm)} \right).
\end{displaymath}
The first product is equal to $u_{n+m} u_n^{-1} u_m^{-1}$. The second product is~1 since $\epsilon_k(hn,hm)=0$ when $k$ divides $h$. The final product is $C'(n+m,m)$. Thus we have $C(hn+hm,hm)=u_{n+m} u_n^{-1} u_m^{-1} C'(n+m,m)$. Let $\bD'=\bD(\bk,\pi'_{\bullet})$ with basis $y^{[k]}$. Then the map $\bD' \to \bD^{(h)}$ given by $y^{[k]} \mapsto u_k x^{[hk]}$ is an isomorphism of $\bk$-algebras, and respects the grading after regrading $\bD^{(h)}$. By Proposition~\ref{gdpa:elt}(b), $\pi'_{\bullet}$ and $\pi^{[h]}_{\bullet}$ are associate, and so $\bD'$ is isomorphic to $\bD^{[h]}$ (by Proposition~\ref{gdpa:assoc}), which completes the proof.
\end{proof}

\begin{proposition} \label{gdpa:transcoh}
Suppose $\pi_h$ belongs to the Jacobson radical of $\bk$. Then $\bD$ is (graded-, Gr\"obner-, or [no adjective]) coherent if and only if $\bD^{[h]}$ is.
\end{proposition}

\begin{proof} By Proposition~\ref{prop:regrade}(c), $\bD^{[h]}$ is isomorphic to a regrade of $\bD^{(h)}$, and so it is enough to shows that $\bD$ is (graded-, Gr\"obner-, or [no adjective]) coherent if and only if $\bD^{(h)}$ is. By Proposition~\ref{prop:regrade}(a), $\bD$ is homogeneous and free of rank $h$ as a module over $\bD^{(h)}$, and so the result follows from standard (and easily proved) basic facts about coherence.
\end{proof}

\begin{proposition} \label{gdpa:ideal}
Suppose $\pi_h$ belongs to the Jacobson radical of $\bk$, let $I \subset \bD$ be a homogeneous ideal, and let $0 \le k < h$. Then there exists an ideal $J_k$ of $\bD^{(h)}$ such that $I^{(h;k)} \cong J_k[k]$ as $\bD^{(h)}$-modules. Furthermore, if $I$ is generated in degrees $\le d$ then $J_k$ is generated in degrees $\le h \lceil (d-k)/h \rceil$.
\end{proposition}
\begin{proof}
The assertion that $J_k$ exists and is an ideal follows directly from Proposition~\ref{prop:regrade}(a). For the degree bound, we may assume without loss of generality that $I$ is generated in degree $d$. Let $0 \le d' <h$ be such that $d+d' \equiv k \pmod{h}$. It is enough to show that $J_k[k]$ is generated in degree $d + d'$. To see this, let $N = d+d' + xh$ for some $x \in \bZ_{+}$.
%\Acom{Rohit: Earlier we wrote $H$ instead of $h$.}.
The product $x^{[d]} x^{[N-d]}$ differs from $x^{[d]}x^{[d']}x^{[N-d-d']}$ by a factor of $C(d' + xh, d')$. It suffices to show that this factor is a unit. Since $\pi_h$ is in the Jacobson radical of $\bk$, $h$ belongs to $b_{\fm, \bullet}$ for every maximal ideal $\fm$ of $\bk$  and there are no carries in the addition $d' + xh$ in any base containing $h$, and so $C(d' + xh, d')$ is  a unit.
\end{proof}

\begin{proposition} \label{gdpa:subalg}
Suppose that $\pi_n=0$. Then $\bD_{<n}$ is a subalgebra of $\bD$, and the natural map $\bD_{<n} \otimes_{\bk} \bD^{(n)} \to \bD$ is an isomorphism of graded $\bk$-algebras.
\end{proposition}

\begin{proof}
To show that $\bD_{<n}$ is a subalgebra, it suffices to show that if $i,j<n$ then $x^{[i]} x^{[j]} \in \bD_{<n}$. This is clear if $i+j<n$. If $i+j>n$ then $\epsilon_n(i,j)=1$, and so $C(i+j,i)=0$ since $\pi_n$ appears as a factor of it; thus, in this case, $x^{[i]} x^{[j]}=0$, which does indeed belong to $\bD_{<n}$. The map $\bD_{<n} \otimes_{\bk} \bD^{(n)} \to \bD$ is an isomorphism of graded $\bk$-modules by Proposition~\ref{prop:regrade}(b). In the present situation, it is clearly compatible with multiplication.
\end{proof}

\begin{proposition} \label{gdpa:coh}
Suppose $\pi_n=0$ for infinitely many $n$. If $\bk$ is coherent (resp.\ noetherian) then $\bD$ is coherent (resp.\ Gr\"obner-coherent).
\end{proposition}

\begin{proof}
Let $n_1<n_2<\cdots$ be the indices with $\pi_{n_i}=0$; note that $n_i \mid n_{i+1}$ by Proposition~\ref{gdpa:elt}(a). It follows from Proposition~\ref{gdpa:subalg} that the map $\bD_{<n_i} \otimes_{\bk} \bD^{(n_i)}_{<n_{i+1}} \to \bD_{<n_{i+1}}$ is an isomorphism of rings, and so $\bD_{<n_{i+1}}$ is flat (even free) as a $\bD_{<n_i}$-module. Thus if $\bk$ is coherent then so is $\bD_{<n_i}$, being finite free over $\bk$. Since $\bD$ is the union of the $\bD_{<n_i}$, it too is coherent (\cite[Proposition~20]{soublin}).

Now suppose $\bk$ is noetherian. Then each $\bD_{n_i}$ is a noetherian module, and thus Gr\"obner-coherent. As with coherence, Gr\"obner-coherence passes to the limit \cite[Proposition~5.3]{grobcoh}.
\end{proof}

\begin{proposition} \label{prop:h-free}
Suppose $\pi_{\bullet}$ be an admissible sequence with $\pi_n = 0$ for infinitely many $n$, and let $\bD = \bD(\bk, \pi_{\bullet})$ be a GDPA. If $M$ is finitely presented (graded) $\bD$-module then there exists an $h$ such that $\pi_h =0$ and $M \cong \bD^{(h)} \otimes_{\bk} N$ as $\bD^{(h)}$-modules for some (graded) $\bD_{<h}$-module $N$. In particular, if $\bk$ is a field then $M$ is free as a $\bD^{(h)}$-module.
\end{proposition}

\begin{proof}
Let $F_1 \to F_0 \to M \to 0$ be a presentation of $M$ with $F_0, F_1$ free of finite rank over $\bD$. Suppose that, in a suitable basis, the matrix entries of the map $f \colon F_1 \to F_0$ (as in the presentation above) only involve variables $x^{[0]}, x^{[1]}, \ldots, x^{[t]}$. Pick an $h > t$ such that $\pi_h = 0$. By Proposition~\ref{gdpa:subalg}, $\bD_{<h}$ is a subalgebra of $\bD$ and $\bD = \bD_{<h} \otimes_{\bk} \bD^{(h)}$ as graded $\bk$-algebras. Let $\ol{F}_i$ be a free $\bD_{<h}$-module with the same basis as $F_i$, and define $\ol{f} \colon \ol{F}_1 \to  \ol{F}_0$ using the same matrix that defines $f$. Then $f$ is obtained from $\ol{f}$ by applying the exact functor $- \otimes_{\bk} \bD^{(h)}$. Thus $M = \coker(f) = \coker(\ol{f}) \otimes_{\bk} \bD^{(h)}$ as required.
\end{proof}

\subsection{Characterization of GDPAs}

We defined GDPA's by explicit construction. However, they can also be characterized intrinsically:

\begin{proposition} \label{prop:chargdpa}
Let $\bD$ be a graded $\bk$-algebra. Then $\bD$ is a GDPA if and only if (1) each graded piece of $\bD$ is free of rank~1 over $\bk$; and (2) for every maximal ideal $\fm$ of $\bk$, the quotient $\bD/\fm \bD$ is isomorphic to $\bS(\bk/\fm, b_{\bullet})$ for some divisible sequence $b_{\bullet}$.
\end{proposition}

If $\bD$ is a GDPA then (1) holds by definition, while (2) follows from Propositions~\ref{gdpa:bc} and~\ref{prop:gdpa-field}. It thus suffices to prove the converse. Let $x^{[n]}$ be a basis of $\bD^{[n]}$, and define $c(n,m) \in \bk$ by $x^{[n-m]} x^{[m]} = c(n,m) x^{[n+m]}$. We choose $x^{[0]}=1$, so that $c(n,0)=1$ for all $n$. It suffices to prove the following statement:
\begin{itemize}
\item[($\ast)$] Given elements $\pi_n \in \bk$ for $2 \le n < N$ such that $C(n,m)=c(n,m)$ for all $0 \le m \le n < N$, there exists an element $\pi_N \in \bk$ such that $C(N,m)=c(N,m)$ for all $0 \le m \le N$.
\end{itemize}
Indeed, if this were true then we could construct a $\pi$-sequence $\pi_{\bullet}$ for which $C(n,m)=c(n,m)$ for all $n \ge m \ge 0$, and so $\bD(\bk, \pi_{\bullet})$ would be isomorphic to $\bD$. Furthermore, condition (2) would imply (by Corollary~\ref{cor:adm}) that $\pi_{\bullet}$ is admissible in each quotient $\bk/\fm$, and therefore admissible.

\begin{lemma} \label{char1}
The condition $(\ast)$ holds when $\bk$ is local.
\end{lemma}

\begin{proof}
Let $\fm$ be the unique maximal ideal of $\bk$ and let $b_{\bullet}$ be a divisible sequence such that $\bD/\fm \bD \cong \bS(\bk/\fm, b_{\bullet})$. Note that if there is no base $b$ carry in the addition $(n-m)+m$ then $c(n,m)$ is nonzero modulo $\fm$, and thus a unit of $\bk$. Let $N$ and $\pi_n$ be given as in $(\ast)$. Note that if $n<N$ is not of the form $b_i$ then $\pi_n$ is a unit. Indeed, we can write $n=(n-m)+m$ for some $0<m<n$ with no carry, and so $C(n,m)=c(n,m)$ is a unit. As $C(n,m)$ is a multiple of $\pi_n$, it follows that $\pi_n$ is a unit.

First suppose that $N$ is not of the form $b_i$, and let $i$ be maximal such that $b_i$ divides $N$. Note that the $i$th base $b$ digit of $N$ is nonzero, and all digits in lower places are 0. Clearly then, there are no carries in the addition $b_i+(N-b_i)$. It follows that $C(N,b_i)$ has no factor of the form $\pi_{b_j}$, and so $\{C(N,b_i)/\pi_N\}$ is a unit of $\bk$. We define $\pi_N$ by
\begin{displaymath}
\pi_N = \{ C(N,b_i)/\pi_N \}^{-1} \cdot c(N,b_i).
\end{displaymath}
Thus, by construction, $C(N,b_i)=c(N,b_i)$. We now show that $C(N,m)=c(N,m)$ for all $0 \le m \le N$. If $m=0$ or $m=N$ the equality is clear, so assume $0<m<N$. Put $n=N-m$. By assumption, the $i$th base $b$ digit of $N$ is nonzero. Suppose that the $i$th base $b$ digit of $n$ or $m$ is nonzero, say $n$. Then we can write $n=b_i+n'$ without carry. We then have
\begin{displaymath}
x^{[n]} x^{[m]} = c(n,b_i)^{-1} x^{[b_i]} x^{[n']} x^{[m]} = c(n,b_i)^{-1} c(N-b_i, m) c(N, b_i) x^{[N]},
\end{displaymath}
and so
\begin{displaymath}
c(N, m) = c(n,b_i)^{-1} c(N-b_i, m) c(N, b_i).
\end{displaymath}
We get a similar identity for $C$. Since we know the $c$'s and $C$'s on the right agree, this gives $c(N,m)=C(N,m)$. Now suppose that the $i$th base $b$ digits of $n$ and $m$ are each zero. There must then be a carry that produces a nonzero $i$th digit in $n+m$. Write $n=n_1+n_2$ and $m=m_1+m_2$ where $b_i \mid n_1, m_1$ and $n_2,m_2<b_1$, and write $n_2+m_2=b_i+\ell$; all of these decompositions are without carry. We then have
\begin{displaymath}
x^{[n]} x^{[m]} = c(n,n_1)^{-1} c(m,m_1)^{-1} c(n_1+m_1, n_1) c(n_2+m_2, n_2) x^{[n_1+m_1]} x^{[n_2+m_2]}
\end{displaymath}
and
\begin{displaymath}
x^{[n_2+m_2]} = c(n_2+m_2, b_i)^{-1} x^{[\ell]} x^{[b_i]}, \qquad
x^{[n_1+n_2]} x^{[\ell]} = C(N-b_i, \ell) x^{[N-b_i]},
\end{displaymath}
and so
\begin{displaymath}
c(N,m) = c(n,n_1)^{-1} c(m,m_1)^{-1} c(n_1+m_1, n_1) c(n_2+m_2, n_2) c(n_2+m_2, b_i)^{-1} c(N-b_i, \ell) C(N, b_i).
\end{displaymath}
Again, there is a similar identity for $C$, which yields $C(N,m)=c(N,m)$.

Now suppose $N=b_i$ with $i \ge 1$. In the sum $b_{i-1}+(b_i-b_{i-1})$ there is exactly one carry, in the $i$th place. Thus $\{C(N,b_i)/\pi_N\}$ has no factor of $\pi_{b_j}$ with $j<i$, and is therefore a unit. We define
\begin{displaymath}
\pi_N = \{ C(N,b_i)/\pi_N \}^{-1} \cdot c(N,b_i).
\end{displaymath}
Once again, $C(N,b_i)=c(N,b_i)$ by construction. We now show that $C(N,m)=c(N,m)$ for $0 \le m \le N$. Again, we assume $m \ne 0, N$, and put $n=N-m$. Since $n$ and $m$ are less than $b_i$ and they sum to $b_i$, the $(i-1)$st digit of $n$ or $m$ (say $n$) must be nonzero. We can therefore write $n=b_{i-1}+n'$, with no carry. We have
\begin{displaymath}
x^{[n]} x^{[m]} = c(n, b_{i-1})^{-1} x^{[b_{i-1}]} x^{[n']} x^{[m]} = c(n,b_{i-1})^{-1} c(N-b_{i-1}, m) c(N, b_i) x^{[N]}.
\end{displaymath}
We have a similar identity for $C$, which shows $C(N,m)=c(N,m)$, and completes the proof.
\end{proof}

\begin{lemma} \label{char2}
Let $\bk$ be a commutative ring and let $x_1, \ldots, x_n$ and $y_1, \ldots, y_n$ be elements of $\bk$. Suppose that for each maximal ideal $\fm$ of $\bk$ there exists an element $a_{\fm} \in \bk_{\fm}$ such that $a_{\fm} x_i=y_i$ holds in $\bk_m$ for all $i$. Then there exists an element $a \in \bk$ such that $ax_i=y_i$ holds in $\bk$ for all $i$.
\end{lemma}

\begin{proof}
Write $a_{\fm}=s_{\fm}^{-1} b_{\fm}$ with $b_{\fm} \in \bk$ and $s_{\fm} \in \bk \setminus \fm$. Choosing $b_{\fm}$ and $s_{\fm}$ appropriately, we can assume that $b_{\fm} x_i=s_{\fm} y_i$ holds in $\bk$ for all $i$. (Here it is important that there are only finitely many $x_i$ and $y_i$.) The $s_{\fm}$ generate the unit ideal of $\bk$, and therefore finitely many of them do. Let $\Sigma$ be a finite set of maximal ideals such that the $s_{\fm}$ with $\fm \in \Sigma$ generate the unit ideal, and choose an expression $\sum_{\fm \in \Sigma} c_{\fm} s_{\fm} = 1$ with $c_{\fm} \in \bk$. Then putting $a=\sum_{\fm \in \Sigma} c_{\fm} b_{\fm}$, we find $ax_i=y_i$ for all $i$.
\end{proof}

\begin{lemma} \label{char3}
The condition $(\ast)$ holds in general.
\end{lemma}

\begin{proof}
Suppose $N$ as in $(\ast)$ is given. We must find $\pi_N$ such that $C(N,m)=c(N,m)$ for all $0<m<N$ (the $m=0,N$ cases being automatic). That is, we must find $\pi_N$ that simultaneously solves the equations
\begin{displaymath}
\pi_N \cdot \{ C(N,m)/\pi_N \} = c(N, m)
\end{displaymath}
for $0<m<N$. We can find a solution after localizing at each maximal ideal by Lemma~\ref{char1}. We can therefore find a solution by Lemma~\ref{char2}.
\end{proof}

\subsection{A $\Tor$ computation}

In this section we prove the following result:

\begin{proposition} \label{prop:tor1}
Let $\bD$ be a GDPA over $\bk$. Then
\begin{displaymath}
\Tor^{\bD}_1(\bk, \bk) = \bigoplus_{n \ge 1} \bk/(\pi_n)[n].
\end{displaymath}
\end{proposition}

We note that this proposition gives another way to see that the $\pi_n$ are intrinsic to $\bD$, up to units. We recall that $\pi_1=0$ by convention.

\begin{lemma}
Let $\pi_{\bullet}$ be an admissible sequence in $\bk$, and let $n \ge 1$. Then the ideal $\fa$ generated by the element $C(n,i)$ for $1 \le i \le n-1$ is the principal ideal $(\pi_n)$.
\end{lemma}

\begin{proof}
It suffices to prove this when $\bk$ is local. Assume this, and let $b_{\bullet}=b_{\fm,\bullet}$ be the associated divisible sequence. If $n \ne b_i$ for some $i$ then there is an decomposition $n=i+j$ that involves no base $b_{\bullet}$ carries. Thus $C(n,i)$ is a unit, and so $\fa$ is the unit ideal, and so $\fa=(\pi_n)$ since $\pi_n$ is a unit. Now suppose $n=b_r$. Then in every decomposition $n=i+j$ with $1 \le i,j \le n-1$, there is a carry in the $r$th digit, and so $\pi_n$ divides $C(n,i)$ for all such $i$. Let $m=b_{r-1}$. Then in the decomposition $n=(n-m)+m$, there is only a carry in the $r$th place, and so $C(n,m)=\pi_n$ up to units. It follows that $\fa=(\pi_n)$, completing the proof.
\end{proof}

\begin{proof}[Proof of Proposition~\ref{prop:tor1}]
Let $I$ be the ideal generated by the $x^{[n]}$ with $n \ge 1$, so that $\bk=\bD/I$. Then $\Tor_1^{\bD}(\bk,\bk)=I/I^2$. The degree $n$ piece of $I^2$ is spanned, as a $\bk$-module, by all products of the form $x^{[i]} x^{[j]}$ with $i+j=n$ and $i,j \ge 1$; it is thus equal to $\fa x^{[n]}$ where $\fa$ is the ideal generated by the $C(n,i)$ with $1 \le i \le n-1$. By the lemma, $\fa=(\pi_n)$, and so the degree $n$ piece of $I/I^2$ is $\bk/(\pi_n)$.
\end{proof}

\section{Coherence results} \label{s:coh}

\subsection{Overview}

In \S \ref{s:coh} we prove two main results. The first, in \S \ref{ss:grobcoh}, states that GDPA's over noetherian rings are coherent. The second, in \S \ref{ss:grcoh} gives more precise results about graded-coherence of GDPA's over non-noetherian rings. The second result depends on work in \S \ref{ss:torbd} and \S \ref{ss:relbd} where we establish bounds on the degrees of relations for ideals in a GDPA. These bounds could be useful even when $\bk$ is noetherian, although they are not needed to prove any coherence results in that case.  Finally, in \S \ref{ss:two-variable} we show that coherence fails for multivariate divided power algebras.

\subsection{Gr\"obner-coherence} \label{ss:grobcoh}

We recall some definitions from \cite[\S 4]{grobcoh}.  A graded module is {\bf Gr\"obner-coherent} if it is graded-coherent and every finitely generated inhomogeneous submodule admits a finite Gr\"obner-basis; this implies coherent, but is stronger (see \cite[Proposition~4.4]{grobcoh}). A graded ring is {\bf Gr\"obner-coherent} if it is so as a module over itself; this implies that all finitely presented modules are Gr\"obner-coherent. The following is our main theorem on coherence of GDPA's.
%\rohit{Paragraph extended and more precise references added}

\begin{theorem} \label{thm:grobcoh}
A GDPA over a noetherian ring $\bk$ is Gr\"obner-coherent.
\end{theorem}

\begin{corollary}
If $\bk$ is noetherian then the classical divided power algebra
%\rohit{Fixed the mistype}
over $\bk$, as well as its $q$-analogs, are Gr\"obner-coherent (and thus coherent as well).
\end{corollary}

\begin{remark}
In fact, we can conclude Gr\"obner-coherence for a larger class of coefficient rings using some of the basic properties of Gr\"obner-coherence. Two examples:
\begin{enumerate}
\item Suppose that $\bk$ is a direct limit of noetherian rings $\bk_i$ with flat transition maps, and let $\bD$ be a GDPA such that the $\pi_n$ belong to the initial ring $\bk_0$. Then $\bD$ is Gr\"obner-coherent. Indeed, if $\bD_i=\bD(\pi_{\bullet}; \bk_i)$ then $\bD_i$ is Gr\"obner-coherent by the theorem, and $\bD$ is the direct limit of the $\bD_i$ (and the transition maps are still flat), and thus Gr\"obner-coherent by \cite[Proposition~5.3]{grobcoh}. For example, if $\bk$ is a polynomial ring in infinitely many variables over a noetherian ring then the classical divided power algebra over $\bk$ is Gr\"obner-coherent.
\item Suppose that $\bk$ is locally noetherian (meaning $\bk_{\fm}$ is noetherian for all maximal ideals $\fm$) and $\bD$ is a GDPA over $\bk$ that is graded-coherent. Then $\bD$ is Gr\"obner-coherent. Indeed, $\bD_{\fm}$ is Gr\"obner-coherent by the theorem, and so $\bD$ is Gr\"obner-coherent by \cite[Proposition~5.1]{grobcoh}. For example, if $\bk$ is a Boolean ring then the classical divided power algebra
%\rohit{Fixed the mistype}
over $\bk$ is Gr\"obner-coherent. (The graded coherence of $\bD$ follows easily from Corollary~\ref{cor:grcohdp} in this case.) \qedhere
\end{enumerate}
\end{remark}

The heart of the proof is contained in the following lemma:

\begin{lemma}
Let $\bk$ be a domain and let $\bD=\bD(\bk, \pi_{\bullet})$ be a GDPA.
%\rohit{Changed ``admissible GDPA" to ``GDPA" as admissibility  is already included in the definition of a GDPA.}.
Suppose that $\pi_n \ne 0$ for all $n \ge 2$ and that $\bD/\fa \bD$ is Gr\"obner-coherent for all non-zero principal ideals $\fa$ of $\bk$. Then $\bD$ is Gr\"obner-coherent.
\end{lemma}

\begin{proof}
Let $I=(x_1, \ldots, x_n)$ be a finitely generated ideal of $\bD$, and let $\bK = \Frac(\bk)$.
%\rohit{Added the definition of $\bK$}.
Since $\bD \otimes_{\bk} \bK \cong \bK[x]$ (Corollary~\ref{gdpa:poly}), the ideal $I \otimes_{\bk} \bK$ is principal; let $z \in I$ be a generator. Write $x_i=\frac{zy_i}{a}$ with $y_i \in \bD$ and $a \in \bk$. Since the $x_i$ generate the ideal $(z)$ over $\bK$, we have an expression $z=\sum_{i=1}^n \frac{w_i x_i}{b}$ with $w_i \in \bD$ and $b \in \bk$. Expressing the $x_i$ in terms of the $y_i$, we see that the ideal $J=(y_1, \ldots, y_n)$ contains the non-zero element $c=ab$ of $\bk$. Note that multiplication by $\frac{z}{a}$ defines an isomorphism $J \to I$ of $\bD$-modules.

Let $\ol{J}$ be the image of $J$ in $\ol{\bD} = \bD/c \bD$. Since $\ol{\bD}$ is coherent by assumption and $\ol{J}$ is finitely generated, we see that $\ol{J}$ is finitely presented as a $\ol{\bD}$-module, and thus as a $\bD$-module as well. The exact sequence
\begin{displaymath}
0 \to c \bD \to J \to \ol{J} \to 0
\end{displaymath}
shows that $J$ is finitely presented. Since $I$ is isomorphic to $J$ as a $\bD$-module, it too is finitely presented. We have thus shown that every finitely generated ideal of $\bD$ is finitely presented, and so $\bD$ is coherent.

We now do a bit more work to show that $\bD$ is Gr\"obner-coherent. We have shown that $\bD$ is graded-coherent, so it remains to show that the initial ideal $\ini(I)$ is finitely generated whenever $I$ is a finitely generated ideal of $\bD$. Maintain the previous notation. One easily sees that we have a short exact sequence
\begin{displaymath}
0 \to c \bD \to \ini(J) \to \ini(\ol{J}) \to 0
\end{displaymath}
Since $\ol{\bD}$ is Gr\"obner-coherent, it follows that $\ini(\ol{J})$ is finitely generated, and so $\ini(J)$ is as well. From the equality of ideals $(a) I = (z) J$, we obtain $a \ini(I) = \ini(z) \ini(J)$. Since $\ini(z) \ini(J)$ is clearly finitely generated, it follows that $a \ini(I)$ is finitely generated. But $a \ini(I)$ is isomorphic to $\ini(I)$ via multiplication by $a$, and so $\ini(I)$ is finitely generated.
\end{proof}

\begin{proof}[Proof of Theorem~\ref{thm:grobcoh}]
Let $\bk$ be a noetherian ring and let $\bD=\bD(\bk, \pi_{\bullet})$ be a GDPA. We will show $\bD$ is Gr\"obner-coherent. By noetherian induction, we can assume that $\bD/\fa \bD$ is Gr\"obner-coherent for all nonzero ideals $\fa$ of $\bk$.

First suppose that $\bk$ is not a domain, and let $xy=0$ with $x \ne 0$ and $y \ne 0$. We have a short exact sequence
\begin{displaymath}
0 \to x\bD \to \bD \to \bD/x \bD \to 0.
\end{displaymath}
The rightmost ring is Gr\"obner-coherent by hypothesis, and thus Gr\"obner-coherent as a module over $\bD$. The leftmost term is a finitely presented as a $\bD$-module, being isomorphic to $\bD \otimes_{\bk} (x)$, and thus finitely presented as a $\bD/y \bD$-module. It follows that $x\bD$ is Gr\"obner-coherent as a $\bD/y\bD$-module, and thus as a $\bD$-module as well. Thus $\bD$ is an extension of Gr\"obner-coherent modules, and thus Gr\"obner-coherent.

Now suppose that $\bk$ is a domain. If $\pi_n=0$ for infinitely many $n$ then $\bD$ is Gr\"obner-coherent by Proposition~\ref{gdpa:coh}. Thus we assume that only finitely many $\pi_n$ are 0. Let $h$ be maximal so that $\pi_h=0$. By Proposition~\ref{gdpa:transcoh}, it suffices to prove that $\bD^{[h]}$ is Gr\"obner-coherent. Note that $\pi^{[h]}_n \ne 0$ for all $n \ge 2$ (this follows directly from the definition of the $h$-transform). If $\fa$ is a nonzero ideal of $\bk$ then $\bD/\fa \bD$ is Gr\"obner-coherent by assumption, and so $(\bD/\fa \bD)^{[h]}$ is Gr\"obner-coherent by Proposition~\ref{gdpa:transcoh}. Thus $\bD^{[h]}$ is Gr\"obner-coherent by the lemma, which completes the proof.
\end{proof}

\subsection{General bounds on $\Tor$} \label{ss:torbd}

Let $R$ be a graded algebra supported in non-negative degrees and put $\bk=R_0$. We regard $\bk$ as an $R$-module by letting positive degree elements act by~0. For a non-zero graded object $M$, let $\maxdeg(M)$ denote the maximum integer $n$ such that $M_n$ is non-zero, or $\infty$ if no maximum exists. For a graded $R$-module $M$, put $t_i(M; R)= \maxdeg(\Tor^R_i(M, \bk))$.

\begin{proposition}
\label{prop:relatively-projective}
Suppose $M=V \otimes_{\bk} R$ for some graded $\bk$-module $V$. Then $\Tor^{R}_i(M, \bk)=0$ for $i>0$.
\end{proposition}

\begin{proof}
Let $F_{\bullet} \to V$ be a free resolution of $V$ as a $\bk$-module. Then $R \otimes_{\bk} F_{\bullet}$ is a free resolution of $M$ as an $R$-module. We can thus compute $\Tor^R_i(M, \bk)$ by tensoring this complex with $\bk$. But this just recovers the complex $F_{\bullet}$, which is exact in positive degrees.
\end{proof}

\begin{proposition} \label{prop:tor2}
Let $M$ be an $R$-module and let
\begin{displaymath}
0 \to K \to F \to M \to 0
\end{displaymath}
be an exact sequence with $\Tor^R_1(F, \bk)=0$. Then
\begin{displaymath}
t_1(M; R) \le t_0(K; R) \le \max(t_1(M; R), t_0(F; R)).
\end{displaymath}
\end{proposition}

\begin{proof}
The long exact exact sequence in $\Tor$ gives an exact sequence
\begin{displaymath}
0 \to \Tor^R_1(M, \bk) \to \Tor^R_0(K, \bk) \to \Tor^R_0(F, \bk)
\end{displaymath}
(since $\Tor^R_1(F, \bk)=0$), from which the stated inequalities are clear.
\end{proof}

\begin{proposition} \label{prop:tor3}
Suppose that $S \subset R$ is a graded subring with $S_0=\bk$ such that $R$ is finite projective over $S$. Let $M$ be an $R$-module. Then
\begin{displaymath}
t_1(M; S) \le \max(t_0(M; R), t_1(M; R)) + t_0(R; S)
\end{displaymath}
and
\begin{displaymath}
t_1(M; R) \le \max(t_0(M; R)+t_0(R; S), t_1(M; S))
\end{displaymath}
\end{proposition}

\begin{proof}
Let
\begin{displaymath}
0 \to K \to F \to M \to 0
\end{displaymath}
be an exact sequence of $R$-modules such that $F$ is projective and $t_0(F; R)=t_0(M; R)$. Then $F$ is also projective as an $S$-module. Thus
\begin{displaymath}
t_1(M; S) \le t_0(K; S) \le t_0(K; R)+t_0(R; S) \le \max(t_0(M; R), t_1(M; R))+t_0(R; S).
\end{displaymath}
The outer two inequalities come from Proposition~\ref{prop:tor2}, while the middle one is clear. Similarly,
\begin{displaymath}
t_1(M; R) \le t_0(K; R) \le t_0(K; S) \le \max(t_1(M; S), t_0(F; S)).
\end{displaymath}
The stated inequality follows from this and the identity $t_0(F;S)=t_0(F;R)+t_0(R;S)$.
\end{proof}

\begin{proposition} \label{prop:tor-vanishing}
Let $M$ be a finitely presented graded $R$-module. Then $\Tor^R_i(M, \bk)_d=0$ for $i \gg d$.
\end{proposition}

\begin{proof}
Suppose $M$ is generated in degrees $\le m$, and $n$ is the minimal integer such that $M_n \ne 0$. Let $V_0=\bigoplus_{i=n}^m M_i$. We then have a natural surjection $R \otimes_{\bk} V_0 \to M$. Moreover, this map is an isomorphism in degree $n$, and so the kernel is supported in degrees $>n$. Applying the same reasoning to the kernel and proceeding inductively, we obtain a resolution $R \otimes_{\bk} V_{\bullet} \to M$ where the degree $d$ piece of $V_i$ vanishes for $i \gg d$. By Proposition~\ref{prop:relatively-projective} the modules $R \otimes_{\bk} V_{\bullet}$ are acyclic for the functor $- \otimes_{R} \bk$, and so $\Tor^{R}_{\bullet}(M, \bk)$ is computed by the complex $V_{\bullet}$, which proves the result.
\end{proof}

\subsection{Bounding relations} \label{ss:relbd}

The following theorem is our main result on bounding the presentation of ideals in a GDPA. The proof occupies the entire section.

\begin{theorem} \label{thm:torbd}
Let $\bD$ be a GDPA over a ring $\bk$, and let $I \subset \bD$ be a homogeneous ideal of $\bD$ generated in degrees $\le d$. Write $I_n=\fa_n x^{[n]}$ where $\fa_n$ is an ideal of $\bk$. Suppose that $N$ is such that $T^{[h]}(\bk/(\fa_i+T(\bk)))$ is killed by $a^{[h]}(N)$ for all $1 \le h \le 2d$ and all $0 \le i \le 3d$. Then $t_1(I/T(I); \bD) \le (2N+3)d$.%\Acom{Rohit: Earlier we wrote $N+3d$.}
\end{theorem}

\begin{lemma}
It suffices to prove the theorem when $\bk$ is local.
\end{lemma}

\begin{proof}
Suppose the result holds when $\bk$ is local, and let us prove it in general. Let $\fm$ be a maximal ideal of $\bk$, and indicate by a prime localization at $\fm$. Since $T$ commutes with localization, $T^{[h]}(\bk'/(\fa_i'+T(\bk')))$ is killed by $a^{[h]}(N)$ for all $1 \le h \le 2d$ and $0 \le i \le 3d$. Thus $t_1(I'/T(I'); \bD') \le N+3d$ by the local case of the theorem. That is, $\Tor_1^{\bD'}(I'/T(I'), \bk')_n=0$ for $n>N+3d$. But $\Tor$ commutes with localization, so we find that the $\bk$-module $\Tor_1^{\bD}(I/T(I), \bk)_n$ localizes to~0 at all maximal ideals $\fm$. It is therefore~0, and so $t_1(I/T(I); \bD) \le N+3d$, as was to be shown.
\end{proof}

For the rest of the section, we assume that $\bk$ is local. We let $\fm$ be the unique maximal ideal of $\bk$, and we let $b_{\bullet}=b_{\fm,\bullet}$ when we have an admissible sequence $\pi_{\bullet}$. For an element $\rho$ of $\bk$ and a $\bk$-module $M$, we let $T_{\rho}(M)$ be the submodule consisting of elements annihilated by $a(n) \rho$ for some $n \ge 1$.

\begin{lemma} \label{degdb1}
Let $I$ be a homogeneous ideal of $\bD$ generated in degrees $\le d$, where $d<b_1$. Fix an element $\rho$ of $\bk$. Suppose that $T(\bk/(\fa_i+T_{\rho}(\bk)))$ is killed by $a(N)$, for all $0 \le i \le d$. Let $t$ be maximal subject to $b_t \mid N$. Then $\Tor^{\bD}_1(I/T_{\rho}(I), \bk)_n=0$ unless $n=b_r+\ell$ for some $1 \le r \le t$, and $0 \le \ell<d$.
\end{lemma}

\begin{proof}
Let $F_0=\bigoplus_{i=0}^d \bD[i]$, let $e_0, \ldots, e_d$ be its standard basis, and let $\Phi_0 \colon F_0 \to \bD$ be the map defined by $\Phi_0(e_i)=x^{[i]}$. Let $F=\bigoplus_{i=0}^d \bD[i] \otimes_{\bk} \fa_i$, thought of as a $\bD$-submodule of $F_0$, let $\Phi \colon F \to I$ be the restriction of $\Phi_0$ to $F$, and let $K=\Phi^{-1}(T_1^{\rho}(I))$. We have a short exact sequence
\begin{displaymath}
0 \to K \to F \to I/T_1^{\rho}(I) \to 0,
\end{displaymath}
and so $\Tor^{\bD}_1(I/T_1^{\rho}(I), \bk)$ is identified with $(K \cap \bD_+ F)/\bD_+ K$. Let $R=\sum_{i=0}^d \gamma_i x^{[n-i]} e_i$ be an element of $K \cap \bD_+F$ of degree $n$, and let $\ol{R}$ be the corresponding element of $\Tor_1$. Write $n=m+\ell$, where $m$ is a multiple of $b_1$ and $0 \le \ell < b_1$. We must show that $\ol{R}=0$ unless $m=b_r+\ell$ for some $1 \le r \le t$ and $0 \le \ell < d$. We proceed in six steps.

{\it Step 1: excluding $m=0$.} Suppose $m=0$, i.e., $n<b_1$. Since $R \in \bD_+ F$, we can write $R=x^{[k]} R'$ for some $R' \in F$ and some $0<k \le n$. Write $\Phi(R)=\epsilon x^{[n]}$ with $\epsilon \in T_1^{\rho}(\bk)$ and $\Phi(R')=\delta x^{[n-k]}$. Then $\epsilon=C(n,k) \delta$. Since $C(n,k)$ is a unit (as $n<b_1$), it follows that $\delta \in T_1^{\rho}(\bk)$. Thus $R' \in K$, and so $R \in \bD_+ K$, and so $\ol{R}=0$. In what follows, we assume $m>0$.

{\it Step 2: bounding $\ell$.} Now suppose $\ell \ge d$. Then $C(n-i,\ell-i)$ is a unit for all $0 \le i \le d$, and so
\begin{displaymath}
R' = \sum_{i=0}^d \frac{\gamma_i x^{[\ell-i]}}{C(n-i,\ell-i)} e_i
\end{displaymath}
is a well-defined element of $F$. Note that $R=x^{[m]}R'$. If $\Phi(R)=\epsilon x^{[n]}$ and $\Phi(R')=\delta x^{[\ell]}$ then $\epsilon=C(n,\ell) \delta$. Since $C(n,\ell)$ is a unit and $\epsilon \in T_1^{\rho}(\bk)$, it follows that $\delta \in T_1^{\rho}(\bk)$, and so $R' \in K$. Thus $R \in \bD_+ K$, and so $\ol{R}=0$. In what follows, we assume $\ell<d$.

{\it Step 3: reduction to two-term relations.} We now show that it suffices to consider elements $R$ of a simple form. This step is not strictly necessary, but will simplify notation in what follows. Put
\begin{displaymath}
\alpha=\sum_{i=0}^{\ell} \frac{C(n,n-i)}{C(n,n-\ell)} \gamma_i, \qquad
\beta=\sum_{i=\ell+1}^d \frac{C(n,n-i)}{C(n,n-d)} \gamma_i.
\end{displaymath}
Note that $C(n,n-\ell)$ is a unit, and, in the second sum, both $C(n,n-i)$ and $C(n,n-d)$ are the same product of $\pi$'s, up to units, and so the expressions make sense. Note also that $\fa_0 \subset \fa_1 \subset \cdots \subset \fa_d$, and so $\alpha \in \fa_{\ell}$ and $\beta \in \fa_d$. We can thus consider the element
\begin{displaymath}
R'=\alpha x^{[m]}e_{\ell}+\beta x^{[n-d]}e_d
\end{displaymath}
of $F$. It is clear that $\Phi(R')=\Phi(R)$, and so $R' \in K$. We claim $R=R'$ modulo $\bD_+K$. To see this, put
\begin{displaymath}
R_1=\sum_{i=0}^{\ell} \frac{\gamma_i x^{[\ell-i]}}{C(n-i,\ell-i)} e_i - \alpha e_{\ell}
\end{displaymath}
and
\begin{displaymath}
R_2=\sum_{i=\ell+1}^d \frac{\gamma_i x^{[n-i-1]}}{C(n-i,1)} e_i - \frac{\beta x^{[n-d-1]}}{C(n-d, 1)} e_d.
\end{displaymath}
All the $C$'s in these expressions are units, and so $R_1$ and $R_2$ are well-defined elements of $F$. We have
\begin{displaymath}
x^{[m]} R_1 = \sum_{i=0}^{\ell} \gamma_i x^{[n-i]} e_i - \alpha x^{[m]} e_{\ell}, \qquad
x^{[1]} R_2 = \sum_{i=\ell+1}^d \gamma_i x^{[n-i]} e_i - \beta x^{[n-d]} e_d,
\end{displaymath}
and so
\begin{displaymath}
R-R' = x^{[m]} R_1 + x^{[1]} R_2.
\end{displaymath}
From the definitions of $\alpha$ and $\beta$, it is clear that $\Phi(x^{[m]} R_1)$ and $\Phi(x^{[1]} R_2)$ vanish. Writing $\Phi(R_1)=\epsilon x^{[\ell]}$ and $\Phi(R_2)=\delta x^{[n-1]}$, we find $C(n,\ell) \epsilon=0$ and $C(n,1) \delta=0$. Thus $\epsilon=0$ and $\delta$ is at least torsion, and so $R_1$ and $R_2$ belong to $K$. Thus $R-R' \in \bD_+ K$, as claimed. In what follows, we assume $R=\alpha x^{[m]} e_{\ell}+\beta x^{[n-d]} e_d$.

{\it Step 4: reduction to $m=b_r$.} Suppose $m$ is not of the form $b_r$. Then we can write $m=m_1+m_2$ where $m_1$ and $m_2$ are positive multiples of $q=b_1$ such that there is no carry when computing either $m_1+m_2$ or $m_1+(m_2-q)$. Put
\begin{displaymath}
R' = \frac{\alpha x^{[m_2]}}{C(m,m_1)} e_{\ell} + \frac{\beta x^{[m_2+\ell-d]}}{C(n-d,m_1)} e_d.
\end{displaymath}
Both the $C$'s here are units: for the second one, observe that
\begin{displaymath}
n-d=m_1+(m_2+\ell-d)=m_1+((m_2-q)+(q+\ell-d)).
\end{displaymath}
Since $q+\ell-d<q$ and $m_1$ and $m_2-q$ are multiples of $q$, there are no carries. Thus $R'$ is an element of $F$, and $R=x^{[m_1]} R'$. Writing $\Phi(R)=\epsilon x^{[n]}$ and $\Phi(R')=\delta x^{[m_2+\ell]}$, we see that $C(n,m_1) \delta=\epsilon$. Since $C(n,m_1)$ is a unit, this shows that $\delta$ is torsion. Thus $R' \in K$ and $R \in \bD_+K$ and $\ol{R}=0$. In what follows, we assume $m=b_r$ for some $r \ge 1$.

{\it Step 5: bounding $r$.} Suppose now that $r>t$ and $r>1$. Put $\Phi(R) = \epsilon x^{[n]}$, where $\epsilon \in T_1^{\rho}(\bk)$. We have
\begin{displaymath}
\epsilon = C(n, \ell) \alpha+C(n,d) \beta.
\end{displaymath}
Note that $C(n,\ell)$ is a unit and $C(n,d)$ is equal to $\pi_{b_1} \cdots \pi_{b_r}$ up to units. We thus see that the image of $\beta$ in $\bk/(\fa_{\ell}+T_1^{\rho}(\bk))$ is killed by $\pi_{b_1} \cdots \pi_{b_r}$. It is therefore killed by $\pi_{b_1} \cdots \pi_{b_{r-1}}$, by our assumption: note that $a(N)=\pi_{b_1} \cdots \pi_{b_t}$, up to units. Let $m'=b_{r-1}$ and $n'=m'+\ell$, so that $C(n',d)$ is $\pi_{n_1} \cdots \pi_{b_{r-1}}$, up to units. The above discussion shows that we can write
\begin{displaymath}
\epsilon' = C(n', \ell) \alpha' + C(n', d) \beta
\end{displaymath}
for some $\alpha' \in \fa_{\ell}$ and $\epsilon' \in T_1^{\rho}(\bk)$. Now consider
\begin{displaymath}
R'=\alpha' x^{[m']} e_{\ell} + \beta x^{[n'-d]} e_d.
\end{displaymath}
Clearly $\Phi(R')=\epsilon'$, and so $R' \in K$. We have
\begin{displaymath}
x^{[m-m']} R' = \alpha' C(m,m') x^{[m]} e_{\ell} + \beta C(n-d,n'-d) x^{[n-d]} e_d.
\end{displaymath}
We thus find
\begin{displaymath}
C(n-d,n'-d) R-x^{[m-m']} R' = \delta x^{[m]} e_{\ell}
\end{displaymath}
for some $\delta \in \bk$. Since both $R$ and $R'$ belong to $K$, so does $\delta x^{[m]} e_{\ell}$. As $\Phi(\delta x^{[m]} e_{\ell}) = C(n, \ell) \delta x^{[n]}$, we see that $C(n,\ell) \delta$ belongs to $T_1^{\rho}(\bk)$. But $C(n, \ell)$ is a unit, and so $\delta$ itself belongs to $T_1^{\rho}(\bk)$. Thus $\delta x^{[m]} e_{\ell} \in \bD_+ K$. This shows that
\begin{displaymath}
C(n-d,n'-d) R-x^{[m-m']} R' \in \bD_+ K.
\end{displaymath}
Of course, $x^{[m-m']} R' \in \bD_+ K$ as well, and $C(n-d,n'-d)$ is a unit, so we find $R \in \bD_+K$ and $\ol{R}=0$.

{\it Step 6: the $r=1$ case.} Finally, suppose that $r>t$ and $r=1$. Note then that $t \le 0$, and so $\bk/(\fa_{\ell}+\bk_{\tors})$ has no torsion. From $\Phi(R)=0$ we see that $\epsilon=C(n, \ell)\alpha+C(n,d) \beta$ belongs to $\bk_{\tors}$. Thus the image of $\beta$ in $\bk/(\fa_{\ell}+\bk_{\tors})$ is killed by $C(n,d)$, and is therefore torsion, and therefore vanishes. We can therefore write $\alpha'+\beta=\epsilon'$ for some $\alpha' \in \fa_{\ell}$ and $\epsilon' \in \bk_{\tors}$. Consider
\begin{displaymath}
R'=\frac{C(n,d) \alpha' x^{[m-1]}}{C(m,1)} e_{\ell}+\frac{\beta x^{[n-d-1]}}{C(n-d,1)} e_d
\end{displaymath}
Note that $C(n,d)$ and $C(m,1)$ are same product of $\pi$'s, up to units, and $C(n-d,1)$ is a unit. We have
\begin{displaymath}
x^{[1]} R'=\alpha'C(n,d) x^{[m]} e_{\ell}+\beta x^{[n-d]} e_d=R+\epsilon'' x^{[m]} e_{\ell}
\end{displaymath}
where $\epsilon''=C(n,d)\epsilon'-\epsilon$. Write $\Phi(R)=\delta x^{[n]}$ with $\delta \in \bk_{\tors}$ and $\Phi(R')=\delta' x^{[n-1]}$. We claim that $\delta' \in K$. First suppose $\ell>0$. Then the above equation shows $C(n,1) \delta'=\delta+\epsilon'' C(n,\ell)$. Since $\ell>0$ we have that $C(n,1)$ is a unit, and so $\delta'$ is torsion. Now suppose $\ell=0$. Then computing $\wt{\Phi}(R')$ directly from the definition of $R'$, we find
\begin{displaymath}
\delta' = \frac{C(n,d)}{C(n,1)} \alpha' + \frac{C(n-1,d)}{C(n-d,1)} \beta
\end{displaymath}
(note $n=m$ since $\ell=0$). All the $C$'s above are units, and the two fractions are equal since $C(n,d) C(n-d,1)=C(n,1) C(n-1,d)$; thus $\delta'$ is a unit times $\alpha'+\beta=\epsilon'$, and is thus torsion. We have thus shown that $\delta'$ is torsion in all cases, and so $R' \in K$. Since $\epsilon'' x^{[m]} e_{\ell} \in \bD_+ K$, this shows $R \in \bD_+ K$.
\end{proof}

\begin{lemma} \label{degbd2}
Let $h=b_s$ for $s \ge 0$ and let $I$ be a homogeneous ideal of $\bD^{(h)}$ generated in degrees $\le d$, where $d<b_{s+1}$. Suppose $T^{[h]}(\bk/(\fa_i+T(\bk)))$ is killed by $a^{[h]}(N)$, for all $0 \le i \le d$. Let $t$ be maximal subject to $b_t \mid N$. Then $\Tor_1^{\bD^{(h)}}(I/T(I),\bk)_n=0$ unless $n=b_r+\ell$ for some $s < r \le t$ and $0 \le \ell \le d$. In particular, $t_1(I/T(I); \bD^{(h)}) \le Nh+d$.
%\Acom{Rohit: Earlier we wrote $N+d$.}
\end{lemma}

\begin{proof}
The regrade of $\bD^{(h)}$ is isomorphic to $\bD^{[h]}$. The regrade of $I$ is an ideal of $\bD^{[h]}$ generated in degres $\le d/h < b^{[h]}_1$. Let $\rho=\pi_{b_1} \cdots \pi_{b_s}$. Then $T_{\rho}(-; \pi^{[h]}_{\bullet})=T(-; \pi_{\bullet})$. With these identifications, the result follows from the Lemma~\ref{degdb1}.
\end{proof}

\begin{proof}[Proof of Theorem~\ref{thm:torbd}]
Let $s$ be maximal subject to $b_s \le d$. First suppose that $d+b_{s} \le b_{s+1}$. Put $h=b_s$. We have an isomorphism of $\bD^{(h)}$-modules $I=\bigoplus_{k=0}^{h-1} J_k[k]$, where $J_k$ is an ideal of $\bD^{(h)}$. Put $d_k=h \lceil (d-k)/h \rceil$. Then by Proposition~\ref{gdpa:ideal}, the ideal $J_k$ is generated in degrees $\le d_k$ and satisfies $J_{k,i}=\fa_{k+i} x^{[i]}$ (for $i$ a multiple of $h$). Note that $d_k \le d+h-1$, so $d_k \le 2d$ and $d_k<b_{s+1}$. The module $T^{[h]}(\bk/(\fa_i+T(\bk)))$ is killed by $a^{[h]}(N)$ for $i=k,\ldots,k+d_k$, as $k+d_k \le 3d$, and so $t_1(J_k/T(J_k); \bD^{(h)}) \le Nh+d_k$ by Lemma~\ref{degbd2}.
%\Acom{Rohit: Earlier we wrote just $N$ instead of $Nh$ from here on in this proof.}.
Thus $t_1(J_k[k]/T(J_k[k]); \bD^{(h)}) \le Nh+d_k+k \le Nh+3d$, and so $t_1(I/T(I); \bD^{(h)}) \le Nh+3d \le (N+3)d$.
%\Acom{Rohit: Earlier we wrote $N + 3d$.}.
Thus $t_1(I/T(I); \bD) \le (N+3)d$ by Proposition~\ref{prop:tor3}. 

Now suppose that $d+b_s>b_{s+1}$; note that this implies $b_{s+1}<2d$. Put $h=b_{s+1}$. As $\bD^{(h)}$-modules, we have $I=\bigoplus_{k=0}^{h-1} J_k[k]$, where once again $J_k$ is an ideal of $\bD^{(h)}$. If $0 \le k \le d$ then $J_k$ is generated in degrees 0 and $h$, while if $d<k \le h-1$ then $J_k$ is generated in degree 0. We have $t_1(J_k; \bD^{(h)})=0$ for $k>d$. Now suppose $0 \le k \le d$. The module $T^{[h]}(\bk/(\fa_i+T(\bk)))$ is killed by $a^{[h]}(N)$ for $i=k$ and $i=k+h$, as $k+h \le 3d$. Thus $t_1(J_k/T(J_k); \bD^{(h)}) \le Nh+h$ by Lemma~\ref{degbd2}. The same reasoning as in the previous paragraph gives $t_1(I/T(I); \bD) \le Nh+h + d \le h(N+1) +d \le (2N+3)d$.
\end{proof}

\subsection{Graded-coherence} \label{ss:grcoh}

Let $\pi_{\bullet}$ be a $\pi$-sequence in $\bk$. We define a {\bf $\pi$-ideal} to be an ideal of $\bk$ generated by some of the $\pi_n$'s. We will consider the following three conditions:
\begin{itemize}
\item[(A1)] The ring $\bk$ is coherent.
\item[(A2)] For any finitely generated ideal $\fa$ of $\bk$ and any $h \ge 1$, the torsion $T^{[h]}(\bk/\fa)$ is bounded, that is, annihilated by $a^{[h]}(n)$ for some $n$ (depending on $h$ and $\fa$).
\item[(A3)] Given any strictly ascending chain of finitely generated $\pi^{[h]}$-ideals $\fa_1 \subset \fa_2 \subset \cdots$ there is some $i$ for which the quotient $\bD^{[h]}/\fa_i \bD^{[h]}$ is graded-coherent.
\end{itemize}
Note that (A3) is vacuously true if there are no strictly ascending chains of $\pi^{[h]}$-ideals. This is the case, for instance, if the $\pi_n$'s belong to the image of a ring homomorphism $\bk_0 \to \bk$ with $\bk_0$ noetherian. Our main result on graded-coherence of GDPAs is the following theorem.

\begin{theorem} \label{thm:grdcoh}
Let $\bD=\bD(\bk, \pi_{\bullet})$ be a generalized divided power algebra. Then $\bD^{[h]}$ is graded-coherent for all $h \ge 1$ if and only if conditions (A1), (A2), and (A3) hold.
\end{theorem}

\begin{corollary} \label{cor:grcohdp}
Let $\bD$ be the classical divided power algebra over $\bk$. Then $\bD$ is graded-coherent if and only if $\bk$ is coherent and for all finitely generated ideals $\fa$ of $\bk$ the module $(\bk/\fa)_{\tors}$ is annihilated by some nonzero integer.
\end{corollary}

\begin{remark}
To conclude that $\bD$ is graded-coherent, it is enough to know that (A3) holds for $h=1$. One still needs (A2) for all $h \ge 1$, however.
\end{remark}

The rest of this section is devoted to the proof of the theorem. We fix $\bD=\bD(\bk, \pi_{\bullet})$.

\begin{lemma} \label{grcoh1}
Suppose (A1) and (A2) hold. Then for any finitely generated ideal $\fa$ of $\bk$ and any $h \ge 1$ the $\bk$-module $T^{[h]}(\bk/\fa)$ is finitely presented.
\end{lemma}

\begin{proof}
By (A2), $T^{[h]}(\bk/\fa)$ is equal to the kernel of the multiplication-by-$a(n)$ map $\bk/\fa \to \bk/\fa$ for some $n$. Since $\bk/\fa$ is finitely presented and $\bk$ is coherent, this kernel is again finitely presented.
\end{proof}

\begin{lemma} \label{grcoh2}
Suppose (A1) and (A2) hold. Let $I$ be a finitely generated ideal of $\bD$. Then $I/T(I)$ is finitely presented.
\end{lemma}

\begin{proof}
By (A2) and Theorem~\ref{thm:torbd}, $t_1(I/T(I); \bD) < \infty$. By (A1), $\Tor_1^{\bD}(I/T(I), \bk)_n$ is finitely presented as a $\bk$-module for each $n$. Combining these two statements, we see that $\Tor^1_{\bD}(I/T(I), \bk)$ is finitely presented as a $\bk$-module, and this implies that $I/T_1(I)$ is finitely presented as a $\bD$-module.
\end{proof}

\begin{lemma} \label{grcoh3}
Suppose that conditions (A1) and (A2) hold. Suppose also that for any $n \ge 2$ the quotient $\bD/(\pi_n)$ is graded-coherent. Then $\bD$ is graded-coherent.
\end{lemma}

\begin{proof}
Let $I$ be a finitely generated ideal of $\bD$. Then $I/T(I)$ is finitely presented by Lemma~\ref{grcoh2}. It follows that $T(I)$ is finitely generated, and to prove the proposition it is enough to show that it is finitely presented.

Suppose $T(\bk)$ is killed by $a(N)$, and let $\bD'=\bD/(a(N))$. The following lemma shows that $\bD'$ is graded-coherent. The natural map $\bD \otimes_{\bk} T(\bk) \to T(\bD)$ is an isomorphism. By Lemma~\ref{grcoh1}, $T(\bk)$ is finitely presented as a $\bk$-module, and so $T(\bD)$ is finitely presented as a $\bD$-module. Since $a(N)$ kills $T(\bD)$, it follows that $T(\bD)$ is finitely presented as a $\bD'$-module. As $T(I)$ is a finitely generated $\bD'$-submodule of $T(\bD)$, it is therefore finitely presented as a $\bD'$-module, and therefore as a $\bD$-module.
\end{proof}

\begin{lemma} \label{grcoh4}
Let $\bk$ be a coherent ring, let $R$ be a graded $\bk$-algebra that is flat over $\bk$, and let $x$ and $y$ be elements of $\bk$ such that $R/(x)$ and $R/(y)$ are graded-coherent. Then $R/(xy)$ is graded-coherent.
\end{lemma}

\begin{proof}
Consider the 4-term exact sequence
\begin{displaymath}
0 \to I \to \bk/(x) \stackrel{y}{\to} \bk/(xy) \to \bk/(y) \to 0
\end{displaymath}
Since $\bk$ is coherent the ideal $I$ is finitely generated. Tensoring up with $R$, we find
\begin{displaymath}
0 \to I \otimes_{\bk} R \to R/(x) \stackrel{y}{\to} R/(xy) \to R/(y) \to 0
\end{displaymath}
This shows that $R/(xy)$ maps onto the graded-coherent ring $R/(y)$ with finitely presented kernel (the left two terms gives a presentation for the kernel), and is thus graded-coherent.
\end{proof}

\begin{lemma} \label{grcoh5}
Suppose that conditions (A1) and (A2) hold. Suppose also that there is a positive integer $h$ such that $\pi_h$ belongs to the Jacobson radical of $\bk$ and the quotient $\bD/(\pi_n)$ is graded-coherent for all proper multiples $n$ of $h$. Then $\bD$ is graded-coherent.
\end{lemma}

\begin{proof}
It suffices to show that $\bD^{(h)}$ is graded-coherent. Of course, it is the same to show that its regrade is coherent, and this is isomorphic to $\bD^{[h]}$. Conditions (A2) still holds for $\pi^{[h]}_{\bullet}$, since $T^{[h']}(-, \pi^{[h]}_{\bullet})=T^{[hh']}(-, \pi_{\bullet})$ by Proposition~\ref{dbltrans}. Since $\bD/(\pi_{hn})$ is graded-coherent for $n>1$, it follows that $\bD^{(h)}/(\pi_{hn})$ is graded-coherent (by Proposition~\ref{gdpa:transcoh}) for $n>1$. Since $\pi^{[h]}_n$ is associate to $\pi_{hn}$ by Proposition~\ref{gdpa:elt}(b), it follows  that $\bD^{[h]}/(\pi^{[h]}_n)$ is graded-coherent for $n>1$ as well. Thus $\bD^{[h]}$ is graded-coherent by Lemma~\ref{grcoh3}.
\end{proof}

\begin{lemma} \label{grcoh6}
Suppose (A1), (A2), and (A3) hold. Then $\bD^{[h]}$ is graded-coherent for all $h \ge 1$.
\end{lemma}

\begin{proof}
It suffices to treat the $h=1$ case. Let $\Sigma$ be the set of $\pi$-ideals $\fa$ for which $\bD/\fa$ is not graded-coherent. It suffices to show that $\Sigma$ is empty, for then $\bD/\fa$ is graded-coherent for every $\pi$-ideal $\fa$, including $\fa=(0)$. Thus suppose for the sake of contradiction that $\Sigma$ is non-empty. Let $\fa$ be a maximal element of $\Sigma$, which exists by (A3), and put $\bD'=\bD/\fa \bD$. Now, $\fa$ contains only finitely many of the $\pi_n$'s, for otherwise $\bD'$ would be coherent by Proposition~\ref{gdpa:coh}. Let $h$ be the largest integer such that $\pi_h \in \fa$. (Note that $\pi_1=0$ is in $\fa$, so this makes sense.) If $n$ is a proper multiple of $h$ then $\pi_n$ does not belong to $\fa$, and so $\fa+(\pi_n)$ is not in $\Sigma$, and so $\bD'/\pi_n \bD'$ is graded-coherent. Thus $\bD'$ is graded-coherent by Lemma~\ref{grcoh5}, a contradiction. We thus see that $\Sigma$ is empty, which completes the proof.
\end{proof}

\begin{lemma} \label{grcoh7}
Suppose $\bD$ is graded-coherent. Then for any finitely generated ideal $\fa$ of $\bk$ the torsion $T(\bk/\fa)$ is bounded.
\end{lemma}

\begin{proof}
Let $\fa$ be given, and let $I$ be the ideal of $\bD$ generated by $\fa x^{[0]}$ and $x^{[1]}$. Use notation as in the first paragraph of the proof of Lemma~\ref{degdb1}. Since $\bD$ is graded-coherent, $K$ is finitely generated. Suppose it is generated in degrees $\le d$. Let $\ol{\beta} \in \bk/\fa$ be killed by $a(n)$. Let $\beta \in \bk$ be a lift of $\ol{\beta}$, and let $\alpha=-a(n) \beta$, so that $\alpha \in \fa$. Let $R=\alpha x^{[n]} e_0+\beta x^{[n-1]} e_1$. Then $\Phi(R)=0$, so $R \in K$. Since $K$ is generated in degrees $\le d$, we can write $R=\sum_{i=1}^d x^{[n-i]} R_i$, where $R_i$ is a degree $i$ element of $K$. (Note that $K_0=0$.) Write $R_i=\alpha_i x^{[i]} e_0+\beta_i x^{[i-1]} e_1$ with $\alpha_i \in \fa$. Since $\Phi(R_i)=0$, we find $\alpha_i+a(i) \beta_i=0$, so, writing $\ol{\beta}_i$ for the image of $\beta_i$ in $\bk/\fa$, we find $a(i) \ol{\beta}_i=0$. From the expression relating $R$ and the $R_i$, we find $\beta=\sum_{i=1}^d C(n-1,i-1) \beta_i$. Thus if $D$ is a common multiple of $1, \ldots, d$ (e.g., $D=d!$) then $a(D) \ol{\beta}=0$. Therefore, $T(\bk/\fa)$ is killed by $a(D)$, which completes the proof.
\end{proof}

\begin{proof}[Proof of Theorem~\ref{thm:grdcoh}]
If (A1), (A2), and (A3) hold then $\bD^{[h]}$ is graded-coherent for all $h \ge 1$ by Lemma~\ref{grcoh6}. Conversely, if $\bD^{[h]}$ is graded-coherent then (A1) necessarily holds, as does (A3), since any quotient of $\bD^{[h]}$ by a finitely generated homogeneous ideal is graded-coherent, and (A2) holds at $h$ by Lemma~\ref{grcoh7}. Thus if $\bD^{[h]}$ is graded-coherent for all $h \ge 1$ then (A2) holds as well.
\end{proof}

\subsection{Divided power algebras in multiple variables} \label{ss:two-variable}

Let $\bk$ be the classical divided power algebra over $\bZ_{(p)}$ in a single variable $y$, regarded as a non-graded ring, and let $\bD$ be the classical divided power algebra over $\bk$ in a single variable $x$. Note that as an non-graded ring, $\bD$ is just the classical divided power algebra over $\bk$ in $x$ and $y$. Let $\fa$ be the principal ideal of $\bk$ generated by $y^{[1]}$. Then the image of $y^{[p^r]}$ in $\bk/\fa$ is annihilated by $p^r$ but no smaller power of $p$. Thus condition (A2) fails, and so $\bD$ is not graded-coherent, and therefore not coherent either.

Here is an explicit example demonstrating the failure of coherence. Consider the ideal $I$ of $\bD$ generated by $y^{[1]}$ and $x^{[1]}$. Then
\begin{displaymath}
(y^{[p^r-1]} x^{[p^r]}) y^{[1]} - (y^{[p^r]} x^{[p^r-1]}) x^{[1]} = 0
\end{displaymath}
is a linear relation of between the generators, and does not come from lower degree relations. Thus $I$ is not finitely presented. Note that if we remember the grading on $\bk$ and regard $\bD$ as bigraded then $I$ is bihomogeneous. This shows that $\bD$ is not even graded-coherent with respect to its bigrading.

\section{Special resolutions} \label{s:sr}

\subsection{Statement of results}

Let $\bD$ be a GDPA over $\bk$. Let $M$ be a $\bD$-module. We say that $M$ is {\bf principal special} if it is isomorphic to a module of the form $(\bD/\fa \bD)^{(h)}$ for some integer $h \ge 1$ and ideal $\fa$ of $\bk$ containing $\pi_h$. We say that $M$ is {\bf special} if it admits a finite length filtration where the graded pieces are principal special. A {\bf special resolution} of a $\bD$-module $M$ is a resolution $S_{\bullet} \to M$ where each $S_i$ is special. The {\bf special dimension} of $M$, denoted $\sd(M; \bD)$ or just $\sd(M)$, is the minimum integer $n$ for which there exists a special resolution $S_{\bullet} \to M$ with $S_i=0$ for all $i>n$, or $\infty$ if no such resolution exists. Note that $\sd(M)=0$ if and only if $M$ is special.

\begin{theorem} \label{thm:sd}
Suppose $\bk$ is noetherian and $M$ is a finitely presented $\bD$-module. Then $\sd(M)<\infty$, that is, $M$ admits a finite length special resolution. Moreover, if $\bk$ has finite Krull dimension $d$ then $\sd(M) \le d+1$.
\end{theorem}

The proof will take the entire section. The basic idea is as follows. First suppose $\bk$ is a field. If infinitely many $\pi_n$'s vanish then every finitely presented module is special. If not (e.g., if $\bD=\bk[x]$) this need not be the case, but at least every submodule of a finitely generated free module is special. Now suppose $\bk$ is a domain with fraction field $\bK$, let $M$ be a finitely presented $\bD$-module, and choose short exact sequence
\begin{displaymath}
0 \to N \to F \to M \to 0
\end{displaymath}
with $F$ finite free. Then, as stated above, $N \otimes_{\bk} \bK$ is special as a $\bD_{\bK}$-module. We show that there is a special $\bD$-module $N'$ such that $N \otimes_{\bk} \bK \cong N' \otimes_{\bk} \bK$. Scaling this isomorphism appropriately, we can assume $N' \subset N$ and that the quotient is a torsion $\bk$-module. The result now follows from noetherian induction and the fact that in a short exact sequence the special dimension of one term can be controlled by the that of the other two.

The notions of special module and special resolution can easily be adapted to the case of graded modules. The above theorem remains true in the graded case, and the proof we give applies without change. In fact, everything still goes through if the coefficient ring $\bk$ is equipped with a grading, though in this case some minor adjustments in our proof must be made (specifically when using the fraction field of $\bk$).

\subsection{Preliminary results}

We begin by proving some basic facts about special resolutions and special dimension.

\begin{lemma} \label{sp1}
Suppose that $f \colon M \to N$ is a surjection of principal special $\bD$-modules. Then $\ker(f)$ is special.
\end{lemma}

\begin{proof}
Suppose $M=(\bD/\fa \bD)^{(r)}$ and $N=(\bD/\fb \bD)^{(s)}$. We proceed by a number of reductions. First, since $f$ is surjective, every element of $N$ is annihilated by $\fa$. But the annihilator of $1 \in N$ is exactly $\fb$, and so $\fa \subset \fb$. Replacing $\bk$ by $\bk/\fa$, we may as well assume $\fa=0$. The map $f$ obviously factors as $M \to M/\fb M \stackrel{\ol{f}}{\to} N$ for some $\ol{f}$. We have a short exact sequence
\begin{displaymath}
0 \to \fb M \to \ker(f) \to \ker(\ol{f}) \to 0.
\end{displaymath}
The module $\fb M$ is special (filter $\fb$ so that the graded pieces are cyclic $\bk$-modules), and so it suffices to show that $\ker(\ol{f})$ is special. Thus we may as well replace $M$ with $M/\fb M$. Furthermore, replacing $\bk$ with $\bk/\fb$, we may assume $\fb=0$.

We have thus reduced to the case $M=\bD^{(r)}$ and $N=\bD^{(s)}$. We claim $s \ge r$. Indeed, suppose $s<r$. Then $x^{[s]}$ annihilates $M$, and therefore annihilates $N$. But this is not the case, since $x^{[s]} \cdot 1 \ne 0$ in $N$. Now, since $\pi_r=\pi_s=0$, we must have $r \mid s$ or $s \mid r$, and so, by the inequality $s \ge r$, we have $r \mid s$. We may as well regrade and replace $\bD^{(r)}$ with $\bD$. In other words, we may assume $M=\bD$ and $N=\bD^{(s)}$.

Now, $N=\bD/I$, where $I$ is the ideal $(x^{[1]}, \ldots, x^{[s-1]})$. Thus $f$ factors as $M \to M/I \stackrel{\ol{f}}{\to} N$ for some $\ol{f}$. Since $M/I$ is equal to $N$, we may as well regard $\ol{f}$ as an endomorphism of $N$. Since it is surjective and $N$ is finitely generated, it is necessarily an isomorphism \stacks{05G8}. Thus $\ker(f)=I$. One easily sees that $I$ has a finite length filtration with successive quotients $\bD^{(s)}$, and is therefore special.
\end{proof}

\begin{lemma} \label{sp2}
Let $M$ be a special $\bD$-module. Then there exists a finite free $\bD$-module $F$ and a surjection $f \colon F \to M$ such that $\ker(f)$ is special.
\end{lemma}

\begin{proof}
Let
\begin{displaymath}
0 \to M'' \to M \to M' \to 0
\end{displaymath}
be a short exact sequence where $M'$ is principal special and $M''$ is special and built out of fewer principal specials than $M$ (e.g., $M''$ is the last piece in the filtration that $M$ is required to have). By induction, we can find a short exact sequence
\begin{displaymath}
0 \to K'' \to F'' \to M'' \to 0
\end{displaymath}
with $F''$ finite free and $K''$ special. Let
\begin{displaymath}
0 \to K' \to F' \to M' \to 0
\end{displaymath}
be a short exact sequence with $F'$ finite free. Then $K'$ is special by the previous lemma. Lift $F' \to M'$ to a map $F' \to M$. The map $F' \oplus F'' \to M$ is then surjective. If $K$ denotes its kernel then we have a short exact sequence
\begin{displaymath}
0 \to K'' \to K \to K' \to 0,
\end{displaymath}
and so $K$ is special.
\end{proof}

\begin{lemma} \label{sp3}
Let $M$ be a $\bD$-module with $0<\sd(M)<\infty$. Then we can find a surjection $f \colon F \to M$ with $F$ finite free such that $\sd(\ker(f))=\sd(M)-1$.
\end{lemma}

\begin{proof}
Let
\begin{displaymath}
0 \to P_r \to \cdots \to P_0 \to M \to 0
\end{displaymath}
be a resolution with $P_i$ special for all $i$. Let
\begin{displaymath}
0 \to Q \to F \to P_0 \to 0
\end{displaymath}
be a short exact sequence with $F$ finite free and $Q$ special; this exists by the previous lemma. Consider the short exact sequence
\begin{displaymath}
0 \to K \to F \to M \to 0.
\end{displaymath}
Let $Q'$ be the kernel of the surjection $F \oplus P_1 \to P_0$. We have a short exact sequence
\begin{displaymath}
0 \to Q \to Q' \to P_1 \to 0,
\end{displaymath}
which shows that $Q'$ is special. The projection map $F \oplus P_1 \to F$ induces a surjection $Q' \to K$, and one easily verifies that
\begin{displaymath}
0 \to P_r \to \cdots \to P_2 \to Q' \to K \to 0
\end{displaymath}
is a resolution of $K$. Thus $\sd(K) \le \sd(M)-1$. From the obvious inequality $\sd(M) \le 1+\sd(K)$, we conclude $\sd(K)=\sd(M)-1$.
\end{proof}

\begin{proposition} \label{sp4}
Let $M$ be a $\bD$-module with $\sd(M) \le r < \infty$. Then there is a resolution
\begin{displaymath}
0 \to P \to F_{r-1} \to \cdots \to F_0 \to M \to 0
\end{displaymath}
where $F_i$ is finite free and $P$ is special.
\end{proposition}

\begin{lemma} \label{sp5}
Let $D$ be a ring, let $Q_{\bullet} \to M$ and $P_{\bullet} \to N$ be resolutions of the same length $r$ such that $Q_0, \ldots, Q_{r-1}$ are free. Then any map $f \colon M \to N$ can be lifted to a map of resolutions $Q_{\bullet} \to P_{\bullet}$, that is, we can fill in the following diagram:
\begin{displaymath}
\xymatrix{
0 \ar[r] & Q_r \ar[r] \ar@{..>}[d] & Q_{r-1} \ar[r] \ar@{..>}[d] & \cdots \ar[r] & Q_0 \ar[r] \ar@{..>}[d] & M \ar[d]^f \ar[r] & 0 \\
0 \ar[r] & P_r \ar[r] & P_{r-1} \ar[r] & \cdots \ar[r] & P_0 \ar[r] & N \ar[r] & 0 }
\end{displaymath}
\end{lemma}

\begin{proof}
When $Q_r$ is also free, this is a standard result about free resolutions. The proof of that result applies here as well: the point is that one never needs to lift maps from $Q_r$.
\end{proof}

\begin{proposition} \label{sp6}
Suppose that
\begin{displaymath}
0 \to M_1 \to M_2 \to M_3 \to 0
\end{displaymath}
is a short exact sequence of $\bD$-modules. Then
\begin{enumerate}
\item $\sd(M_1) \le \max(\sd(M_2), \sd(M_3), 1)$.
\item $\sd(M_2) \le \max(\sd(M_1), \sd(M_3))$.
\item $\sd(M_3) \le \max(\sd(M_1), \sd(M_2))+1$.
\end{enumerate}
In particular, if two of the modules admit finite length special resolutions then so does the third.
\end{proposition}

\begin{proof}
(a) Let $r=\max(\sd(M_2), \sd(M_3), 1)$. First suppose $r \ge 2$. Let $P_{\bullet} \to M_2$ and $Q_{\bullet} \to M_3$ be special resolutions of length $r$ such that $P_0, \ldots, P_{r-1}$ and $Q_0, \ldots, Q_{r-1}$ are free. Lift the $M_2 \to M_3$ to a map of resolutions $P_{\bullet} \to Q_{\bullet}$. Then we have a resolution
\begin{displaymath}
0 \to P_r \to P_{r-1} \oplus Q_r \to \cdots \to P_1 \oplus Q_2 \to \ker(P_0 \oplus Q_1 \to Q_0) \to M_1 \to 0.
\end{displaymath}
Since $P_0 \oplus Q_1 \to Q_0$ is a surjection of free modules, its kernel is free, and so the above is a special resolution of $M_1$ of length at most $r$.

Now suppose $r=1$. Let
\begin{displaymath}
0 \to Q_1 \to Q_0 \to M_3 \to 0
\end{displaymath}
be a special resolution with $Q_0$ free. Let
\begin{displaymath}
0 \to F_1 \to F_0 \to M_2 \to 0
\end{displaymath}
be a special resolution with $F_0$ free. Lift $Q_0 \to M_3$ to a map $Q_0 \to M_2$ and lift $F_0 \to M_2 \to M_3$ to a map $F_0 \to Q_0$. Let $P_0=F_0 \oplus Q_0$, and let $P_1$ be the kernel of the natural surjection $P_0 \to M_2$. We have a short exact sequence
\begin{displaymath}
0 \to F_1 \to P_1 \to Q_0 \to 0,
\end{displaymath}
and so $P_1$ is special. The square
\begin{displaymath}
\xymatrix{
P_0 \ar[r] \ar[d] & Q_0 \ar[d] \\
M_2 \ar[r] & M_3 }
\end{displaymath}
commutes, and all maps are surjective. Thus $P_0 \to Q_0$ maps $P_1$ into $Q_1$. We thus have a resolution
\begin{displaymath}
0 \to P_1 \to \ker(P_0 \oplus Q_1 \to Q_0) \to M_1 \to 0.
\end{displaymath}
Since $P_0 \to Q_0$ is surjective, its kernel is free, and we have a short exact sequence
\begin{displaymath}
0 \to \ker(P_0 \to Q_0) \to \ker(P_0 \oplus Q_1 \to Q_0) \to Q_1 \to 0,
\end{displaymath}
and so $\ker(P_0 \oplus Q_1 \to Q_0)$ is special. Thus $\sd(M_1) \le 1$.

(b) Let $r=\max(\sd(M_1), \sd(M_3))$. If $r=0$ then $M_2$ is clearly special, so assume $r \ge 1$. Let $P_{\bullet} \to M_1$ and $Q_{\bullet} \to M_3$ be special resolutions of length $r$ such that $Q_0, \ldots, Q_{r-1}$ are free. Following the proof of the ``horseshoe lemma,'' we can build a partial resolution
\begin{displaymath}
0 \to K \to P_{r-1} \oplus Q_{r-1} \to \cdots \to P_0 \oplus Q_0 \to M_2 \to 0,
\end{displaymath}
where $K$ is by definition the kernel of the map $P_{r-1} \oplus Q_{r-1} \to P_{r-2} \oplus Q_{r-2}$ (or the map $P_{r-1} \oplus Q_{r-1} \to M_2$ if $r=1$). By the snake lemma, we have a short exact sequence
\begin{displaymath}
0 \to P_r \to K \to Q_r \to 0,
\end{displaymath}
and so $K$ is special. Thus the above is a special resolution of $M_2$, and so $\sd(M_2) \le r$.

(c) Let $r=\max(\sd(M_1), \sd(M_2))$, and let $P_{\bullet} \to M_1$ and $Q_{\bullet} \to M_2$ be special resolutions of length $r$ such that $P_0, \ldots, P_{r-1}$ are free. Lift the map $M_1 \to M_2$ to a map of resolutions $P_{\bullet} \to Q_{\bullet}$. Then we have a resolution
\begin{displaymath}
0 \to P_r \to P_{r-1} \oplus Q_r \to \cdots \to P_0 \oplus Q_1 \to Q_0 \to M_3 \to 0,
\end{displaymath}
and so $\sd(M_3) \le r+1$.
\end{proof}

\subsection{Special resolutions}

We now start on the proof of Theorem~\ref{thm:sd}.

\begin{lemma} 
Suppose $\bk$ is a field and infinitely many of the $\pi_n$ vanish. Then any finitely presented (graded) $\bD$-module is special.
\end{lemma}

\begin{proof}
Let $M$ be a finitely presented $\bD$-module. By Proposition~\ref{prop:h-free}, there exists $h$ and a finitely generated $\bD_{<h}$-module $N$ such that $M = N \otimes_{\bk} \bD^{(h)}$ as a module over $\bD = \bD_{<h} \otimes_{\bk} \bD^{(h)}$. Since $\bD_{<h}$ is a local artinian ring, $N$ admits a finite filtration such that each graded piece gets annihilated by the maximal ideal. Tensoring with the exact funtor $- \otimes_{\bk} \bD^{(h)}$, we see that $M$ has a filtration such that each graded piece is a free $\bD^{(h)}$-module and so $M$ is special.
\end{proof}

\begin{lemma}
\label{spfield}
Suppose $\bk$ is a field and $F$ is a finite free $\bD$-module. Then any finitely generated (graded) submodule of $F$ is special.
\end{lemma}
\begin{proof}
The case when infinitely many of the $\pi_n$ vanish follows from the previous lemma and so we may assume that only finitely many of the $\pi_n$ are zero. Let $h$ be the largest such that $\pi_h = 0$ and let $R = \bD_{<h}$. Then $\bD$ is isomorphic to $R \otimes_{\bk} \bD^{(h)} =R[y]$ where $y$ is an indeterminate of degree $h$. Let $I$ be a submodule of $F$. Since $R$ is a local artinian ring (say with maximal ideal $\fm$) there exists a finite filtration of $I$ such that the graded pieces are of the form $\tfrac{\fm^iI}{\fm^{i+1}I}$.  Clearly, each such graded piece is a torsion-free $\bk[y]$-module. Since $\bk[y]$ is a principal ideal domain the result follows from the structure theorem of finitely generated modules over a principal ideal domain.
\end{proof}

\begin{lemma}
Let $\bk$ be a domain and let $\bK=\Frac(\bk)$.
%\rohit{Fixed the mistype}.
Then every special $\bD \otimes_{\bk} \bK$-module $P$ has the form $Q \otimes_{\bk} \bK$ for some special $\bD$-module $Q$.
\end{lemma}

\begin{proof}
By a ``lattice'' in a $\bD \otimes_{\bk} \bK$ module $P$, we mean a $\bD$-submodule $Q$ of $P$ such that $Q \otimes_{\bk} \bK = P$. Thus the lemma states that every special $\bD \otimes_{\bk} \bK$-module admits a special lattice, that is, one that is special as a $\bD$-module.

First suppose $P$ is principal special. Then $P \cong (\bD \otimes_{\bk} \bK)^{(h)}$ for some $h$ with $\pi_h=0$. One can then take $Q=\bD^{(h)}$. Thus the lemma holds in this case.

Now suppose $P$ is an arbitrary special module. Choose an exact sequence
\begin{displaymath}
0 \to P' \to P \to P'' \to 0
\end{displaymath}
where $P''$ is principal special and $P'$ is special and built out of fewer principal specials than $P$. Let $Q'' \subset P''$ be a special lattice, which exists by the previous paragraph, and let $Q' \subset P'$ be a special lattice, which exists by induction. Let $\wt{Q}'' \subset P$ be a finitely generated $\bD$-submodule of $P$ that surjects onto $Q''$. (One can construct $\wt{Q}''$ by lifting generators of $Q''$ and taking the $\bD$-module they span.) Then $\wt{Q}'' \cap P'$ is a finitely generated $\bD$-submodule of $P'$, as it coincides with $\ker(\wt{Q}'' \to Q'')$ and $Q''$ is finitely presented. Since $Q'$ is a lattice in $P'$, it follows that $\wt{Q}'' \cap P'$ is contained in $\alpha Q'$ for some nonzero $\alpha \in \bK$. Let $Q=\alpha Q' + \wt{Q}''$. Then we have a short exact sequence
\begin{displaymath}
0 \to Q' \stackrel{\alpha}{\to} Q \to Q'' \to 0,
\end{displaymath}
and so $Q$ is special. It is clear that $Q$ is a lattice in $P$, and so the result follows.
\end{proof}

\begin{proof}[Proof of Theorem~\ref{thm:sd}]
We first show that $\sd(M)$ is finite. We proceed by noetherian induction, so we assume the result holds if $M$ has nonzero annihilator in $\bk$. Note that if $P$ is special as a $\bD/\fa \bD$-module then it is also special as a $\bD$-module, and so if $M$ is annihilated by $\fa \subset \bk$ then $\sd(M; \bD) \le \sd(M; \bD/\fa \bD)$. We now consider two cases.

{\it Case 1: $\bk$ is not a domain.} Let $xy=0$ with $x,y \ne 0$. We have an exact sequence
\begin{displaymath}
0 \to xM \to M \to M/xM \to 0,
\end{displaymath}
and so by Proposition~\ref{sp6}(b), we have
\begin{equation}
\label{eqsd1}
\sd(M) \le \max(\sd(xM), \sd(M/xM))
\end{equation}
Since $xM$ and $M/xM$ have non-zero annihilators, the right side is finite by the inductive hypothesis, and so $\sd(M)<\infty$.

{\it Case 2: $\bk$ is a domain.} Let $\bK=\Frac(\bk)$. Let
\begin{displaymath}
0 \to M' \to F \to M \to 0
\end{displaymath}
be a short exact sequence with $F$ finite free. Then $M' \otimes_{\bk} \bK \subset F \otimes_{\bk} \bK$, and so $M' \otimes_{\bk} \bK$ is special by the Lemma~\ref{spfield}. We can therefore find a special $\bD$-module $P$ and an isomorphism $M' \otimes_{\bk} \bK \to P \otimes_{\bk} \bK$. Scaling this isomorphism, we can assume that $M'$ maps into $P$. Let $N$ and $N'$ be the kernel and cokernel of the map $M' \to P$, so that we have a 4-term exact sequence
\begin{displaymath}
0 \to N \to M' \to P \to N' \to 0.
\end{displaymath}
Breaking this up into two short exact sequences and applying Proposition~\ref{sp6}, we find
\begin{displaymath}
\sd(M') \le \max(\sd(N), \sd(N'), 1),
\end{displaymath}
and so
\begin{equation}
\label{eqsd2}
\sd(M) \le \max(\sd(N), \sd(N'), 1)+1.
\end{equation}
But $N$ and $N'$ have nonzero annihilator in $\bk$, and so have finite special dimension by the inductive hypothesis. Thus $\sd(M)<\infty$.

Now suppose $\bk$ has finite Krull dimension $d$, and let us show $\sd(M) \le d+1$. We proceed again by noetherian induction. If $\bk$ is not a domain, then with notation as in Case~1, we have $\sd(xM) \le d+1$ and $\sd(M/xM) \le d+1$ by the inductive hypothesis, and so \eqref{eqsd1} gives $\sd(M) \le d+1$. Now suppose $\bk$ is a domain, and use notation as in Case~2. If $d=0$, i.e., $\bk$ is a field, then $M'$ is special and $\sd(M) \le 1$. Now suppose $d \ge 1$. Then the support of $N$ and $N'$ has dimension strictly less than $d$, and so $\sd(N) \le d$ and $\sd(N') \le d$ by the inductive hypothesis, and so \eqref{eqsd2} gives $\sd(M) \le d+1$.
\end{proof}

\section{Grothendieck groups} \label{s:groth}

\subsection{Notation}

For a coherent ring $D$, we let $\Mod_D^{\fpres}$ denote the category of finitely presented $D$-modules, and we let $\rK(D)$ be the Grothendieck group of $\Mod_D^{\fpres}$. If $D$ is graded then we let $\uMod_D^{\fpres}$ denote the category of finitely presented graded $D$-modules, and we let $\uK(D)$ be its Grothendieck group. We also apply these definitions to non-graded rings by regarding them as graded and concentrated in degree~0. The group $\uK(D)$ is naturally a module over the ring $R=\bZ[t,t^{-1}]$ via $t[M]=[M[1]]$. We define $R_n$ to be $\frac{1}{t^n-1} R$, thought of as an $R$-submodule of $\Frac(R)$, and put $R_{\infty}=\bigcup_{n \ge 1} R_n$.

\subsection{Overview}

Fix a GDPA $\bD$. The purpose of \S \ref{s:groth} is to study $\rK(\bD)$ and $\uK(\bD)$. The existence of special resolutions gives us a spanning set for both of these groups, and the main difficulty lies in understanding the relations these classes satisfy. For this, we need to construct interesting maps out of the Grothendieck group. We first concentrate on the graded case. In \S \ref{ss:groth-ratl} we construct the most obvious map out of $\uK(\bD)$, the Hilbert series $\rH$. It turns out that this is enough to obtain a description of $\uK(\bD)$ up to $R$-torsion. Based on the nature of the Hilbert series, in \S \ref{ss:groth-conj} we formulate a plausible description of $\uK(\bD)$ (Conjecture~\ref{conj:k}). The Hilbert series is not powerful enough to prove this conjecture, so in \S \ref{ss:grothL} we define a subtler invariant, denoted $\rL$. In \S \ref{ss:groth-proof}, we manage to prove Conjecture~\ref{conj:k} under a certain hypothesis using $\rH$ and $\rL$ (Theorem~\ref{thm:k}). In \S \ref{ss:groth-compare}, we show that $\rK(\bD)$ can be obtained from $\uK(\bD)$ in a straightforward manner, and thus Conjecture~\ref{conj:k} also predicts the structure of $\rK(\bD)$. In \S \ref{ss:groth-classical} we apply the results to the classical divided power algebra. We show that the hypothesis of Theorem~\ref{thm:k} is met, and thus deduce a complete description of the Grothendieck group. Finally, in \S\ref{ss:Kq}, we show that the hypothesis of Theorem~\ref{thm:k} is not met for the $q$-divided power algebra over $\bZ[q]$, and so our results do not give a full description of the Grothendieck group in this case.

\subsection{The classes $[M(\fa,h)]$}

We begin by recording a useful spanning set for the Grothendieck group and some obvious relations they satisfy. For an ideal $\fa$ of $\bk$ containing $\pi_h$, let $M(\fa,h)$ be the $\bD$-module $(\bD/\fa \bD)^{(h)}$. Note that any ideal contains $\pi_1=0$, and $M(\fa,1)=\bD/\fa \bD$.

\begin{proposition} \label{prop:mah}
Let $\fa$ be an ideal of $\bk$ containing $\pi_h$. Then:
\begin{enumerate}
\item The classes $[M(\fa,h)]$ span both $\rK(\bD)$ and $\uK(\bD)$.
\end{enumerate}
Now suppose $k \mid h$ and $\fa$ also contains $\pi_k$ (e.g., $k=1$). Then:
\begin{enumerate}
\setcounter{enumi}{1}
\item We have $\frac{h}{k} [M(\fa,h)]=[M(\fa,k)]$ in $\rK(\bD)$.
\item We have $\frac{1-t^h}{1-t^k} [M(\fa,h)]=[M(\fa,k)]$ in $\uK(\bD)$.
\end{enumerate}
\end{proposition}

\begin{proof}
(a) Follows immediately from the theorem on special resolutions (Theorem~\ref{thm:sd}). For (b) and (c), note that
\begin{displaymath}
M(\fa,k)=(\bD/\fa \bD)^{(k)}=(\bD/\fa \bD)_{<h}^{(k)} \otimes_{\bk} (\bD/\fa \bD)^{(h)}
\end{displaymath}
and $(\bD/\fa \bD)_{<h}^{(k)}$ admits a filtration where the successive quotients are $\bk[k], \ldots, \bk[(h-1)k]$.
\end{proof}

\subsection{Hilbert series} \label{ss:groth-ratl}

Let $M$ be a finitely presented graded $\bD$-module. Define the {\bf Hilbert series} of $M$, denoted $\rH_M(t)$, by
\begin{displaymath}
\rH_M(t) = \sum_{n \in \bZ} [M_n] t^n,
\end{displaymath}
where $[M_n]$ denotes the class of the $\bk$-module $M_n$ in $\rK(\bk)$. It is clear that the Hilbert series construction factors through the Grothendieck group and defines an $R$-linear map
\begin{displaymath}
\rH \colon \rK(\bD) \to \rK(\bk) \lpp t \rpp.
\end{displaymath}
Our main result about the Hilbert series is the following proposition:

\begin{proposition} \label{prop:hilb}
Let $M$ be a finitely presented graded $\bD$-module. Then there there exists $r \in \bN$ and elements $a_1(t), \ldots, a_r(t) \in R \otimes \rK(\bk)$ such that
\begin{displaymath}
\rH_M(t) = \sum_{i=1}^r \frac{a_i(t)}{1-t^i}.
\end{displaymath}
\end{proposition}

\begin{proof}
By Proposition~\ref{prop:mah}(a), it suffices to check this for the module $M=M(\fa, h)$. We have
\begin{displaymath}
M_n = \begin{cases}
\bk/\fa & \text{if $h \mid n$} \\
0 & \text{if $h \nmid n$} \end{cases}
\end{displaymath}
Thus $\rH_M(t)=\frac{[\bk/\fa]}{1-t^h}$, and the result follows.
\end{proof}

We now describe $\uK(\bD)$ up to $R$-torsion. Consider the map
\begin{displaymath}
\varphi_0 \colon \rK(\bk) \to \rK(\bD), \qquad [M] \mapsto (t-1) [\bD \otimes_{\bk} M]
\end{displaymath}
Let $S$ be the $R$-subalgebra of $\Frac(R)$ generated by $\frac{1}{t^n-1}$ for $n \ge 1$. We then have:

\begin{proposition} \label{prop:kgp}
The map $\varphi \colon \rK(\bk) \otimes_{\bZ} S \to \rK(\bD) \otimes_R S$ induced by $\varphi_0$ is an isomorphism.
\end{proposition}

\begin{proof}
Define $\psi \colon \rK(\bD) \otimes_R S \to \rK(\bk) \otimes_{\bZ} S$ by $\psi([M])=\rH_M(t)$. This map is well-defined by the previous proposition. We claim that $\psi$ is inverse to the map $\varphi$ in the statement of the theorem. We first check that $\psi \circ \varphi$ is the identity. Thus let $M$ be a finitely generated $\bk$-module. Then $\varphi([M])=(1-t) [M \otimes_{\bk} \bD]$. We have $\rH_{M \otimes_{\bk} D}(t)=\frac{[M]}{1-t}$, and so $\psi(\varphi([M]))=[M]$. We now check that $\varphi \circ \psi$ is the identity. It suffices to check $\varphi(\psi([M]))=[M]$ when $M=M(\fa,h)$, since these span $\rK(\bD)$. We have $\psi([M])=\frac{1}{1-t^h} [\bk/\fa]$, and so $\varphi(\psi([M]))=\frac{1-t}{1-t^h} [\bD/\fa \bD]$, which equals $[M]$ by Proposition~\ref{prop:mah}(c). This completes the proof.
\end{proof}

\subsection{Conjectural description of $\uK(\bD)$} \label{ss:groth-conj}

Suppose $M$ is a finitely presented graded $\bD$-module then, by Proposition~\ref{prop:hilb}, we can write
\begin{displaymath}
\rH_M(t) = \sum_{n \ge 1} \frac{a_n(t)}{1-t^n}
\end{displaymath}
where $a_n(t) \in R \otimes_{\bZ} \rK(\bk)$. If we express $[M]$ as an $R$-linear combination of classes of the form $[M(\fa,h)]$ then only those classes with $n \mid h$ can contribute to $a_n$ (assuming we do not artifically insert canceling factors into the numerator and denominator). That is, $a_n$ is an $R$-linear combination of classes of the form $[N]$, where $N$ belongs to the category $\cC_n$ of finitely generated $\bk$-modules supported on $\bigcup_{n \mid h} V(\pi_h)$. This suggests that $a_n$ might be well-defined in $R \otimes \rK_0(\cC_n)$. We conjecture that this is the case, and that, moreover, it completely explains the structure of $\uK(\bD)$. We now give a precise statement.

For notational ease, let $\cK_n=\rK_0(\cC_n)$. For $M \in \cC_n$ we write $[M]_n$ for its class in $\cK_n$. Note that $\cC_1$ is the category of all finitely generated $\bk$-modules, and for $n \mid m$ we have an inclusion $\cC_m \subset \cC_n$. We define $\ul{\sK}$ to be the quotient of the $R$-module
\begin{displaymath}
\bigoplus_{n \ge 1} \cK_n \otimes_{\bZ} R_n
\end{displaymath}
by the relations
\begin{displaymath}
\frac{[M]_m}{1-t^n} = \frac{[M]_n}{1-t^n}
\end{displaymath}
for $M \in \cC_m$ and $n \mid m$, where here the left side belongs to the $m$th summand and the right side to the $n$th summand. (Note that $\frac{1}{1-t^n} \in R_m$, so that the left side above does indeed belong to the $m$th summand.)

We now define an $R$-linear map
\begin{equation} \label{eq:phi}
\varphi \colon \ul{\sK} \to \uK(\bD).
\end{equation}
Suppose $n \mid h$ and $\fa$ is an ideal containing $\pi_h$. We then put
\begin{displaymath}
\varphi \left( \frac{[\bk/\fa]_n}{1-t^n} \right) = \frac{1-t^h}{1-t^n} [M(\fa,h)].
\end{displaymath}
Since $\cK_n$ is spanned by the classes $[\bk/\fa]$ as above, this specifies $\varphi$ uniquely. We leave it to the reader to verify that $\varphi$ is well-defined. It follows from Proposition~\ref{prop:mah}(a) that $\varphi$ is surjective. Our conjectural description of $\uK(\bD)$ is:

\begin{conjecture} \label{conj:k}
The map $\varphi$ is an isomorphism.
\end{conjecture}

\subsection{The $\rL$ invariant} \label{ss:grothL}

Let $\cC_+$ be the category of finitely generated $\bk$-modules supported on $\bigcup_{n \ge 2} V(\pi_n)$, and let $\cK_+=\rK_0(\cC_+)$. For $M \in \cC_+$, let $[M]_+$ be the class of $M$ in $\cK_+$. We extend this notation to all $\bk$-modules $M$ by putting $[M]_+=0$ if $M \not\in \cC_+$. For a graded $\bk$-module $M$, we define $\lbb M \rbb$ as $\sum_{n \in \bZ} [M_n]_+ t^n$. For a finitely presented graded $\bD$-module $M$, we define
\begin{displaymath}
\rL^0_M(t) = \sum_{i=0}^{\infty} (-1)^i [\Tor^{\bD}_i(M, \bk)]_+ \in \cK_+ \lpp t \rpp.
\end{displaymath}
For any fixed $d$ we have $\Tor^{\bD}_i(M, \bk)_d=0$ for $i \gg 0$ (Proposition~\ref{prop:tor-vanishing}). It follows that in the above sum, any fixed power of $t$ occurs only finitely many times, and so the sum is well-defined. We also put
\begin{displaymath}
\rL_M(t) = \frac{\rL^0_M(t)}{1-t}.
\end{displaymath}
Our main result on $\rL$ is the following:

\begin{proposition} \label{prop:L}
We have a well-defined $R$-module homomorphism
\begin{displaymath}
\rK(\bD) \to R_{\infty}/R_1 \otimes_{\bZ} \cK_+, \qquad [M] \mapsto \rL_M(t)
\end{displaymath}
\end{proposition}

We need some lemmas before proving this. In what follows, all $\bD$-modules are finitely presented and graded.

\begin{lemma} \label{lem:L1}
Let $M$ be a $\bD$-module. Then $\Tor_i^{\bD}(M, \bk)$ belongs to $\cC_+$ for $i \gg 0$.
\end{lemma}

\begin{proof}
If $M=\bD/\fa \bD$ then the $\Tor$ in question vanishes for $i>0$ (Proposition~\ref{prop:relatively-projective}). If $M=M(\fa,h)$ for some $h \ge 2$ then all the $\Tor$'s are annihilated by $\pi_h$. The result now follows from the theorem on special resolutions (Theorem~\ref{thm:sd}).
\end{proof}

\begin{lemma} \label{lem:L2}
Suppose that
\begin{displaymath}
0 \to M_1 \to M_2 \to M_3 \to 0
\end{displaymath}
is a short exact sequence of $\bD$-modules. Then $\rL_{M_1}(t)+\rL_{M_3}(t)=\rL_{M_2}(t)+\delta$ for some $\delta \in R_1 \otimes \cK_+$. If each $M_i$ supported on $\bigcup_{n \ge 2} V(\pi_n)$ as a $\bk$-module then $\delta=0$.
\end{lemma}

\begin{proof}
Let $N$ be such that $\Tor^{\bD}_i(M_j, \bk)$ belongs to $\cC_+$ for all $i \ge N$ and all $j \in \{1,2,3\}$. This exists by Lemma~\ref{lem:L1}. Let $d_0$ be the maximal degree occurring in the groups $\Tor^{\bD}_i(M_j, \bk)$ for $0 \le i < N$ and $j \in \{1,2,3\}$. Consider the degree $d>d_0$ piece of the long exact sequence in $\Tor$:
\begin{displaymath}
\cdots \to \Tor^{\bD}_{i+1}(M_3, \bk)_d \to \Tor^{\bD}_i(M_1, \bk)_d \to \Tor^{\bD}_i(M_2, \bk)_d \to \Tor^{\bD}_i(M_3, \bk)_d \to \cdots
\end{displaymath}
This is an exact sequence in the category $\cC_+$, and so the alternating sum of the classes in $\cK_+$ is zero. It follows that the coefficient of $t^d$ in $\delta^0 = \rL^0_{M_1}-\rL^0_{M_2}+\rL^0_{M_3}$ vanishes for all $d>d_0$, and so $\delta = \frac{\delta^0}{t-1} \in R_1$. If each $M_i$ is supported on $\bigcup_{n \ge 2} V(\pi_n)$ then all the $\Tor$'s in question belong to $\cC_+$, so the above reasoning applies to all $d \in \bZ$, and so $\delta=0$.
\end{proof}

\begin{lemma} \label{lem:L3}
Let $M=M(\fa,h)$ with $h \ge 1$. Then
\begin{displaymath}
\rL_M(t)=\frac{[\bk/\fa]_+}{1-t^h}.
\end{displaymath}
\end{lemma}

\begin{proof}
We have $\frac{1-t^h}{1-t} [M] = [\bD/\fa \bD]$ in $\rK(\bD/\fa \bD)$ by Proposition~\ref{prop:mah}(c), and so $\frac{1-t^h}{1-t} \rL_M(t) = \rL_{\bD/\fa \bD}(t)$ (by the previous lemma if $h \ge 2$, and is trivially true for $h=1$). By Proposition~\ref{prop:relatively-projective}, the group $\Tor^{\bD}_i(\bD/\fa\bD, \bk)$ vanishes for $i>0$ and equals $\bk/\fa$ for $i=0$, and so $\rL^0_{\bD/\fa\bD}(t)=[\bk/\fa]_+$. The result follows.
\end{proof}

\begin{proof}[Proof of Proposition~\ref{prop:L}]
By Lemma~\ref{lem:L2}, $\rL$ gives a well-defined homomorphism $\rK(\bD) \to \bZ\lpp t \rpp/R_1 \otimes \cK_+$. Proposition~\ref{prop:mah}(a) and Lemma~\ref{lem:L3} show that the image is contained in $R_{\infty}/R_1 \otimes \cK_+$.
\end{proof}

\subsection{Conditional proof of Conjecture~\ref{conj:k}} \label{ss:groth-proof}

For $n \ge 2$, we have an inclusion $\cC_n \subset \cC_+$, and thus an induced homomorphism $\cK_n \to \cK_+$.

\begin{theorem} \label{thm:k}
Assume the following condition holds:
\begin{itemize}
\item[$(\ast)$] For all $n \ge 2$ the map $\cK_n \to \cK_+$ is injective.
\end{itemize}
Then Conjecture~\ref{conj:k} is true.
\end{theorem}

\begin{proof}
Let $\ul{\sK}^1 \subset \ul{\sK}$ be the $R$-span of classes of the form $\frac{[M]}{t-1}$ with $M$ a finitely generated $\bk$-module. Then $\ul{\sK}^1$ is identified with $\rK(\bk) \otimes R_1$. Let $\ul{\sK}^2$ be the quotient of $\ul{\sK}$ by $\ul{\sK}^1$. Then $\ul{\sK}^2$ is identified with the submodule of $R_{\infty}/R_1 \otimes \cK_+$ consisting of elements of the form $\sum_{n \ge 2} \frac{a_n(t)}{t^n-1}$ where $a(n) \in \cK_n \otimes R$ (this identification uses $(\ast)$).

We first claim that the map
\begin{displaymath}
\ul{\sK}^2 \to R_{\infty}/R_1 \otimes \cK_+, \qquad f \mapsto \rL_{\varphi(f)}(t)
\end{displaymath}
is well-defined and simply the identity. If $f \in \ul{\sK}^1$ then $\rL_{\varphi(f)}(t)=0$ by the definition of $\rL$. Thus the above map is well-defined. The module $\ul{\sK}^2$ is generated by elements of the form $f=\frac{[\bk/\fa]_+}{t^n-1}$ where $\fa$ is an ideal containing $\pi_h$ and $n \mid h$. By definition, we have $\varphi(f)=\frac{1-t^h}{1-t^n} [M(\fa,h)]$. Thus, by Lemma~\ref{lem:L3}, we have
\begin{displaymath}
\rL_{\varphi(f)}(t)=\frac{1-t^h}{1-t^n} \cdot \rL_{M(\fa,h)}(t) = \frac{1-t^h}{1-t^n} \cdot \frac{[\bk/\fa]_+}{1-t^h} = f,
\end{displaymath}
and the claim is proved.

We next claim that the map
\begin{displaymath}
\ul{\sK}^1 \to R_1 \otimes \rK(\bk), \qquad f \mapsto \rH_{\varphi(f)}(t)
\end{displaymath}
is the identity. The module $\ul{\sK}^1$ is spanned by elements of the form $f=\frac{[\bk/\fa]}{t-1}$. We have $\varphi(f)=[\bD/\fa \bD]$ by definition. Thus
\begin{displaymath}
\rH_{\varphi(f)}(t) = \rH_{\bD/\fa \bD}(t) = \frac{[\bk/\fa]}{t-1} = f,
\end{displaymath}
and the claim is proved.

We now prove the theorem. It suffices to show that $\varphi$ is injective. Thus suppose $\varphi(x)=0$. Let $\ol{x}$ be the image of $x$ in $\ul{\sK}^2$. Then $\ol{x}=\rL_{\varphi(x)}(t)=0$ by the first paragraph. Thus $x \in \ul{\sK}^1$. But then $x=\rH_{\varphi(x)}(t)=0$ by the second paragraph. This completes the proof.
\end{proof}

\subsection{Comparison of $\uK(\bD)$ and $\rK(\bD)$} \label{ss:groth-compare}

The forgetful functor $\uMod_{\bD}^{\fpres} \to \Mod_{\bD}^{\fpres}$ induces a map $\uK(\bD) \to \rK(\bD)$ that obvious kills $(t-1) \uK(\bD)$. In fact:

\begin{proposition} \label{prop:Kgrade}
The natural map $\uK(\bD)/(t-1)\uK(\bD) \to \rK(\bD)$ is an isomorphism.
\end{proposition}

The proof closely follows the proof of the so-called fundamental theorem of K-theory, as presented in \cite[\S 5]{srinivas}. We will make extensive use of the ring $\bD[u]$, which we grade using the usual grading on $\bD$ and setting $\deg(u)=1$. Note that if $\bD=\bD(\bk,\pi_{\bullet})$ then (ignoring the grading) $\bD[u]=\bD(\bk[u], \pi_{\bullet})$, and so $\bD[u]$ is coherent since $\bk[u]$ is noetherian.

\begin{lemma} \label{lem:Kgr1}
We have an equivalence of categories $\uMod_{\bD[u,u^{-1}]}^{\fpres} = \Mod_{\bD}^{\fpres}$.
\end{lemma}

\begin{proof}
Let $\varphi \colon \bD \to \bD[u,u^{-1}]_0$ be the map given by $\varphi(x)=u^{-\deg(x)} x$ on homogeneous elements $x \in \bD$. Then $\varphi$ is an isomorphism. If $M$ is a graded $\bD[u,u^{-1}]$-module then its degree~0 piece is a module over $\bD[u,u^{-1]}]_0$ and thus, via $\phi$, over $\bD$. Conversely, if $N$ is a $\bD[u,u^{-1}]_0$-module, then $\bigoplus_{k \in \bZ} N u^k$ is a graded $\bD[u,u^{-1}]$-module, with $u$ acting in the obvious manner, and $x \in \bD$ acting by $(u^{-\deg(x)} x) u^{\deg(x)}$. One easily sees that these constructions are inverse to each other.
\end{proof}

Let $\alpha \colon \uK(\bD) \to \uK(\bD[u])$ be the map induced by the functor $M \mapsto M \otimes_{\bD} \bD[u]$, and let $\beta \colon \uK(\bD[u]) \to \uK(\bD)$ be the map induced by the functor $M \mapsto M \stackrel{\rL}{\otimes}_{\bD[u]} \bD$, where here $\bD$ is thought of as a $\bD[u]$-module by $\bD=\bD[u]/(u)$. We note that $\bD$ has projective dimension~1 as a $\bD[u]$-module, and so the left-derived functor of $- \otimes_{\bD[u]} \bD$ does indeed induce a map on K-theory.

\begin{lemma} \label{lem:Kgr2}
The maps $\alpha$ and $\beta$ are mutually inverse.
\end{lemma}

\begin{proof}
It is clear that $\beta\alpha = \id$, so we must show $\alpha\beta = \id$. For $h \ge 1$, consider the following diagram
\begin{displaymath}
\xymatrix{
\uK(\bD) \ar[r]^{\alpha} & \uK(\bD[u]) \\
\uK(\bk/(\pi_h)) \ar[u]^{f_h} \ar[r]^{i_h} & \uK(\bk/(\pi_h)[u]) \ar[u]_{f'_h} }
\end{displaymath}
The map $f_h$ is induced by the functor $M \mapsto (M \otimes_{\bk} \bD)^{(h)}$, and $f'_h$ is defined similarly. The map $i_h$ is extension of scalars. The theorem on special resolutions (Theorem~\ref{thm:sd}) states that the maps $f'_h$ are jointly surjective, that is, the sum of their images is the entire $\uK(\bD[u])$. The map $i_h$ is an isomorphism by the following lemma. It follows from this that $\alpha$ is surjective. Since $\beta\alpha=\id$, we have $\alpha\beta\alpha=\alpha$, and thus $\alpha\beta=\id$ since $\alpha$ is surjective.
\end{proof}

\begin{lemma} \label{lem:Krg3}
Let $\bk$ be a non-graded ring, and regard $\bk[u]$ as graded by $\deg(u)=1$. Then extension of scalars induces an isomorphism $\uK(\bk) \to \uK(\bk[u])$.
\end{lemma}

\begin{proof}
The map $\rK(\bk) \otimes_{\bZ} R \to \uK(\bk[u])$ induced by extension of scalars is an isomorphism by \cite[Proposition~5.4]{srinivas}. However, we have an obvious identification $\rK(\bk) \otimes_{\bZ} R=\uK(\bk)$ since $\bk$ is non-graded.
\end{proof}

\begin{proof}[Proof of Proposition~\ref{prop:Kgrade}]
Let $\uMod_{\bD[u]}^0$ denote the category of finitely presented graded $\bD[u]$-modules annihilated by a power of $u$, and let $\uK^0(\bD[u])$ be its Grothendieck group. We then have the localization sequence
\begin{displaymath}
\uK^0(\bD[u]) \to \uK(\bD[u]) \to \uK(\bD[u,u^{-1}]) \to 0.
\end{displaymath}
By d\'evissage, $\uK^0(\bD[u])=\uK(\bD)$. Combining this and Lemmas~\ref{lem:Kgr1} and~\ref{lem:Kgr2} with the above sequence gives the diagram
\begin{displaymath}
\xymatrix{
\uK^0(\bD[u]) \ar[r] & \uK(\bD[u]) \ar[r] & \uK(\bD[u,u^{-1}]) \ar[r] & 0 \\
\uK(\bD) \ar@{=}[u] \ar[r] & \uK(\bD) \ar@{=}[u] \ar[r] & \rK(\bD) \ar@{=}[u] \ar[r] & 0 }
\end{displaymath}
We claim that the bottom left map is multiplication by $1-t$, which will complete the proof. Thus let $M \in \uMod_{\bD}$. Starting with $[M]$ in the bottom left group and going up and right, we obtain the class in $\uK(\bD[u])$ obtained by treating $M$ as a $\bD[u]$-module with $u$ acting by~0. We now want to move this class down under $\beta$. By definition, $\beta([M])$ is the class of $M \stackrel{\rL}{\otimes}_{\bD[u]} \bD$. Using the resolution $\bD[u][1] \stackrel{u}{\to} \bD[u]$ of $\bD$ (where $[1]$ indicates shift in grading), we see that $\beta([M])$ is the class of the complex $M[1] \to M$, i.e., $(1-t)[M]$.
\end{proof}

Combining the above result with Proposition~\ref{prop:kgp} yields a rational description of $\rK(\bD)$:

\begin{proposition}
The map $\rK(\bk) \otimes \bQ \to \rK(\bD) \otimes \bQ$ taking $[M]$ to $[M \otimes_{\bk} \bD]$ is an isomorphism.
\end{proposition}

We also give a conjecutral description of $\rK(\bD)$. Define $\sK$ to be the quotient of $\bigoplus_{n \ge 1} \cK_n$ by the relations $\frac{m}{n} [M]_m = [M]_n$ for $M \in \cC_m$ and $n \mid m$, where the left side belongs to the $m$th summand and the right to the $n$th summand. Then $\ul{\sK}/(t-1) \sK$ is identified with $\sK$, and so the map \eqref{eq:phi} induces a natural map $\sK \to \rK(\bD)$. Thus Proposition~\ref{prop:Kgrade} gives us:

\begin{proposition}
Suppose Conjecture~\ref{conj:k} holds for $\bD$. Then the natural map $\sK \to \rK(\bD)$ is an isomorphism.
\end{proposition} 

\subsection{The classical divided power algebra} \label{ss:groth-classical}

We now assume that $\bD$ is the classical divided power algebra over $\bk$.

\begin{proposition} \label{prop:classast}
The condition $(\ast)$ of Theorem~\ref{thm:k} holds.
\end{proposition}

\begin{proof}
By definition, $\cC_+$ is the category of finitely generated $\bk$-modules that are torsion as abelian groups. If $n \ge 2$ is a power of the prime $p$ then $\cC_n$ is the category of finitely generated $\bk$-modules that are annihilated by a power of $p$. By the Chinese remainder theorem, $\cC_n$ is a summand of $\cC_+$, and so $\cK_n \to \cK_+$ is injective. If $n \ge 2$ is not a prime power then $\cC_n=0$.
\end{proof}

For a prime number $p$, let $R_{p^{\infty}}=\bigcup_{n \ge 1} R_{p^n}$.

\begin{corollary} \label{cor:groth-class}
We have short exact sequences
\begin{displaymath}
0 \to R_1 \otimes \rK(\bk) \to \uK(\bD) \to \bigoplus_p R_{p^{\infty}}/R_1 \otimes \rK(\bk/p \bk) \to 0
\end{displaymath}
and
\begin{displaymath}
0 \to \rK(\bk) \to \rK(\bD) \to \bigoplus_p \bQ_p/\bZ_p \otimes \rK(\bk/p \bk) \to 0,
\end{displaymath}
where in both lines the sum is over all prime numbers $p$.
\end{corollary}

\begin{proof}
These sequences simply come from computing $\ul{\sK}$ and $\sK$. In the notation of the proof of Theorem~\ref{thm:k}, the first short exact sequence is simply
\begin{displaymath}
0 \to \ul{\sK}^1 \to \uK(\bD) \to \ul{\sK}^2 \to 0. \qedhere
\end{displaymath}
\end{proof}

\begin{example} \label{ex:ktors}
Let $\bk=\bZ_p$, let $\fp$ be the maximal ideal of $\bk$, and let $h>1$ be a power of $p$. Note that $\pi_h \in \fp$. We have $[\bD/\fp \bD]=[\bD]-[\bD]=0$, and so $[M(\fp, h)]$ is killed by $\frac{1-t^h}{1-t}$ by Proposition~\ref{prop:mah}(c). However, we have $\rL_{M(\fp,h)}(t)=\frac{[\bF_p]_+}{1-t^h}$ and $[\bF_p]_+$ is not zero in $\cK_+$; thus $\rL_{M(\fp,h)}(t)$ is non-zero in $R_{\infty}/R_1 \otimes \cK_+$, and so $[M(\fp,h)]$ is non-zero in $\rK(\bD)$. We therefore have an example of a non-zero class in $\rK(\bD)$ that is killed by a non-zero element of $R$.
\end{example}

\begin{remark} \label{rmk:dercat}
Let $\cD$ be the bounded derived category of finitely presented non-graded $\bD$-modules. Let $\cF$ be the full subcategory on objects that can be represented by bounded complexes whose terms have the form $M \otimes_{\bk} \bD$ with $M$ a finitely presented $\bk$-module. Let $\cD^t$ be the full subcategory of $\cD$ on objects 
%\rohit{Replaced ``on objects" twice in this paragraph by ``of objects". That's what the referee suggested.}
represented by bounded complexes of finitely presented $\bD$-modules that are torsion as $\bZ$-modules. Then $\rK(\cF) = \rK(\bk)$ and
\begin{displaymath}
\rK(\cD^t/(\cF \cap \cD^t)) = \bigoplus_p \bQ_p/\bZ_p \otimes \rK(\bk/p \bk).
\end{displaymath}
Thus Corollary~\ref{cor:groth-class} suggests an equivalence
\begin{displaymath}
\frac{\cD}{\cF} \cong \frac{\cD^t}{\cD^t \cap \cF}.
\end{displaymath}
Such an equivalence also seems plausible due to the form of special resolutions. (This remark is not specific to the classical divided power algebra, and should apply to any GDPA.)
\end{remark}

\subsection{The $q$-divided power algebra} \label{ss:Kq}

Let $\bD$ be the $q$-divided power algebra
%\rohit{Fixed the mistype}
over $\bZ[q]$. We now study the groups $\cK_n$ and $\cK_+$, and ultimately show that condition $(\ast)$ of Theorem~\ref{thm:k} does not hold. For an integer $n \ge 1$, we write $\bZ[\zeta_n]$ for the quotient $\bZ[q]/(\Phi_n)$, where $\zeta_n$ is the image of $q$, a primitive $n$th root of unity. The group $\rK(\bZ[\zeta_n])$ canonically decomposes as $\bZ \oplus \Cl(\bQ(\zeta_n))$, where $\bZ$ is generated by the class of $\bZ[\zeta_n]$ itself, and $\Cl(\bQ(\zeta_n))$ is the class group of the number field $\bQ(\zeta_n)$, which is finite. Now, it is a general fact that if $\fa$ and $\fb$ are ideals in a noetherian ring $\bk$ then we have an exact sequence
\begin{displaymath}
\rK(\bk/(\fa+\fb)) \to \rK(\bk/\fa) \oplus \rK(\bk/\fb) \to \rK(\bk/\fa \fb) \to 0.
\end{displaymath}
Taking $\fa=(\Phi_n)$ and $\fb=(\Phi_m)$ yields
\begin{displaymath}
\rK(\bZ[q]/(\Phi_n,\Phi_m)) \to \rK(\bZ[\zeta_n]) \oplus \rK(\bZ[\zeta_m]) \to \rK(\bZ[q]/(\Phi_n \Phi_m)) \to 0.
\end{displaymath}
Applying the above sequence repeatedly and taking a direct limit yields a presentation
\begin{displaymath}
\bigoplus_{t<s, n \mid t,s} \rK(\bZ[q]/(\Phi_t, \Phi_s)) \stackrel{\alpha}{\to} \bigoplus_{n \mid t} \rK(\bZ[\zeta_t]) \to \cK_n \to 0.
\end{displaymath}
Now, observe that $(\Phi_t,\Phi_s)$ is a height~2 prime of $\bZ[q]$, and so the image of $\alpha$ lands in the sum of the class groups. In particular, after tensoring with $\bQ$ we see that $\alpha$ is 0. Thus $\cK_n \otimes \bQ$ has a natural basis indexed by the multiples of $n$. A similar analysis holds for $\cK_+ \otimes \bQ$: it has a natural basis indexed by integers $\ge 2$. We thus find:

\begin{proposition} \label{prop:qdprat}
The condition $(\ast)$ of Theorem~\ref{thm:k} holds rationally. Thus the map $\ul{\sK} \otimes \bQ \to \uK(\bD) \otimes \bQ$ induced from the map $\varphi$ in Conjecture~\ref{conj:k} is an isomorphism.
\end{proposition}

For $n \mid t$, let $\rN_t \colon \rK(\bZ[\zeta_t]) \to \rK(\bZ[\zeta_n])$ be the norm map. On the $\bZ$ summand, this is multiplication by the degree $[\bQ(\zeta_t) \colon \bQ(\zeta_n)]$, while on the class group summand it is the usual norm map on ideals. Summing these maps yields a map
\begin{displaymath}
\wt{\rN} \colon \bigoplus_{n \mid t} \rK(\bZ[\zeta_t]) \to \rK(\bZ[\zeta_n]).
\end{displaymath}
We claim that $\wt{\rN}$ kills the image of $\alpha$. Thus let $t<s$ be multiples of $n$. The ideal $(\Phi_t,\Phi_s)$ is the unit ideal unless $s=tp^r$ for some prime $p$, so assume this is the case. Let $\fm$ be a maximal ideal of $\bZ[q]$ containing $(\Phi_t,\Phi_s)$, and let $\fm_t$ and $\fm_s$ be the images of $\fm$ in $\bZ[\zeta_t]$ and $\bZ[\zeta_s]$. Then $\fm_t$ is totally ramified in $\bZ[\zeta_s]$, and $\fm_s$ is the unique prime above it. Thus the norm of $\fm_s$ down to $\bZ[\zeta_t]$ is $\fm_t$. By the compatibility of norms in towers, we see that $\fm_t$ and $\fm_s$ have the same norm down to $\bZ[\zeta_n]$. This establishes the claim. We thus see that $\wt{\rN}$ induces a map
\begin{displaymath}
\rN \colon \cK_n \to \rK(\bZ[\zeta_n]).
\end{displaymath}
It follows immediately from the definition that $\rN$ is left-inverse to the canonical map $\rK(\bZ[\zeta_n]) \to \cK_n$, and so we find:

\begin{proposition} \label{prop:qdpinj}
The canonical map $\rK(\bZ[\zeta_n]) \to \cK_n$ is a split injection.
\end{proposition}

We now come to the main point:

\begin{proposition}
The condition $(\ast)$ of Theorem~\ref{thm:k} does not hold. In fact, the map $\cK_{39} \to \cK_+$ is not injective.
\end{proposition}

\begin{proof}
Fix a primitive cube root of unity $\eta \in \bF_{13}$, and give $M=\bF_{13}$ the structure of a $\bZ[q]$-module by letting $q$ act by $\eta$. The module $M$ is killed by $\Phi_3$ and $\Phi_{39}$, and can thus be regarded as a module over $\bZ[\zeta_3]$ and $\bZ[\zeta_{39}]$. Let $\fp$ be the prime of $\bZ[\zeta_3]$ over~13 corresponding to $\eta$, i.e., $\fp=(13, \zeta_3-\eta)$, and let $\fq$ be the unique prime of $\bZ[\zeta_{39}]$ over $\fp$. Then $M \cong \bZ[\zeta_3]/\fp$ as a $\bZ[\zeta_3]$-module and $M \cong \bZ[\zeta_{39}]/\fq$ as a $\bZ[\zeta_{39}]$-module. Now, the ideal $\fp$ is principal since $\bQ(\zeta_3)$ has class number~1, and so the class of $M$ in $\rK(\bZ[\zeta_3])$ is zero. It follows that the class of $M$ vanishes in $\cK_3$, and thus in $\cK_+$ as well. On the other hand, $\fq$ is non-principal (see \cite{mathoverflow}). Thus the class of $M$ in $\rK(\bZ[\zeta_{39}])$ is non-zero, and therefore the class of $M$ in $\cK_{39}$ is non-zero by Proposition~\ref{prop:qdpinj}. This class is therefore a non-zero element of the kernel of $\cK_{39} \to \cK_+$.
\end{proof}

\section{Injective dimension and duality} \label{s:inj}

\subsection{Self-injectivity}

Throughout this section, $\bk$ denotes a field and $\bD$ a GDPA over $\bk$. Our goal is to prove the following theorem:

\begin{theorem} \label{thm:selfinj}
The injective dimension of $\bD$ as a $\bD$-module is at most~1, in both the graded and ungraded cases. More precisely:
\begin{enumerate}
\item If infinitely many of the $\pi_n$ vanish then $\bD$ is injective as a graded $\bD$-module.
\item If infinitely many of the $\pi_n$ vanish then $\bD$ has injective dimension~1 in the category of all $\bD$-modules, and $\Ext^i_{\bD}(M, \bD)=0$ for all finitely presented $M$ and all $i>0$.
\item If only finitely many of the $\pi_n$ vanish then $\bD$ has injective dimension~1 in both the graded and ungraded categories.
\end{enumerate}
\end{theorem}

\begin{lemma} \label{lem:dualfree}
(In this lemma, $\bk$ can be any ring.) Suppose $\pi_h=0$. Then $\Hom_{\bk}(\bD_{<h}, \bk)$ is free of rank one as a $\bD_{<h}$-module.
\end{lemma}

\begin{proof}
As a $\bk$-module, $\bD_{<h}$ is free with basis $x^{[0]}, \ldots, x^{[h-1]}$. Thus $\Hom_{\bk}(\bD_{<h}, \bk)$ is also free. Let $\lambda_0, \ldots, \lambda_{h-1}$ be the dual basis, so that $\lambda_i(x^{[j]})=\delta_{i,j}$. We have $x^{[i]} \lambda_{h-1}=u \lambda_{h-1-i}$, where $u \in \bk$ is given by $x^{[i]} x^{[h-1-i]} = u x^{[h-1]}$. We claim that $u$ is a unit. Indeed, if $\fm$ is any maximal ideal of $\bk$ then $h=b_{\fm,i}$ for some $i$, and so there are no carries in the sum $i+(h-1-i)$ when working in base $b_{\fm,\bullet}$. It follows that $u$ is a unit in $\bk_{\fm}$ for all $\fm$, and thus a unit. This proves the claim. It now follows that $\lambda_{h-1}$ is a basis for $\Hom_{\bk}(\bD_{<h}, \bk)$ as a $\bD_{<h}$-module.
\end{proof}

\begin{lemma} \label{lem:truncinj}
Suppose $\pi_n=0$. Then $\bD_{<n}$ is injective as a module over itself, both in the graded and non-graded settings.
\end{lemma}

\begin{proof}
Let $D=\bD_{<n}$, and let $M$ be an arbitrary $D$-module. We have
\begin{displaymath}
\Hom_D(M, D) \cong \Hom_D(M, \Hom_{\bk}(D, \bk)) = \Hom_{\bk}(M, \bk).
\end{displaymath}
The first isomorphism comes from Lemma~\ref{lem:dualfree}, while the second is a standard adjunction. Since the functor $\Hom_{\bk}(-, \bk)$ is exact, it follows that $D$ is injective as an $D$-module. The same proof works in the graded case.
\end{proof}

\begin{proof}[Proof of Theorem~\ref{thm:selfinj}(b)]
We first prove the $\Ext$ vanishing statement. By dimension shifting, it suffices to treat the $i=1$ case. By d\'evissage, it suffices to treat the case where $M=\bD/I$ with $I$ a finitely generated ideal. For this, it is enough to show that a given map $\varphi \colon I \to \bD$ extends to a map $\psi \colon \bD \to \bD$. Let $f_1, \ldots, f_r$ be generators for $I$. Pick $n \gg 0$ such that $\pi_n=0$ and the elements $f_i$ and $\varphi(f_i)$ all belong to $\bD_{<n}$. Let $I_0 \subset \bD_{<n}$ be the ideal generated by the $f$'s, and let $\varphi_0 \colon I_0 \to \bD_{<n}$ be the restriction of $\varphi$. We have an isomorphism $\bD=\bD_{<n} \otimes_{\bk} \bD^{(n)}$, under which $I$ corresponds to $I_0 \otimes \bD^{(n)}$ and $\varphi$ corresponds to $\varphi_0 \otimes \id$. Since $\bD_{<n}$ is injective as a module over itself by Lemma~\ref{lem:truncinj}, the map $\varphi_0$ extends to a map $\psi_0 \colon \bD_{<n} \to \bD_{<n}$. Then $\psi=\psi_0 \otimes \id$ is an extension of $\varphi$.

We now show that the injective dimension of $\bD$ is at most one. This follows from the previous paragraph by general considerations, as follows. First, suppose that $M=\bD/I$ is a cyclic $\bD$-module. Since $\bD$ has countable dimension over $\bk$, the ideal $I$ is countably generated. Write $I=\bigcup_{n \ge 1} I_n$ with $I_n$ finitely generated, and put $M_n=\bD/I_n$ so that $M$ is the direct limit of the $M_n$ and each $M_n$ is finitely presented. We have
\begin{displaymath}
\rR \Hom_{\bD}(M, \bD)=\rR \varprojlim \rR \Hom_{\bD}(M_i, \bD).
\end{displaymath}
Since $\rR^i \Hom_{\bD}(M_i, \bD)=0$ for $i>0$ and $\rR^i \varprojlim=0$ for $i>1$, we find that $\rR^i \Hom_{\bD}(M, \bD)=0$ for $i>1$. It follows from this that $\Ext^i_{\bD}(M, \bD)=0$ for $i>1$ for any $\bD$-module $M$ (see \stacks{0A5T}), and so $\bD$ has injective dimension at most one. Example~\ref{ex:ext} below shows that the injective dimension is exactly~1.
\end{proof}

\begin{remark}
The first two paragraphs of the above argument can be easily adapted to prove the following statement. Let $\{A_i\}_{i \in I}$ be a family of finite dimensional (not necessarily commutative) $\bk$-algebras such that for any finite subset $J \subset I$ the algebra $\bigotimes_{i \in J} A_i$ is self-injective, and suppose $\# I =\aleph_r$.  Then the algebra $\bigotimes_{i \in I} A_i$ has injective dimension at most $r+1$ over itself. For instance, if $\{G_i\}_{i \in I}$ is a family of finite groups and $G \subset \prod_{i \in I} G_i$ is the subgroup where all but finitely many coordinates are the identity then $\bk[G]$ has injective dimension at most $r+1$.
\end{remark}

\begin{example} \label{ex:ext}
Suppose infinitely many $\pi_n$ vanish and let $I$ be the ideal $(x^{[1]}, x^{[2]}, \ldots)$ of $\bD$. We show that $\Ext^1_{\bD}(\bk, \bD) \ne 0$ in the ungraded category, where $\bk=\bD/I$. To do this, it suffices to construct a homomorphism $\varphi \colon I \to \bD$ that does not extend to $\bD$. Let $b_1, b_2, \ldots$ be a strictly increasing sequence with $\pi_{b_i}=0$ for all $i$, and define
\begin{displaymath}
g = \sum_{i \ge 1} x^{[b_i-1]},
\end{displaymath}
regarded as a formal sum. For any $n \ge 1$ we have $x^{[n]} x^{[b_i-1]}=0$ for all $i \gg 0$; in fact, as soon as $b_i>n$ there is necessarily a carry when computing $n+(b_i-1)$ in base $b_{\bullet}$, and then the product will vanish. It follows that the product $gx$ is a well-defined element of $\bD$ for all $x \in I$, and one easily sees that putting $\varphi(x)=gx$ gives a well-defined module homomorphism $I \to \bD$. Suppose that $\varphi$ extended to $\bD$. Then we would have $gx=\varphi(x)=hx$ for all $x \in I$, where $h=\varphi(1)$ is an actual element of $\bD$. But this is clearly impossible, since $gx^{[b_{i+1}]}$ is a sum of $i$ monomials, while the number of monomials in $hx^{[b_{i+1}]}$ is bounded as $i$ varies.
\end{example}

\begin{proof}[Proof of Theorem~\ref{thm:selfinj}(a)]
Let $I$ be a non-zero finitely generated homogeneous ideal and let $\varphi \colon I[d] \to \bD$ be a map of graded modules. The argument in the first paragraph of the proof of Theorem~\ref{thm:selfinj}(b) applies and shows that $\varphi$ extends to a map $\psi \colon \bD[d] \to \bD$. Since the space $\uHom(\bD, \bD)_d$ is at most one-dimensional, it follows that $\psi$ is unique. It follows from this that if $J$ is a finitely generated homogeneous ideal containing $I$ then $\varphi$ extends uniquely to $J$; indeed, any two extensions to $J$ would further extend to $\bD$, and therefore coincide.

Now suppose $I$ is an arbitrary non-zero homogeneous ideal and let $\varphi \colon I[d] \to \bD$ be a map of graded modules. Write $I=\bigcup_{n \ge 1} I_n$ with $I_n$ finitely generated and $I_1 \ne 0$. Let $\psi \colon \bD[d] \to \bD$ be an extension of $\varphi \vert_{I_1}$. Then $\varphi \vert_{I_n}$ and $\psi \vert_{I_n}$ are two extensions of $\varphi \vert_{I_1}$ to $I_n$, and therefore must coincide. It follows that $\psi \vert_I = \varphi$, and so $\varphi$ extends to $\bD$. By a graded version of Baer's criterion, it follows that $\bD$ is injective as a graded $\bD$-module.
\end{proof}

\begin{proof}[Proof of Theorem~\ref{thm:selfinj}(c)]
We give the proof in the non-graded case, the argument in the graded case being identical. Let $n$ be maximal such that $\pi_n=0$. Then $\bD=A \otimes_{\bk} B$, where $A=\bD_{<n}$ is finite dimensional and $B=\bD^{(n)}\cong \bk[x]$. For a $B$-module $M$, let $F(M)=\Hom_B(\bD, M)$. This functor is right adjoint to the forgetful functor from $\bD$-modules to $B$-modules, and therefore takes injectives to injectives. Since $\bD$ is projective as a $B$-module, the functor $F$ is exact. We have
\begin{displaymath}
F(M) = \Hom_B(A \otimes B, M) = \Hom_{\bk}(A, M) = \Hom_{\bk}(A, \bk) \otimes_{\bk} M \cong A \otimes_{\bk} M,
\end{displaymath}
where in the first step we used an adjunction, in the second we used that $A$ is finite dimensional, and in the third we used Lemma~\ref{lem:dualfree}. In particular, we see that $F(B) \cong \bD$. Now, it is well-known that $B$ has injective dimension~1 as a module over itself. Let
\begin{displaymath}
0 \to B \to I_0 \to I_1 \to 0
\end{displaymath}
be an injective resolution. Applying $F$, we obtain a length one injective resolution of $\bD$ as a $\bD$-module, and so $\bD$ has injective dimension at most one. On the other hand,
\begin{displaymath}
0 \to B \stackrel{x}{\to} B \to B/(x) \to 0
\end{displaymath}
is a non-split extension, and is easily seen to remain non-split after applying $- \otimes_{\bk} A$, and so the injective dimension is exactly one.
\end{proof}

\begin{remark} \label{rmk:noninj-free}
Suppose infinitely many of the $\pi_n$ vanish. Then $\bD$ is not noetherian, and so the Bass--Papp theorem asserts that some infinite direct sum of injective $\bD$-modules is non-injective. We give an explicit example in the graded case. Let $n_1<n_2<\cdots$ be an infinite divisible sequence with $\pi_{n_i}=0$ for all $i$. Let $I$ be the ideal generated by the $x^{[i]}$ with $i \ge 1$. Let $M=\prod_{i \ge 1} \bD[1-n_i]$ and $M_0 \subset M$ be the direct sum. Define $\varphi \colon \bD \to M$ by using multiplication by $x^{[n_i-1]}$ on the $i$th factor. One then verifies the following that $\varphi \vert_I$ maps into $M_0$ but that $\varphi \vert_I$ cannot be extended to a map $\bD \to M_0$. Thus $M_0$ is not injective. Therefore, not all free $\bD$-modules are injective.
\end{remark}

\subsection{Injective dimension}

Let $R$ be a ring. Let $\rD(R)$ be the derived category of $R$-modules, and let $\rD^{\rb}_{\rm fp}(R)$ be the full subcategory of complexes $M$ such that $\rH^i(M)$ is finitely presented for all $i$ and non-zero for only finitely many $i$. We say that an object of $\rD(R)$ has {\bf injective amplitude} $[a, b]$ if it is isomorphic in $\rD(R)$ to a complex of injectives $I^{\bullet}$ with $I^i=0$ for $i \not\in [a,b]$. This is equivalent to $\Ext^i_R(N, M)=0$ for all (or even just cyclic) modules $N$ and $i \not\in [a,b]$ \stacks{0A5T}. We say that $M$ has {\bf finite injective dimension} if it has injective amplitude $[a,b]$ for some $a,b$. Similar definitions apply in the graded case.

Suppose $\bk$ is a noetherian ring and $E_0$ is a complex of $\bk$-modules with injective amplitude $[a,b]$. Then $E=E_0 \otimes_{\bk} \bk[x]$ has injective amplitude $[a,b+1]$ as a complex of $\bk[x]$-modules, see \stacks{0A6J}. The following theorem generalizes this statement to GDPA's:

\begin{theorem} \label{thm:injdim}
Let $\bk$ be a noetherian ring and let $\bD$ be a GDPA over $\bk$. Let $E_0$ be a complex of $\bk$-modules with finitely generated cohomology and injective amplitude $[a,b]$, and put $E=\bD \otimes_{\bk} E_0$. Then:
\begin{enumerate}
\item If $M$ is a finitely presented $\bD$-module then $\Ext^i_{\bD}(M, E)=0$ for $i \not\in [a,b+1]$.
\item $E$ has injective amplitude $[a,b+2]$ as a complex of $\bD$-modules.
\end{enumerate}
The same statements hold in the graded case.
\end{theorem}

\begin{corollary}
Suppose that $\bk$ is a Gorenstein noetherian ring of finite Krull dimension and let $\bD$ be a GDPA over $\bk$. Then $\bD$ has finite injective dimension as a (graded or ungraded) $\bD$-module.
\end{corollary}

\begin{proof}[Proof of Theorem~\ref{thm:injdim}]
Statement~(a) implies statement~(b) using an argument similar to that in the proof of Theorem~\ref{thm:selfinj}(b). Thus it is enough to prove~(a). We assume, by noetherian induction, that the statement is true for every proper quotient of $\bk$. We now prove it for $\bk$, in several steps.

{\it Step 1.} We first prove the result assuming $M$ has non-zero annihilator $\fa$. We have
\begin{displaymath}
\rR \Hom_{\bD}(M, E)=\rR \Hom_{\bD/\fa \bD}(M, \rR \Hom_{\bD}(\bD/\fa \bD, E)).
\end{displaymath}
By \stacks{0A6A}, we have
\begin{displaymath}
\rR \Hom_{\bD}(\bD/\fa \bD, E) = \bD \otimes_{\bk} \rR \Hom_{\bk}(\bk/\fa \bk, E_0).
\end{displaymath}
Let $E'_0=\rR \Hom_{\bk}(\bk/\fa \bk, E_0)$. Then $E_0'$ has injective amplitude $[a, b]$ and finitely generated cohomology. Let $E'=\bD/\fa \bD \otimes_{\bk/\fa} E_0'$. By the above, we have
\begin{displaymath}
\rR \Hom_{\bD}(M, E) = \rR \Hom_{\bD/\fa \bD}(M, E').
\end{displaymath}
By the inductive hypothesis, the right side only has cohomology in degrees $[a, b+1]$, and so the same is true of the left side.

{\it Step 2.} We now prove the result assuming $\bk$ is not a domain. Let $x, y \in \bk$ be non-zero such that $xy=0$. We have a short exact sequence
\begin{displaymath}
0 \to xM \to M \to M/xM \to 0,
\end{displaymath}
and thus a triangle
\begin{displaymath}
\rR \Hom_{\bD}(M/xM, E) \to \rR \Hom_{\bD}(M, E) \to \rR \Hom_{\bD}(xM, E) \to
\end{displaymath}
Both $xM$ and $M/xM$ are finitely presented and have non-zero annihilator, and so the outside terms only have cohomology in degrees $[a, b+1]$ by Step~1. Thus the same is true of the middle term.

{\it Step 3.} We now prove the result assuming $\bk$ is a domain. Fix $i>b+1$. Let $x \in \bk$ be non-zero. Consider the exact sequence
\begin{displaymath}
0 \to M[x] \to M \stackrel{x}{\to} M \to M/xM \to 0.
\end{displaymath}
Both $M[x]$ and $M/xM$ are finitely presented with non-zero annihilator, and so $\Ext^j_{\bD}(M[x], E)$ and $\Ext^j_{\bD}(M/xM, E)$ vanish for $j \not\in [a, b+1]$ by Step~1. It follows that the map
\begin{displaymath}
x \colon \Ext^i_{\bD}(M, E) \to \Ext^i_{\bD}(M, E)
\end{displaymath}
is surjective, and an isomorphism if $i>b+2$. However, $\Ext^i_{\bD}(M,E)$ is a finitely generated $\bD$-module, and so a surjective endomorphism of it is necessarily an isomorphism \stacks{05G8}; thus multiplication by $x$ is an isomorphism for $i=b+2$ as well. We have thus shown that every non-zero element of $\bk$ acts invertibly on $\Ext^i_{\bD}(M,E)$, that is to say, $\Ext^i_{\bD}(M,E)$ is in fact a vector space over $\bK=\Frac(\bk)$. We thus have
\begin{displaymath}
\Ext^i_{\bD}(M,E)=\bK \otimes_{\bk} \Ext^i_{\bD}(M,E)=\Ext^i_{\bD'}(M',E'),
\end{displaymath}
where the primes denote extension to $\bK$. We have $E'=\bD' \otimes_{\bK} E_0'$, and $E_0'$ is (represented by) a complex of finite dimensional $\bK$ vector spaces in degrees $[a,b]$. Thus $E'$ is a complex of finitely generated free $\bD'$-modules in degrees $[a,b]$. Since $\bD'$ has injective dimension at most~1 (Theorem~\ref{thm:selfinj}), it follows that $E'$ has injective amplitude $[a,b+1]$. Thus $\Ext^i_{\bD'}(M',E')=0$, and the result is proved.
\end{proof}

As stated, the above theorem holds in the graded case as well. However, in that case we can prove a stronger result with additional assumptions:

\begin{theorem} \label{thm:injdim2}
Let $\bk$, $\bD$, $E_0$, and $E$ be as in Theorem~\ref{thm:injdim}, though now $E_0$ is graded. Assume that every maximal ideal of $\bk$ contains infinitely many of the $\pi_n$'s. Then:
\begin{enumerate}
\item If $M$ is a finitely presented graded $\bD$-module then $\uExt^i_{\bD}(M, E)=0$ for $i \not\in [a,b]$.
\item $E$ has injective amplitude $[a,b+1]$ as a complex of graded $\bD$-modules.
\end{enumerate}
\end{theorem}

\begin{proof}
As in the proof of Theorem~\ref{thm:injdim}, it suffices to prove~(a), which we again do by noetherian induction. Steps~1 and~2 from that proof apply here, so it suffices to treat the case where $\bk$ is a domain. We consider two cases.

{\it Case~1: $\bk$ is a field.} We then have that $\bD$ is injective as a graded $\bD$-module (Theorem~\ref{thm:selfinj}(a)). We can represent $E_0$ by a complex of finite dimensional $\bk$-vector spaces concentrated in degrees $[a,b]$, and so we find that $E$ is a complex of injective $\bD$-modules concentrated in the same degrees. Thus $E$ has injective amplitude $[a,b]$, which certainly implies the required $\uExt$ vanishing.

{\it Case 2: $\bk$ is not a field.} Fix $i>b$. As in Step~3 of the proof of Theorem~\ref{thm:injdim}, we find that every non-zero element of $\bk$ acts bijectively on $\uExt^i(M,E)$. It follows that each graded piece $\uExt^i(M,E)_n$ is both a $\Frac(\bk)$ vector space and a finitely generated $\bk$-module, and thus vanishes. This proves the theorem.
\end{proof}

\begin{remark}
Theorems~\ref{thm:injdim} and~\ref{thm:injdim2} both require $E_0$ to have finitely generated cohomology. We are not sure if this is necessary. Also, in part~(b) of both theorems, the upper bound on the amplitude increases by~1; we do not know if this is necessary either.
\end{remark}

The following is a slight variant of Theorem~\ref{thm:injdim} where we relax the finiteness condition on $E_0$ at the expense of adding a different hypothesis. We state only the ungraded version, though it is also true in the graded case.

\begin{theorem} \label{thm:almost-inj}
Let $\bD$ be a GDPA over a noetherian ring $\bk$. Let $\fp$ be a prime ideal of $\bk$ containing infinitely many of the $\pi_n$'s. Let $E_0$ be the injective envelope of $\bk/\fp$ and let $E=\bD \otimes_{\bk} E_0$. Then $\Ext^n_{\bD}(M, E)=0$ for all $n \ge 1$ and all finitely presented $\bD$-modules $M$. Thus $E$ has injective dimension at most~1.
\end{theorem}

\begin{proof}
We just sketch a proof. First one reduces to the case where $\bk$ is local and $\fp=\fm$ is the unique maximal ideal. Next, as in the proof of Theorem~\ref{thm:injdim}, one can reduce to the case where $\bk$ is a domain and one already knows the result for modules with non-zero annihilators. If $\bk$ is a field, the result follows from Theorem~\ref{thm:selfinj}. Otherwise, let $x$ be a non-zero element of $\fm$. As in Step~3 of the proof of Theorem~\ref{thm:injdim}, we see that multiplication by $x$ is bijective on $\Ext^i_{\bD}(M,E)$ for $i \ge 1$.  Let $F_{\bullet} \to M$ be a resolution of $M$ by finitely generated free $\bD$-modules. (Here is the one place we use that $M$ is finitely presented.) Then $\Ext^i_{\bD}(M, E)$ is a subquotient of $\Hom_{\bD}(F_i, E)$, which is isomorphic to a finite direct sum of $E_0$'s. Since every element of $E_0$ is killed by some power of $\fm$, multiplication by $x$ cannot act bijectively on any nonzero subquotient of $\Hom_{\bD}(F_i, E)$. We thus find that $\Ext^i_{\bD}(M, E)$ vanishes for $i \ge 1$, which proves the theorem.
\end{proof}

\begin{remark} \label{rmk:almost-inj}
In fact, the module $E$ in Theorem~\ref{thm:almost-inj} is not injective in general. To see this let $\bD$ be the classical divided power algebra over $\bZ_p$ and let $I$ be the ideal generated by elements of positive degree. Then we have a $\bD$-module homomorphism $I[-1] \to \bD \otimes_{\bZ_p} \bQ_p/\bZ_p$ (note that the injective envelope of $\bZ_p/p\bZ_p$ is $\bQ_p/\bZ_p$) given by $x^{[i]} \mapsto x^{[i-1]} \otimes \tfrac{1}{i}$. Clearly this map does not extend to $\bD[-1]$ because there are no nonzero maps from $\bD[-1] \to \bD \otimes_{\bZ_p} \bQ_p/\bZ_p$.
\end{remark}

\subsection{Duality}

Let $R$ be a coherent ring. We say that a complex $\omega_R \in \rD(R)$ is a {\bf dualizing complex} if:
\begin{enumerate}
\item $\omega_R$ has finite injective dimension;
\item $\rH^i(\omega_R)$ is finitely presented for all $i$; and
\item the natural map $R \to \rR \Hom_R(\omega_R, \omega_R)$ is a quasi-isomorphism.
\end{enumerate}
Dualizing complexes are discussed in \stacks{0A7A} under the assumption that $R$ is noetherian. Some of the basic results remain true in our more general setting: for example, if $\omega_R$ is a dualizing complex then $\rR \Hom_R(-, \omega_R)$ gives a duality of the category $\rD^{\rb}_{\rm fp}(R)$ (the proof given in \stacks{0A7C} applies). We note that if $R$ is noetherian and admits a dualizing complex then it has finite Krull dimension \stacks{0A80}.

Our main result on duality for GDPA's is the following theorem. We treat only the ungraded case, but the graded case goes through in exactly the same manner.

\begin{theorem} \label{thm:duality}
Let $\bk$ be a noetherian ring with dualizing complex $\omega_{\bk}$ and let $\bD$ be a GDPA over $\bk$. Then $\omega_{\bD}=\omega_{\bf} \otimes_{\bk} \bD$ is a dualizing complex for $\bD$.
\end{theorem}

\begin{proof}
By Theorem~\ref{thm:injdim}, $\omega_{\bD}$ has finite injective dimension, and so condition (a) holds. Since $\rH^i(\omega_{\bD})=\bD \otimes_{\bk} \rH^i(\omega_{\bk})$ and $\rH^i(\omega_{\bk})$ is a finitely generated $\bk$-module, condition (b) holds. Finally, by \stacks{0A6A}, the natural map
\begin{displaymath}
\bD \otimes_{\bk} \rR \Hom_{\bk}(\omega_{\bk}, \omega_{\bk}) \to \rR \Hom_{\bD}(\omega_{\bD}, \omega_{\bD})
\end{displaymath}
is an isomorphism. Since the left side is isomorphic to $\bD$ via the natural map, we conclude that (c) holds. This completes the proof.
\end{proof}

\begin{corollary}
In the context of the theorem, the category $\rD^{\rb}_{\rm fp}(\bD)$ is self-dual (i.e., equivalent to its opposite).
\end{corollary}

\begin{corollary}
Suppose $\bk$ is a regular noetherian ring of finite Krull dimension. Then $\bD$ is a dualizing complex for $\bD$.
\end{corollary}

\begin{proof}
Under the hypotheses on $\bk$, we can take $\omega_{\bk}=\bk$.
\end{proof}

\section{Torsion and finitely presented modules} \label{s:tor}
\subsection{Torsion in finitely presented modules}

An element $m$ of a $\bD$-module $M$ is {\bf torsion} if $x^{[n]} m =0$ for $n \gg 0$. A $\bD$-module is {\bf torsion} if all of its elements are, and {\bf torsion-free} if it has no nonzero torsion element.

\begin{theorem} \label{notors}
Let $\bD=\bD(\bk, \pi_{\bullet})$ be a GDPA over the noetherian ring $\bk$. The following are equivalent:
\begin{enumerate}
\item Every finitely presented $\bD$-module (graded or not) is torsion-free.
\item Every maximal ideal of $\bk$ contains infinitely many $\pi_n$'s.
\end{enumerate}
\end{theorem}

\begin{proof}
Suppose (b) does not hold, and let $\fm$ be a maximal ideal containing only finitely many of the $\pi_n$. Then $\bD/\fm \bD$ has the form $A \otimes_{\bk} B$ where $A$ is a finite dimensional algebra over $\bk/\fm$ and $B$ is a polynomial ring in one variable over $\bk/\fm$. This algebra certainly has finitely presented modules with nonzero torsion; for example, $A$ itself (with the variable in $B$ acting by 0). Thus (a) does not hold.

Now suppose (b) holds, and let $M$ be a finitely presented $\bD$-module. We must show that $M$ is torsion-free. Suppose that $m \in M$ is torsion, and let $T$ be the $\bD$-submodule of $M$ generated by $m$. Then $T$ is finitely presented, since it is a finitely generated submodule of a finitely presented module, and torsion. Thus $T=0$ by the following lemma, which completes the proof.
\end{proof}

\begin{lemma}
Suppose condition (b) of the Theorem holds, and let $T$ be a finitely presented torsion $\bD$-module. Then $T=0$.
\end{lemma}

\begin{proof}
First suppose $\bk$ is a field. By Proposition~\ref{prop:h-free}, there is an $h \gg 0$ with $\pi_h=0$ such that $T$ is free over $\bD^{(h)}$. Since $T$ is finite dimensional over $\bk$, it follows that $T$ must have rank~0 over $\bD^{(h)}$, and so $T=0$.

We now treat the general case. If $\fm$ is a maximal ideal of $\bk$ then $T/\fm T=0$ by the previous paragraph, and so $T_{\fm}=0$ by Nakayama's lemma (again, $T$ is finitely generated as a $\bk$-module). Since $T_{\fm}=0$ for all maximal ideals $\fm$, it follows that $T=0$.
\end{proof}

\subsection{Extensions between finitely presented and torsion modules} \label{ss:ext-torsion}

The purpose of this section is to prove the following theorem.

\begin{theorem} \label{thm:ext-torsion}
Let $\bD$ be a GDPA over the noetherian ring $\bk$. Assume that $\bk$ is complete with respect to an ideal $\fI$ containing infinitely many of the $\pi_n$. Let $M$ and $T$ be graded $\bD$-modules, with $M$ finitely presented and $T$ torsion. Then $\uExt^n_{\bD}(T, M)=0$ for all $n \ge 0$.
\end{theorem}

\begin{remark} \label{rmk:ext-torsion}
Here is an example showing that Theorem~\ref{thm:ext-torsion} can fail without the completeness hypothesis. Let $\bk=\bZ_{(p)}$, let $M=\bD/(x^{[1]})$, and let $N \subset M$ be the submodule of strictly positive degree elements. Note that every homogeneous element of $N$ is killed by a power of $p$, and so $N$ is naturally a $\bZ_p$-module. We have an exact sequence
\begin{equation} \label{eq:eq5}
0 \to N \to M \to \bZ_{(p)} \to 0.
\end{equation}
Applying $\uHom_{\bD}(\bZ_{(p)}, -)_0$, we obtain an exact sequence
\begin{displaymath}
0 \to \bZ_{(p)} \to \uExt^1_{\bD}(\bZ_{(p)}, N)_0 \to \uExt^1_{\bD}(\bZ_{(p)}, M)_0.
\end{displaymath}
Thus $\uExt^1_{\bD}(\bZ_{(p)}, N)_0$ is a $\bZ_p$-module containing a copy of $\bZ_{(p)}$. It cannot be equal to $\bZ_{(p)}$, and so $\uExt^1_{\bD}(\bZ_{(p)}, M)_0$ must be non-zero. To obtain an explicit extension, take the push-out of \eqref{eq:eq5} along a map $N \stackrel{a}{\to} N \to M$, where $a \in \bZ_p \setminus \bZ_{(p)}$, and the second map is the inclusion.
\end{remark}

\begin{remark}
We also remark that Theorem~\ref{thm:ext-torsion} can fail in the non-graded case, as Example~\ref{ex:ext} shows.
\end{remark}

\begin{lemma} \label{extor1}
Suppose $\pi_n=0$ and put $D=\bD_{<n}$. Let $d<n-1$ be an integer and let $M$ be a graded $\bk$-module supported in non-negative degrees. Then $\rR \uHom_D(\bk, M \otimes_{\bk} D)_d=0$.
\end{lemma}

\begin{proof}
Since $D$ is noetherian and $\bk$ is finitely generated, the $\rR \uHom$ in question commutes with direct limits in $M$, and so we can assume $M$ is finitely generated. First suppose $\bk$ is a field. We can then assume $M=\bk[e]$ for some $e \ge 0$. Since $D$ is self-injective (by Lemma~\ref{lem:truncinj}), we have $\uExt^i_D(\bk, D[e])_d=0$ for $i>0$. We also have $\uHom_D(\bk, D[e])_d=\uHom_D(\bk[d-e], D)_0=0$, since $x^{[n-1-(d-e)]}$ acts non-trivially on $x^{[d-e]}$ (see the proof of Lemma~\ref{lem:dualfree}). Thus the result holds when $\bk$ is a field.

We now proceed by noetherian induction. We thus assume the result holds for every proper quotient of $\bk$. Suppose $M$ has non-zero annihilator $\fa$. Then
\begin{displaymath}
\begin{split}
\rR \uHom_D(\bk, M \otimes_{\bk} D)_d
&= \rR \uHom_{D/\fa D}(\bk \stackrel{\rL}{\otimes}_D D/\fa D, M \otimes_{\bk} D)_d \\
&= \rR \uHom_{D/\fa D}(\bk/\fa, M \otimes_{\bk} D)_d = 0,
\end{split}
\end{displaymath}
where the final equality comes from the inductive hypothesis. From this, we can reduce to the case where $\bk$ is a domain (as in Step~2 of the proof of Theorem~\ref{thm:injdim}) and $M=\bk$ (by d\'evissage it suffices to treat cyclic modules $M$, and the only cyclic module with zero annihilator is $\bk$).

Let $a$ be a nonzero element of $\bk$. Using the exact sequence
\begin{displaymath}
0 \to \bk \stackrel{a}{\to} \bk \to \bk/(a) \to 0
\end{displaymath}
in the $M$ variable, we find that multiplication by $a$ is bijective on $\uExt^n_D(\bk, D)_d$. But this $\uExt$ group is a finitely generated $\bk$-module. Thus if $\bk$ is not a field then it must vanish. And we have already treated the field case.
\end{proof}

\begin{lemma} \label{extor2}
Suppose $\fI=0$. Then $\rR \uHom_{\bD}(\bk, \bD)=0$.
\end{lemma}

\begin{proof}
Let $n_1<n_2<\ldots$ be a sequence with $\pi_{n_i}=0$, and put $D_i=\bD_{<n_i}$. Let $A_i=D_i^{(n_{i-1})}$, a subalgebra of $D_i$. We have an isomorphism of algebras
%\rohit{Fixed the mistype}
$D_i \cong A_i \otimes D_{i-1}$ (see Proposition~\ref{gdpa:subalg}). Let $Q^{(i)}_{\bullet} \to \bk$ be a projective resolution over $A_i$, with $Q^{(i)}_0=A_i$. Then $P^{(i)}_{\bullet}=Q^{(0)}_{\bullet} \otimes \cdots \otimes Q^{(i)}_{\bullet}$ is a projective resolution of $\bk$ over $D_i$, and we have a natural map $P^{(i)}_{\bullet} \to P^{(i+1)}_{\bullet}$ coming from the natural $\bk$-linear map $\bk \to Q^{(i+1)}_{\bullet}$. One easily sees that the direct limit of the $P^{(i)}_{\bullet}$ is a projective resolution of $\bk$ over $\bD$. We thus see that $\rR \uHom_{\bD}(\bk, \bD)$ is computed by $\uHom_{\bD}(\varinjlim P^{(i)}_{\bullet}, \bD)$. One easily sees that this is isomorphic to $\varprojlim C_i^{\bullet}$, where $C_i^{\bullet}=\uHom_{D_i}(P^{(i)}_{\bullet}, \bD)$. It follows from the construction of $P^{(i)}_{\bullet}$ that the natural map $C_i^{\bullet} \to C_{i-1}^{\bullet}$ is surjective. By general results on inverse limits of complexes (see, for instance, \cite[Ch.~III, Cor.~1.2]{lubkin}), we thus have a short exact sequence
\begin{displaymath}
0 \to \rR^1 \varprojlim \rH^{n-1}(C^{\bullet}_i) \to \rH^n(\varprojlim C^{\bullet}_i) \to \varprojlim \rH^n(C^{\bullet}_i) \to 0
\end{displaymath}
Since $C^{\bullet}_i$ computes $\rR \uHom_{D_i}(\bk, \bD)$, this yields
\begin{displaymath}
0 \to \rR^1 \varprojlim \uExt^{n-1}_{D_i}(\bk, \bD) \to \uExt_{\bD}^n(\bk, \bD) \to \varprojlim \uExt^n_{D_i}(\bk, \bD) \to 0
\end{displaymath}
In any particular degree, the outer terms vanish for $i \gg 0$ by Lemma~\ref{extor1}, and so the outer terms are~0. (Note: this reasoning does not apply in the ungraded case, and the vanishing does not hold.) This completes the proof.
\end{proof}

\begin{lemma} \label{extor5}
Let $M$ be a graded $\bD$-module such that $\rR \uHom_{\bD}(\bk, M)=0$. Then we have $\rR \uHom_{\bD}(T, M)=0$ for all graded torsion modules $T$.
\end{lemma}

\begin{proof}
We proceed in three steps.

{\it Step 1: $T$ is finitely generated and concentrated in degree~0.} Thus $T$ is just a finitely generated $\bk$-module with $\bD_+$ acting by~0. Let $F_{\bullet} \to T$ be a resolution of $T$ by finitely generated free $\bk$-modules. Regard $F_i$ as a $\bD$-module by letting $\bD_+$ act by zero. We have a convergent spectral sequence
\begin{displaymath}
\uExt^i_{\bD}(F_j,M) \implies \uExt^{i+j}_{\bD}(T,M).
\end{displaymath}
Since each $F_j$ is a finite direct sum of $\bk$'s, each $\uExt^i_{\bD}(F_j,M)$ vanishes by our hypothesis, and so the result follows.

{\it Step 2: $T$ is finitely generated.} We then have $\bD_+^n T=0$ for sufficiently large $n$. Filtering by powers of $\bD_+$ and passing to the associated graded, we can thus assume $\bD_+T=0$. But then $T$ is a finite sum of shifts of modules concentrated in a single degree, and so the result follows from Step~1.

{\it Step 3: arbitrary $T$.} Write $T=\varinjlim T_i$ with $T_i$ finitely generated. Then
\begin{displaymath}
\rR \uHom_{\bD}(T, M) = \rR \varprojlim \rR \uHom_{\bD}(T_i, M).
\end{displaymath}
Since $\rR \uHom_{\bD}(T_i, M)=0$ for all $i$, the result follows.
\end{proof}

\begin{lemma} \label{extor3}
Theorem~\ref{thm:ext-torsion} holds if $\fI=0$.
\end{lemma}

\begin{proof}
By Theorem~\ref{thm:sd}, it suffices to treat the case where $M$ is principal special, say $M=\bD'=(\bD/\fa \bD)^{(h)}$ with $\pi_h \in \fa$. Note that $\bD'$ is a quotient ring of $\bD$, and (after regrading) is itself a GDPA. For a torsion $\bD$-module $T$, we have
\begin{displaymath}
\rR \uHom_{\bD}(T, \bD')=\rR \uHom_{\bD'}(T \stackrel{\rL}{\otimes}_{\bD} \bD', \bD').
\end{displaymath}
Note that $\Tor^{\bD}_p(T, \bD')$ is a torsion $\bD'$-module for all $p$. Thus, renaming $\bD'$ to $\bD$, we have reduced to the case $M=\bD$. As $\rR \uHom_{\bD}(\bk, \bD)=0$ by Lemma~\ref{extor2}, the result follows from Lemma~\ref{extor5}.
\end{proof}

\begin{lemma} \label{extor4}
Theorem~\ref{thm:ext-torsion} holds if $T$ is annihilated by a power of $\fI$.
\end{lemma}

\begin{proof}
Filtering by powers of $\fI$, it suffices to treat the case where $T$ is annihilated by $\fI$. We then have
\begin{displaymath}
\rR \uHom_{\bD}(T, M)=\rR \uHom_{\bD/\fI \bD}(T, \rR \uHom_{\bD}(\bD/\fI \bD, M)).
\end{displaymath}
Since $\uExt^i_{\bD}(\bD/\fI \bD, M)$ is finitely presented as a $\bD/\fI \bD$-module for all $i$, the result follows from the $\fI=0$ case (Lemma~\ref{extor3}).
\end{proof}

\begin{lemma} \label{dercompl}
Let $M$ and $T$ be modules over a ring $A$ such that $M$ is complete with respect to an ideal $I \subset A$ and $\Tor_i^A(T,A/IA)=0$ for all $i>0$. Then
\begin{displaymath}
\rR \Hom_A(T,M) = \rR \varprojlim \rR \Hom_{A/I^i}(T/I^iT, M/I^iM)
\end{displaymath}
Similar statements hold in the graded case with $\uHom$.
\end{lemma}

\begin{proof}
We have
\begin{displaymath}
\begin{split}
\rR \Hom_A(T,M)
&\cong \rR \Hom_A(T, \rR \varprojlim M/I^i M) \\
&\cong \rR \varprojlim \rR \Hom_A(T, M/I^i M) \\
&\cong \rR \varprojlim \rR \Hom_{A/I^i}(T/I^i T, M/I^i M)
\end{split}
\end{displaymath}
The first isomorphism follows from the completeness of $M$; note that the transition maps in the inverse system $M/I^i M$ are surjective, and so there are no higher inverse limits. The second isomorphism follows from the (derived version of the) universal property of inverse limits. And the third is derived adjunction of tensor and $\Hom$, combined with the vanishing of the higher $\Tor$'s of $T$ with $A/I^i$.
\end{proof}

\begin{proof}[Proof of Theorem~\ref{thm:ext-torsion}]
By Lemma~\ref{extor5}, it suffices to treat the case $T=\bk$. By Lemma~\ref{dercompl}, we have
\begin{displaymath}
\rR \uHom_{\bD}(T, M) = \rR \varprojlim \rR \uHom_{\bD/\fI^i \bD}(T/\fI^i T, M/\fI^i M).
\end{displaymath}
Since $\uExt^i_{\bD/\fI^i \bD}(T/\fI^i T, M/\fI^i M)=0$ for all $i$ by Lemma~\ref{extor4}, the result follows. (Note that the hypotheses of Lemma~\ref{dercompl} are satisfied: $T$ is free over $\bk$, so all higher $\Tor$'s with it vanish, and each graded piece of $M$ is finitely generated over $\bk$, and thus $\fI$-adically complete, and so $M=\varprojlim M/\fI^i M$ in the category of graded $\bD$-modules.)
\end{proof}

\section{Nearly finitely presented modules} \label{s:nfp}

\subsection{Definitions}

Let $\bD$ be a GDPA. For a graded
%\rohit{Added the word ``graded"}
$\bD$-module $M$, we put
\begin{displaymath}
\tau^{\le n}(M) = \bigoplus_{k=0}^n M_n, \qquad \tau_{\ge n}(M) = \bigoplus_{k=n}^{\infty} M_n.
\end{displaymath}
Then $\tau_{\ge n}(M)$ is a $\bD$-submodule of $M$, while $\tau^{\le n}$ is a quotient $\bD$-module of $M$.

\begin{definition}
A $\bD$-module $M$ is {\bf nearly finitely presented} (nfp) if each $M_n$ is finitely generated as a $\bk$-module, $M_n=0$ for $n \ll 0$, and there exists a finitely presented $\bD$-module $N$, called a {\bf weak fp-envelope} of $M$, such that $\tau_{\ge n}(M) \cong \tau_{\ge n}(N)$ for some $n$. We let $\Mod_{\bD}^{\nfp}$ be the full subcategory of $\Mod_{\bD}$ on the nearly finitely presented modules.
\end{definition}

\begin{definition}
Let $M$ be a $\bD$-module. An {\bf fp-envelope} of $M$ is a map of $\bD$-modules $f \colon M \to N$ with $N$ finitely presented such that any other map from $M$ to a finitely presented $\bD$-module factors through $f$.
\end{definition}

\begin{definition}
A map of $\bD$-modules $f \colon M \to N$ is a {\bf near isomorphism} if $f_n \colon M_n \to N_n$ is an isomorphism for all $n \gg 0$.
\end{definition}

\subsection{The complete case}

Throughout this section we fix a GDPA over a noetherian ring $\bk$. We assume that $\bk$ is complete with respect to an ideal $\fI$ containing infinitely many of the $\pi_n$'s. We show, in this setting, that there is a very good theory of nfp modules.

\begin{lemma} \label{lem:exttau}
Let $\varphi \colon M \to M'$ be a near isomorphism of $\bD$-modules and let $N$ be a finitely presented $\bD$-module. Then the restriction map
\begin{displaymath}
\varphi^* \colon \uExt^i_{\bD}(M', N) \to \uExt^i_{\bD}(M, N)
\end{displaymath}
is an isomorphism for all $i$.
\end{lemma}

\begin{proof}
Consider the 4-term exact sequence
\begin{displaymath}
0 \to K \to M \stackrel{\varphi}{\to} M' \to C \to 0
\end{displaymath}
where $K$ and $C$ are the kernel and cokernel of $\varphi$. Since $\varphi$ is a near isomorphism, both $K$ and $C$ are torsion. By Theorem~\ref{thm:ext-torsion}, we have $\uExt^i_{\bD}(K, N)=\uExt^i_{\bD}(C, N)=0$ for all $i$. The result follows.
\end{proof}

\begin{proposition} \label{prop:fpenv}
The notions of ``weak fp-envelope'' and ``fp-envelope'' coincide for nfp modules. Precisely:
\begin{enumerate}
\item Suppose $M$ is an nfp module and $M'$ is a weak fp-envelope, so that we have an isomorphism $\tau_{\ge n}(M) \to \tau_{\ge n}(M')$ for some $n$. Then there exists a unique map $M \to M'$ restricting to the given isomorphism in degrees $\ge n$, and this makes $M'$ into an fp-envelope of $M$.
\item Suppose $M$ is an nfp module and $M \to M'$ is an fp-envelope. Then this map is a near isomorphism, and so $M'$ is a weak fp-envelope of $M$.
\end{enumerate}
In particular, every nfp module admits an fp-envelope, and any two weak fp-envelopes are canonically isomorphic.
\end{proposition}

\begin{proof}
(a) Let $M$ be an nfp module and let $M'$ be a weak fp-envelope, so that we are given an isomorphism $\tau_{\ge n}(M) \cong \tau_{\ge n}(M')$ for some $n$. By Lemma~\ref{lem:exttau}, the map $\tau_{\ge n}(M) \to \tau_{\ge n}(M)' \subset M'$ extends uniquely to a map $\varphi \colon M \to M'$. Note that $\varphi$ is a near isomorphism. Thus if $N$ is finitely presented then the map
\begin{displaymath}
\varphi^* \colon \uHom_{\bD}(M', N) \to \uHom_{\bD}(M, N)
\end{displaymath}
is an isomorphism by Lemma~\ref{lem:exttau}, which exactly says that $\varphi$ is an fp-envelope.

(b) Let $M \to M'$ be an fp-envelope of the nfp module $M$, and let $M''$ be a weak fp-envelope of $M$. By (a), $M''$ is canonically an fp-envelope of $M$. Since any two fp-envelopes are canonically isomorphic, we see $M'$ is also a weak fp-envelope.
\end{proof}

For an nfp module $M$, let $\Phi(M)$ denote its fp-envelope. Then $\Phi$ defines a functor $\Mod_{\bD}^{\nfp} \to \Mod_{\bD}^{\rm fp}$, and there is a natural transformation $M \to \Phi(M)$.

\begin{proposition}
The category $\Mod_{\bD}^{\nfp}$ is an abelian subcategory of $\Mod_{\bD}$.
\end{proposition}

\begin{proof}
Let $f \colon M \to N$ be a map of nfp modules. Then we have a commutative square
\begin{displaymath}
\xymatrix{
M \ar[r]^f \ar[d] & N \ar[d] \\
\Phi(M) \ar[r]^{\Phi(f)} & \Phi(N) }
\end{displaymath}
The vertical maps are near isomorphisms. It follows that the map $\ker(f) \to \ker(\Phi(f))$ is also a near isomorphism, and so $\ker(f)$ is nfp. Similarly for the cokernel and image of $f$. This shows that $\Mod_{\bD}^{\nfp}$ is an abelian subcategory of $\Mod_{\bD}$.
\end{proof}

\begin{proposition}
The functor $\Phi$ is exact.
\end{proposition}

\begin{proof}
Suppose that
\begin{displaymath}
0 \to M_1 \to M_2 \to M_3 \to 0
\end{displaymath}
is an exact sequence of nfp modules. Then the sequence
\begin{displaymath}
0 \to \Phi(M_1) \to \Phi(M_2) \to \Phi(M_3) \to 0
\end{displaymath}
is isomorphic to the original one in all sufficiently high degrees. Thus the homology of this sequence is both torsion and finitely presented, and therefore vanishes.
\end{proof}

Let $\Mod_{\bD}^{\tors}$ denote the category of finitely generated torsion modules. This is a Serre subcategory of $\Mod_{\bD}^{\nfp}$.

\begin{proposition}
The inclusion functor $\Mod_{\bD}^{\rm fp} \to \Mod_{\bD}^{\nfp}$ induces an equivalence of categories $\Mod_{\bD}^{\rm fp} \to \Mod_{\bD}^{\nfp}/\Mod_{\bD}^{\tors}$. The functor $\Phi \colon \Mod_{\bD}^{\nfp} \to \Mod_{\bD}^{\rm fp}$ induces a quasi-inverse.
\end{proposition}

\begin{proof}
Since $\Phi$ is exact and kills $\Mod_{\bD}^{\tors}$, it does factor through the Serre quotient. If $M$ is an nfp module, then the kernel and cokernel of $M \to \Phi(M)$ are torsion, and so this map is an isomorphism in the Serre quotient. Of course, if $M$ is finitely presented, then $\Phi(M)=M$. This proves the claim.
\end{proof}

\begin{proposition}
An extension of nfp modules is again nfp.
\end{proposition}

\begin{proof}
Consider an exact sequence
\begin{displaymath}
0 \to M_1 \to M_2 \to M_3 \to 0
\end{displaymath}
with $M_1$ and $M_3$ nfp. We must show that $M_2$ is nfp. Let $M_1'$ and $M_3'$ be the fp-envelopes of $M_1$ and $M_3$. Pusing out the above extension along the map $M_1 \to M_1'$, we obtain a new extension
\begin{displaymath}
0 \to M_1' \to M_2' \to M_3 \to 0.
\end{displaymath}
The map $M_2 \to M_2'$ is a near isomorphism, so it suffices to show $M_2'$ is nfp. By Lemma~\ref{lem:exttau}, the pullback map $\Ext^1(M_3', M_1') \to \Ext^1(M_3, M_1')$ is an isomorphism. Thus there is an extension
\begin{displaymath}
0 \to M_1' \to M_2'' \to M_3' \to 0
\end{displaymath}
such that $M_2'$ is the pullback of $M_2''$ along $M_3 \to M_3'$. Since $M_2''$ is an extension of finitely presented modules, it is finitely presented. The map $M_2' \to M_2''$ is a near isomorphism, and so $M_2'$ is nfp. The result follows.
\end{proof}

\subsection{Decompletion over noetherian rings}

In this section we prove some results about completions that are needed in the following section. We fix a noetherian ring $R$ and an ideal $\fa$ of $R$. We assume that $R$ is $\fa$-adically separated, or equivalently, that $\fa$ is contained in the Jacobson radical of $R$. For a $R$-module $M$, we write $\wh{M}$ for its $\fa$-adic completion.

\begin{proposition}
Let $M$ and $N$ be finitely generated $R$-modules. Then there exists an integer $k \ge 0$ such that for any $n \ge k$ the natural map $\fa^{n-k} \Hom_{R}(M, \fa^k N) \to \Hom(M, \fa^n N)$ is an isomorphism.
\end{proposition}

\begin{proof}
The map in the proposition is clearly injective, so it suffices to prove it is surjective. Let $B = \bigoplus_{n \ge 0} \fa^n$ be the blow-up algebra, a graded noetherian $R$-algebra, and let $H=\bigoplus_{n \ge 0} \Hom_{R}(M, \fa^n N)$, a graded $B$-module. Let $F \to M$ be a surjection with $F$ a finite free $R$-module, and define $H'$ similarly but with $F$ in place of $M$. Then $H$ is naturally a $B$-submodule of $H'$. However, $H'$ is just a finite direct sum of copies of $\bigoplus_{n \ge 0} \fa^n N$, and thus a finitely generated $B$-module. By noetherianity, $H$ is finitely generated as a $B$-module. Suppose $H$ is generated in degrees $\le k$. Then, since $B$ is generated in degree~1 (as a $R$-algebra), we find that the natural map $B_{n-k} \otimes_{R} H_k \to H_n$ is surjective for any $n \ge k$. This proves the proposition.
\end{proof}

\begin{proposition} \label{prop:complhom}
Let $M$ and $N$ be finitely generated $R$-modules. Then the natural map
\begin{displaymath}
\Phi \colon \Hom_{R}(M, N)^{\wedge} \to \Hom_{\wh{R}}(\wh{M}, \wh{N})
\end{displaymath}
is an isomorphism.
\end{proposition}

\begin{proof}
We first prove that $\Phi$ is injective. Suppose $\Phi(f)=0$, and let $(f_i) \in \Hom_{R}(M, N)$ be a sequence converging to $f$ in the $\fa$-adic topology. Since $\Phi(f)=0$, we have $\Phi(f_i) \to 0$. Thus, for any $n \ge 0$, we have $\Phi(f_i) \in \fa^n \Hom_{\wh{R}}(\wh{M}, \wh{N})$ for all $i \gg n$. Fix such an $i$. Thus $f_i$ maps $M$ into $N \cap \fa^n \wh{N} = \fa^n N$, and so $f_i \in \Hom_{R}(M, \fa^n N)$. By the previous proposition, we find $f_i \in \fa^{n-k} \Hom_{R}(M, N)$ for some $k$ depending only on $M$ and $N$. This shows that $f_i \to 0$ in $\Hom_{R}(M, N)$ for the $\fa$-adic topology, and so $f=0$.

We now prove that $\Phi$ is surjective. Let $f \colon \wh{M} \to \wh{N}$ be given. It suffices to find a sequence $(g_n)$ in $\Hom_{R}(M, N)$ converging to $f$. Let
\begin{displaymath}
0 \to K \to F \to M \to 0
\end{displaymath}
be a presentation for $M$ with $F$ finite free, and let $e_1, \ldots, e_r$ be a basis for $F$. Let $x_i=f(e_i) \in \wh{N}$ and choose $y^{(n)}_i \in N$ such that $x_i-y^{(n)}_i \in \fa^n \wh{N}$. Define $h_n \colon F \to N$ by $h_n(e_i)=y^{(n)}_i$. We have $h_n(K) \subset \fa^n N$.

Let $H$ be the image of the restriction map $\Hom_{R}(F, N) \to \Hom_{R}(K, N)$. Let $k \ge 0$ be such that
\begin{displaymath}
\Hom_{R}(K, \fa^n) = \fa^{n-k} \Hom_{R}(K, \fa^k N)
\end{displaymath}
for $n \ge k$, and let $\ell$ be such that
\begin{displaymath}
H \cap \fa^n \Hom_{R}(K, N) = \fa^{n-\ell}(H \cap \fa^{\ell} \Hom_{R}(K, N))
\end{displaymath}
for $n \ge k$. The number $k$ exists by the previous proposition, while the number $\ell$ exists by the Artin--Rees lemma. We have
\begin{displaymath}
h_n \vert_K \in H \cap \Hom_{R}(K, \fa^n N)
=H \cap \fa^{n-k} \Hom_{R}(K, \fa^k N) = \fa^{n-k-\ell} (H \cap \fa^{\ell} \Hom_{R}(K, \fa^k N)),
\end{displaymath}
and so $h_n \vert_K \in \fa^{n-k-\ell} H$. In other words, we can find maps $h'_i \colon F \to N$ and elements $a_i \in \fa^{n-k-\ell}$, for $i$ in some index set $I_n$, such that $h_n$ and $\sum_{i \in I_n} a_i h'_i$ have the same restriction to $K$.

Let $g_n=h_n-\sum_{i \in I_n} a_i h'_i$. Then $g_n$ restricts to~0 on $K$, and thus defines an element of $\Hom_{R}(M, N)$. We have $g_n(e_i)=h_n(e_i)$ modulo $\fa^{n-k-\ell}$, and $h_n(e_i)=x_i$ modulo $\fa^n$. Thus $f-g_n$ maps into $\fa^{n-k-\ell} \wh{N}$, that is, $f-g_n \in \Hom_{\wh{R}}(\wh{M}, \fa^{n-k-\ell} \wh{N})$. Another application of the previous proposition now gives $g_n \to f$, which completes the proof.
\end{proof}

\begin{proposition}
Let $M$ and $N$ be finitely generated $R$-modules. Then $\Isom(M, N)$ is open in $\Hom(M, N)$ for the $\fa$-adic topology.
\end{proposition}

\begin{proof}
If $M$ and $N$ are not isomorphic then $\Isom(M, N)$ is empty, and thus open. Thus assume $M$ and $N$ are isomorphic. It suffices to treat the case where $M=N$. Furthermore, since $\Aut(M)$ is a group, it suffices to show that some open neighborhood of $\id_M$ in $\End(M)$ is contained in $\Aut(M)$. In fact, $\id+\fa \End(M) \subset \Aut(M)$. To see this, let $f \colon M \to M$ be an endomorphism such that $f-\id \in \fa \End(M)$. Then $f \colon M/\fa M \to M/\fa M$ is surjective, and so $f$ is surjective by Nakayama's lemma (here it is important that $\fa$ is contained in the Jacobson radical of $R$). Surjectivity of $f$ implies injectivity \stacks{05G8}, and so $f \in \Aut(M)$.
\end{proof}

\begin{proposition}
Let $f \colon M \to N$ be a map of finitely generated $R$-modules. If $\wh{f} \colon \wh{M} \to \wh{N}$ is an isomorphism, then $f$ is an isomorphism.
\end{proposition}

\begin{proof}
Consider the sequence
\begin{displaymath}
0 \to K \to M \stackrel{f}{\to} N \to C \to 0.
\end{displaymath}
Since completion is exact, we find $\wh{K}=\wh{C}=0$, and so $K=C=0$. Thus $f$ is an isomorphism.
\end{proof}

\begin{proposition} \label{prop:decompl}
Let $M$ and $N$ be finitely generated $R$-modules. If $\wh{M} \cong \wh{N}$ then $M \cong N$.
\end{proposition}

\begin{proof}
Identify $Y=\Hom_{\bD}(M, N)$ with a subset of $X=\Hom_{\wh{\bD}}(\wh{M}, \wh{N})$. We have shown that $X$ is the $\fa$-adic completion of $Y$, and so $Y$ is dense in $X$. The set of isomorphisms in $X$ is open, and non-empty by hypothesis. Since $Y$ is dense in $X$, it follows that $Y$ meets the locus of isomorphisms. Thus there is a map $M \to N$ that induces an isomorphism on completions, and is therefore itself an isomorphism.
\end{proof}

\subsection{Uniqueness of weak fp-envelopes}

Fix a GDPA $\bD$ over a noetherian ring $\bk$. We assume that the Jacobson radical $\fI$ of $\bk$ contains infinitely many of the $\pi_n$'s. Our main result is:

\begin{theorem} \label{thm:weakfp}
Let $M$ and $N$ be finitely presented $\bD$-modules such that $\tau_{\ge n}(M)$ is isomorphic to $\tau_{\ge n}(N)$ for some $n$. Then $M$ is isomorphic to $N$.
\end{theorem}

\begin{corollary} \label{cor:weakfp}
Any two weak fp-envelopes of an nfp module are isomorphic.
\end{corollary}

Beware that the isomorphism provided by the theorem is not canonical! We require a few lemmas before proving the theorem. We write $\wh{M}$ for the $\fI$-adic completion of $M$. Note that if $M$ is a graded $\bk$-module then $\wh{M}$ is the direct sum of the completions of the graded pieces of $M$.

\begin{lemma} \label{lem:trunchom}
Let $M$ and $N$ be $\bD$-modules with $M$ finitely presented. Let $n$ be larger than the degrees of generators and relations of $M$. Then the natural map $\uHom_{\bD}(M, N)_0 \to \uHom_{\bD}(\tau^{\le n}(M), \tau^{\le n}(N))_0$ is an isomorphism.
\end{lemma}

\begin{proof}
Informally, to give a map $M \to N$ one says where the generators go and checks that relations to go to~0, and all of this can be done by looking up to degree $n$. We now give a formal argument. We have an exact sequence
\begin{displaymath}
0 \to \tau_{>n}(N) \to N \to \tau^{\le n}(N) \to 0.
\end{displaymath}
Applying $\uHom_{\bD}(M, -)_0$, we find
\begin{displaymath}
0 \to \uHom_{\bD}(M, N)_0 \to \uHom_{\bD}(M, \tau^{\le n}(N))_0 \to \uExt^1_{\bD}(M, \tau_{>n}(M))_0.
\end{displaymath}
Note that $\uHom_{\bD}(M, \tau_{>n}(N))_0=0$, since the generators of $M$ must map to 0 for degree reasons. Now, consider an exact sequence
\begin{displaymath}
0 \to K \to F \to M \to 0
\end{displaymath}
with $F$ free. The map $\uHom_{\bD}(K, \tau_{>n}(N))_0 \to \uExt^1(M, \tau_{>n}(N))_0$ is surjective. But once again, this $\uHom$ is~0 for degree reasons. We have thus shown that the natural map $\uHom_{\bD}(M, N)_0 \to \uHom_{\bD}(M, \tau^{\le n}(N))_0$ is an isomorphism. It is clear that the natural map
\begin{displaymath}
\uHom_{\bD}(\tau^{\le n}(M), \tau^{\le n}(N))_0 \to \uHom_{\bD}(M, \tau^{\le n}(N))_0
\end{displaymath}
is also an isomorphism, and so the result follows.
\end{proof}

\begin{lemma} \label{lem:truncisom}
Let $M$ and $N$ be finitely presented $\bD$-modules. Let $n$ be larger than the degrees of generators and relations of both $M$ and $N$. Then the natural map $\uIsom_{\bD}(M, N)_0 \to \uIsom_{\bD}(\tau^{\le n}(M), \tau^{\le n}(N))_0$ is an isomorphism.
\end{lemma}

\begin{proof}
Injectivity follows immediately from the previous proposition. Let $\ol{f} \colon \tau^{\le n}(M) \to \tau^{\le n}(N)$ be an isomorphism, and let $\ol{g}$ be its inverse. By the previous proposition, $\ol{f}=\tau^{\le n}(f)$ and $\ol{g}=\tau^{\le n}(g)$. Since $\tau^{\le n}(gf)=\tau^{\le n}(\id_M)$, the previous proposition gives $gf=\id_M$. Similarly $fg=\id_N$, and the result follows.
\end{proof}

\begin{proof}[Proof of Theorem~\ref{thm:weakfp}]
Let $M$ and $N$ be finitely presented $\bD$-modules with $\tau_{\ge n}(M) \cong \tau_{\ge n}(N)$. Suppose that the generators and relations of $M$ and $N$ are in degrees $\le d$. Let $M_0=\tau^{\le d}(M)$ and $N_0=\tau^{\le d}(N)$, both of which are finitely generated modules over the noetherian ring $R=\tau^{\le d}(\bD)$. Since $\tau_{\ge n}(M) \cong \tau_{\ge n}(N)$, we find $\tau_{\ge n}(\wh{M}) \cong \tau_{\ge n}(\wh{N})$. Proposition~\ref{prop:fpenv} yields an isomorphism $\wh{M} \cong \wh{N}$ of $\wh{\bD}$-modules. Applying $\tau^{\le d}$, we obtain an isomorphism $\wh{M}_0 \cong \wh{N}_0$ of $\wh{R}$-modules. Proposition~\ref{prop:decompl} now gives an isomorphism $M_0 \cong N_0$ of $R$-modules. (Here we use $\fa=\fI \cdot R$.) Finally, Lemma~\ref{lem:truncisom} gives $M \cong N$, completing the proof.
\end{proof}

\subsection{Some counterexamples}

We now give some examples showing that the hypotheses in some of our results are necessary.

\subsubsection{An nfp module without an fp-envelope}

Let $\bk=\bZ_{(p)}$, let $\bD$ be the classical divided power algebra over $\bk$, and let $M=\bD/(x^{[1]})$.

\begin{proposition}
\label{prop:counter-example}
The nfp module $\tau_{\ge 1}(M)$ does not have an fp-envelope.
\end{proposition}

\begin{proof}
Since $M_n$ is torsion for $n \ge 1$, we see that $\tau_{\ge 1}(M)$ is naturally a $\bZ_p$-module. Furthermore, since there is no bound on the exponent of the group $M_n$, it follows that $\bZ_p$ acts faithfully on $\tau_{\ge 1}(M)$.

Now, suppose that $N$ is an fp-envelope of $\tau_{\ge 1}(M)$. Then
\begin{displaymath}
\uHom_{\bD}(\tau_{\ge 1}(M), \tau_{\ge 1}(M))_0 = \uHom_{\bD}(\tau_{\ge 1}(M), M)_0 = \uHom_{\bD}(N, M)_0.
\end{displaymath}
The first equality is elementary, while the second comes from the definition of fp-envelope. We have shown that the left group above contains $\bZ_p$, while the right group is a finitely generated $\bZ_{(p)}$-module. This is a contadiction, and so the result follows.
\end{proof}

This shows that the completeness assumption in place in Proposition~\ref{prop:fpenv} cannot be removed.

\subsubsection{An nfp modules with multiple weak fp-envelopes}

Let $\bk$ be the ring of integers in a number field $\bK$, let $\bD$ be the classical divided power algebra
%\rohit{Fixed the mistype}
over $\bk$, and let $M=\bD/(x^{[1]})$.

\begin{proposition}
\label{prop:class-group}
If the class group of $\bK$ is non-trivial then the nfp module $\tau_{\ge 1}(M)$ admits non-isomorphic weak fp-envelopes.
\end{proposition}

\begin{proof}
Each graded piece of $\tau_{\ge 1}(M)$ is torsion as a $\bk$-module, and is thus a module over $\wh{\bk} = \bk \otimes_{\bZ} \wh{\bZ}$. Let $\fa$ be an ideal of $\bk$. Then $\fa \otimes_{\bk} \wh{\bk}$ is isomorphic to $\wh{\bk}$. Thus $\tau_{\ge 1}(M \otimes_{\bk} \fa)$ is isomorphic to $\tau_{\ge 1}(M)$. It follows that both $M$ and $M \otimes_{\bk} \fa$ are weak fp-envelopes of $M$. If $\fa$ is not principal then these modules are not isomorphic, since their degree~0 pieces are not isomorphic.
\end{proof}

We thus see that in Corollary~\ref{cor:weakfp}, the hypothesis that infinitely many of the $\pi_n$'s belong to the Jacobson radical is necessary. (Note that in the above example, the Jacobson radical of $\bk$ is 0, and $\pi_n$ is nonzero for all $n \ge 1$.)

\begin{remark}
In fact, there is a simpler, though less interesting, counterexample: taking $\bD=\bk[x]$, the nfp module $\bD$ itself admits non-isomorphism fp-envelopes, such as $\bD$ and $\bD \oplus \bD/(x)$.
\end{remark}

\section{Open problems} \label{s:prob}

\setcounter{enumi}{0}
\newcommand{\prob}{\removelastskip\vskip.5\baselineskip\noindent\refstepcounter{enumi}{\bf \arabic{enumi}.}\ }

\prob
Does there exist a graded-coherent GDPA $\bD$ such that $\bD^{[h]}$ is not graded-coherent for some $h$?

\prob
Let $\bD$ be the classical divided power algebra over $\bk$. Suppose that $\bD$ is graded-coherent and $\bD \otimes \bQ \cong (\bk \otimes \bQ)[x]$ is Gr\"obner-coherent. Is $\bD$ Gr\"obner-coherent?

\prob
In Proposition~\ref{prop:tor1} we computed $\Tor^{\bD}_1(\bk, \bk)$. Can one compute $\Tor^{\bD}_{\bullet}(\bk, \bk)$ as a co-algebra? If $M$ is a finitely presented $\bD$-module, does $\Tor^{\bD}_{\bullet}(M, \bk)$ admit a nice structure as a co-module?

\prob
If $f \colon M \to N$ is a surjection of special modules, is $\ker(f)$ special?

\prob
Can special dimension be detected by the vanishing of some kind of derived functor?

\prob
Conjecture~\ref{conj:k} on the structure of $\uK(\bD)$ when $(\ast)$ does not hold, e.g., for the $q$-divided power algebra.
%\rohit{Fixed the mistype}.
To prove this, it would be useful to have a refined version of the $\rL$ invariant that picks off the $\cK_n$ piece of a class in $\uK(\bD)$, much as $\rL$ picks off the $\cK_+$ piece. At the very least, there should be a ``residue map'' $\uK(\bD) \to \cK_n \otimes \bQ[\zeta_n]$ for each $n \ge 1$, where $\zeta_n$ is primitive $n$th root of unity, corresponding to taking the residue of an element of $\ul{\sK}$ at $t=\zeta_n$. One can construct this map for the $q$-divided power algebra over $\bZ[q]$ using Proposition~\ref{prop:qdprat}, but for a general GDPA we do not know how to construct it.

\prob
Does the equivalence in Remark~\ref{rmk:dercat} hold? Can one reconstruct $\cD$ from $\cF$, $\cD^t$, and $\cD^t \cap \cF$ in some direct manner?

\prob
Are weak fp-envelopes unique when $\bk=\bZ$?

\prob
Suppose $\bk$ is noetherian and its Jacobson radical contains infinitely many of the $\pi_n$'s. Is $\Mod_{\bD}^{\nfp}$ abelian?

\end{document}